\documentclass[reqno]{amsart}

\usepackage{amssymb,amsmath}
\usepackage{euscript}
\usepackage{graphicx}
\usepackage{tikz}
\usepackage{tikz-cd}
\usepackage{hyperref}
\usepackage{mathrsfs}
\usepackage{dsfont}
\usepackage{mathtools}

\newtheorem{thm}{Theorem}[section]
\newtheorem{cor}[thm]{Corollary}
\newtheorem{lem}[thm]{Lemma}
\newtheorem{prop}[thm]{Proposition}
\newtheorem{defin}[thm]{Definition}
\newtheorem{def-lem}[thm]{Definition-Lemma}
\newtheorem{conj}[thm]{Conjecture}

\theoremstyle{remark}

\newtheorem{warning}[thm]{Warning}
\newtheorem{rem}[thm]{Remark}

\newtheorem{example}[thm]{Example}

\newtheorem{ques}{Question}
\newtheorem{assu}{Assumption}

\numberwithin{equation}{section}

\newcommand{\bbC}{\mathbb{C}}

\newcommand{\bbF}{\mathbb{F}}
\newcommand{\bbM}{\mathbb{M}}

\newcommand{\bbR}{\mathbb{R}}

\newcommand{\bbS}{\mathbb{S}}
\newcommand{\bbX}{\mathbb{X}}

\newcommand{\bbZ}{\mathbb{Z}}

\newcommand{\scrE}{\mathscr{E}}
\newcommand{\scrP}{\mathscr{P}}

\newcommand{\calC}{\mathcal{C}}
\newcommand{\calD}{\mathcal{D}}
\newcommand{\calF}{\mathcal{F}}
\newcommand{\calM}{\mathcal{M}}
\newcommand{\calO}{\mathcal{O}}
\newcommand{\calP}{\mathcal{P}}
\newcommand{\calT}{\mathcal{T}}

\newcommand{\frakF}{\mathfrak{F}}
\newcommand{\frakM}{\mathfrak{M}}

\newcommand{\frakR}{\mathfrak{R}}
\newcommand{\frakW}{\mathfrak{W}}

\newcommand{\frakY}{\mathfrak{Y}}
\newcommand{\frakZ}{\mathfrak{Z}}

\DeclareMathOperator{\codim}{codim}
\DeclareMathOperator{\coker}{coker}
\DeclareMathOperator{\crit}{Crit}
\DeclareMathOperator{\grad}{grad}

\DeclareMathOperator{\ind}{ind}
\DeclareMathOperator{\ls}{LS}
\DeclareMathOperator{\sing}{Sing}

\newcommand{\flow}{\mathrm{Flow}}
\newcommand{\fr}{\mathrm{fr}}
\newcommand{\fred}{\mathrm{Fred}}
\newcommand{\ho}{\mathrm{Ho}}
\newcommand{\identity}{\mathrm{Id}}
\newcommand{\loc}{\mathrm{loc}}

\newcommand{\spectra}{\mathrm{Sp}}
\newcommand{\sq}{\mathrm{Sq}}
\newcommand{\sset}{\mathrm{sSet}}
\newcommand{\std}{\mathrm{std}}
\renewcommand{\th}{\mathrm{Th}}
\newcommand{\unit}{\mathds{1}}

\newcommand{\abs}[1]{\lvert#1\rvert}
\newcommand{\brak}[1]{\langle#1\rangle}

\title[Floer homotopy theory and degenerate Lagrangian intersections]{Floer homotopy theory \\ and degenerate Lagrangian intersections}
\author{Kenneth Blakey}
\address{Department of Mathematics, MIT, 182 Memorial Drive, Cambridge, MA 02139, U.S.A.} \email{kblakey@mit.edu}

\begin{document}

\begin{abstract} 
We give a lower bound on the number of intersection points of a Lagrangian pair via Steenrod squares on Lagrangian Floer cohomology induced from a Floer homotopy type. The main technical input is a computation of the associated graded of the action-filtration of the Floer homotopy type in terms of Morse homotopy theory (precisely, Conley index theory). We also prove a lower bound using the quantum cap product on Lagrangian Floer cohomology.
\end{abstract}

\maketitle
\tableofcontents

\section{Introduction}
\subsection{Overview}
Let $(M,\omega)$ be a symplectic $2n$-manifold and $(L_0,L_1)$ a pair of Lagrangians in $M$. A fundamental question in symplectic geometry is:

\begin{ques}\label{ques:overview}
How many points does $(L_0,L_1)$ have to intersect in?
\end{ques}

This question has been well-studied and we will give a brief and non-exhaustive review shortly, but, before our review, we will quickly describe the present article's contribution. By using Steenrod squares on the Lagrangian Floer cohomology of $(L_0,L_1)$ via a Floer homotopy type (which exists under Assumption \ref{assu:main}), we establish a lower bound on the number of (possibly degenerate) intersection points of $(L_0,L_1)$. Moreover, we apply our main result to infinitely many examples to improve the lower bound on the number of intersection points; all previously known results, and a new ``quantum cap-length'' estimate, cannot detect these improved lower bounds in the aforementioned examples. A sketch of the construction of our examples is as follows, cf. Subsection \ref{subsec:application}. Consider an embedding $\iota:\bbC P^2\hookrightarrow S^7$.\footnote{The existence of such an embedding is given by \cite[Theorem A]{Fuq94}.} We may plumb together two disjoint copies of the cotangent bundle $T^*S^7$ along $n$-many disjoint copies of $\bbC P^2$; this process will result in a symplectic 14-manifold together with a Lagrangian pair $(L_0,L_1)$, each diffeomorphic to $S^7$, which intersect cleanly in $n$-many disjoint copies of $\bbC P^2$. With respect to a natural framed brane structure on $(L_0,L_1)$, the Floer homotopy type of $(L_0,L_1)$ will be 
    \begin{equation}
    \Sigma^\infty_+\big(\bbC P^2\vee\Sigma^8\bbC P^2\vee\cdots\vee\Sigma^{8(n-1)}\bbC P^2\big).
    \end{equation}
Via elementary methods, any compactly supported Hamiltonian deformation of $(L_0,L_1)$ which intersects in a discrete set must do so in at least $n$ (possibly degenerate) points of pairwise distinct action. Since $\sq^2$ is non-vanishing on $n$ distinct Lagrangian Floer cohomology classes of $(L_0,L_1)$ (which satisfy a degree condition), it follows from Theorem \ref{thm:main2} that any compactly supported Hamiltonian deformation of $(L_0,L_1)$ which intersects in a discrete set must do so in at least $2n$ (possibly degenerate) points of pairwise distinct action.

We will now give a brief and non-exhaustive review of previous works.

\begin{assu}\label{assu:overview}
During this review, we will consider the following collection of assumptions.
\begin{enumerate}
\item $(M,\omega)$ is either closed or a Liouville manifold.
\item $L_i$, $i\in\{1,2\}$, is connected, closed, and \emph{relatively exact}.\footnote{Recall, relatively exact means $\omega\cdot\pi_2(M,L)=0$.}
\item $(L_0,L_1)$ is Hamiltonian isotopic via a compactly supported Hamiltonian isotopy.
\end{enumerate}
\end{assu}

\begin{conj}[Arnol'd Conjecture]
Suppose Assumption \ref{assu:overview}, then we have the lower bound 
    \begin{equation}
    \abs{L_0\cap L_1}\geq\crit(L_0), 
    \end{equation}
where 
    \begin{equation}
    \crit(L_0)\equiv\min\big\{\abs{\crit(f)}:f\in C^\infty(L_0)\big\}.
    \end{equation}
\end{conj}

\begin{rem}
An application of the Weinstein neighborhood theorem shows that, if the Arnol'd conjecture is true, then the lower bound must be sharp.
\end{rem}

In the case of the 0-section of a cotangent bundle of a closed smooth manifold,
the (degenerate) Arnol’d conjecture is essentially known, cf. \cite{LS85}.

Under transversality hypotheses, Floer had proven the following.

\begin{thm}[Theorem 1 \cite{Flo88b}]
Suppose Assumption \ref{assu:overview} and that $L_0$ intersects $L_1$ transversely, then we have the lower bound
    \begin{equation}
    \abs{L_0\cap L_1}\geq\sum\operatorname{rank}H_j(L_0;\bbZ/2). 
    \end{equation}
\end{thm}

Both Hofer and Floer independently established a lower bound using Lusternik-Schnirelmann (LS) theory; this is a tool used to study the minimal number of critical points any smooth function on a smooth manifold can possibly have which, in direct contrast to Morse theory, does not use non-degeneracy hypotheses. The \emph{LS category} of a space $Y$ is 
    \begin{equation}
    \ls(Y)\equiv\inf_k\big\{Y=U_0\cup\cdots\cup U_k:U_j\subset Y\;\textrm{open and contractible}\big\}.\footnote{The definition of $\ls(Y)$ differs in the literature by adding $\pm1$. In particular, the definition of $\ls(Y)$ in \cite{Pea94}, a reference we use extensively in Section \ref{sec:conleyindextheory}, is $+1$ higher than the definition in the present article.}
    \end{equation} 
The LS category of a space is notoriously difficult to compute, however, it has a simple lower bound. Let $R$ be a commutative ring. We define the \emph{$R$-cup-length} of $Y$ as
    \begin{equation}
    c_R(Y)\equiv\inf_k\big\{\forall\alpha_1,\ldots,\alpha_k\in\widetilde{H}^*(Y;R),\alpha_1\smile\cdots\smile\alpha_k=0\big\},\footnote{The definition of $c_R(Y)$ differs in the literature by adding $\pm1$. In particular, the definition in the present article agrees with the definition in \cite{Hof88,HP22}.}
    \end{equation}
and it is an elementary exercise to show 
    \begin{equation}
    c_R(Y)-1\leq\ls(Y).
    \end{equation}
We refer to \cite{CLOT03} for more details concerning LS theory.

\begin{thm}[Theorem 3 \cite{Hof88}, Theorem 1 \cite{Flo89a}]\label{thm:cuplength}
Suppose Assumption \ref{assu:overview}, then we have the lower bound 
    \begin{equation}
    \abs{L_0\cap L_1}\geq c_{\bbZ/2}(L_0).
    \end{equation}
\end{thm}

Recent work of Hirschi-Porcelli has generalized the previous theorem using Floer homotopy theory. Let $\frakR$ be a ring spectrum. Again, we may define the \emph{$\frakR$-cup-length} of $Y$ as
    \begin{equation}
    c_\frakR(Y)\equiv\inf_k\big\{\forall\alpha_1,\ldots,\alpha_k\in \widetilde{\frakR}^*(Y),\alpha_1\smile\cdots\smile\alpha_k=0\big\}, 
    \end{equation}
where $\widetilde{\frakR}^*$ is the (reduced) multiplicative cohomology theory determined by $\frakR$. Let $E\to Y$ be a real vector bundle, where we now assume $Y$ is homotopy equivalent to a finite CW-complex. Recall, the Thom space of $E$ is the pointed space $\th(E)$ defined as the quotient space of the disk bundle of $E$ by the sphere bundle of $E$. Also, recall the Thom spectrum of $E$ is simply the suspension spectrum $Y^E\equiv\Sigma^\infty\th(E)$. The notion of Thom space/spectrum extends to real virtual bundles over $Y$ and, in both cases, is independent of all choices in the construction up to homotopy. We say $E$ is \emph{$\frakR$-orientable} if there is a map 
    \begin{equation}
    o:Y^E\to\Sigma^{\operatorname{rank} E}\frakR,
    \end{equation}
called a Thom class of $E$, such that the restriction to a fiber,
    \begin{equation}
    o\vert_y:\Sigma^{\operatorname{rank} E}\bbS\to\Sigma^{\operatorname{rank} E}\frakR,
    \end{equation}
represents plus or minus the identity of $\frakR$ after (de)suspension. Here, $\bbS$ is the sphere spectrum and, when $E$ is a real virtual bundle, $\operatorname{rank}E$ should be interpereted as the virtual rank.

\begin{thm}[Theorem 1.9 \cite{HP22}]
Suppose Assumption \ref{assu:overview} and that the index bundles over the various Floer moduli spaces of pseudo-holomorphic curves appearing are $\frakR$-orientable, then we have the lower bound 
    \begin{equation}
    \abs{L_0\cap L_1}\geq c_\frakR(L_0). 
    \end{equation}
\end{thm}

\begin{rem}
Hirschi-Porcelli's theorem recovers Hofer-Floer's theorem by using the Eilenberg-Maclane spectrum $H\bbF_2$ since any real virtual bundle is $H\bbF_2$-orientable. \cite[Proposition 1.19]{Por22} gives conditions for when we may use other ring spectra.
\end{rem}

\subsection{Quantum cap-length}
We will digress a bit to present a generalization of Theorem \ref{thm:cuplength} to the case that $(L_0,L_1)$ is not Hamiltonian isotopic via a compactly supported Hamiltonian isotopy; this uses the quantum cap product on Lagrangian Floer homology.

\begin{assu}\label{assu:caplength}
During this digression, we will consider either of the two following collections of assumptions. The first collection is the following.
\begin{enumerate}
\item $(M,\omega)$ is closed.
\item $L_i$ is closed and relatively exact.
\item The integral of $\omega$ over any annulus with one boundary component on $L_0$ and one boundary component on $L_1$ is zero.
\end{enumerate}
The second collection is the following.
\begin{enumerate}
\item $(M,\omega)$ is a Liouville manifold.
\item $L_i$ is exact and either closed or cylindrical at infinity.
\end{enumerate}
\end{assu}

Using Assumption \ref{assu:caplength}, we may define the $\bbZ/2$-Lagrangian Floer homology of $(L_0,L_1)$, denoted $HF(L_0,L_1;\bbZ/2)$. The quantum cap product is a way of giving $HF(L_0,L_1;\bbZ/2)$ the structure of a module over the cohomology ring of each $L_i$. This is done by (1) deforming $(L_0,L_1)$ via a compactly supported Hamiltonian isotopy to intersect transversely, (2) choosing a regular admissible $\omega$-compatible almost complex structure $J$ for $(L_0,L_1)$, and (3) defining a map
    \begin{equation}
    *:HF(L_0,L_1;\bbZ/2)\otimes_{\bbZ/2} H^*(L_i;\bbZ/2)\to HF(L_0,L_1;\bbZ/2)
    \end{equation}
induced by counting $J$-holomorphic strips connecting intersection points of $(L_0,L_1)$ with Lagrangian boundary conditions and a marked point constraint on $L_i$.

\begin{thm}\label{thm:caplength}
Suppose Assumption \ref{assu:caplength} and that there exists 
    \begin{equation}
    \alpha_1,\ldots,\alpha_k\in H^*(L_i;\bbZ/2),\;\;\deg\alpha_j\geq1
    \end{equation}
and $\beta\in HF(L_0,L_1;\bbZ/2)$ such that
    \begin{equation}
    \beta*\alpha_1*\cdots*\alpha_k\neq0, 
    \end{equation}
then any compactly supported Hamiltonian deformation of $(L_0,L_1)$ which intersects in a discrete set must do so in at least $k+1$ points of pairwise distinct action.
\end{thm}

\begin{rem}
If $(L_0,L_1)$ actually satisfies Assumption \ref{assu:overview}, then we may use the Lagrangian PSS isomorphism to identify $HF(L_0,L_1;\bbZ/2)$ and $H_*(L_i;\bbZ/2)$. The upshot is that the quantum cap-length estimate will reduce to Hofer-Floer's cup-length estimate.
\end{rem}

\subsection{Brief overview of Floer homotopy}\label{subsec:introfloer}
There are some limitations of the previously stated results. For example, with the exception of the quantum cap-length result, we require the Lagrangian pair to be Hamiltonian isotopic via a compactly supported Hamiltonian isotopy. Another shortcoming is that we would like to use more of the operations on (co)homology to say things about Lagrangian intersections, e.g., Steenrod squares. However, Lagrangian Floer cohomology \emph{a priori} only admits Steenrod squares when there exists a Floer homotopy type underlying it; this leads us to consider Floer homotopy theory. 

Cohen-Jones-Segal introduced Floer homotopy theory in \cite{CJS95} as an idea that associates a stable homotopy type to the geometric data arising from Floer theory. The authors illustrated their ideas by constructing the Morse homotopy type, i.e., the stable homotopy type associated to Morse theory; moreover, they proved the Morse homotopy type recovers the stable homotopy type of the underlying closed smooth manifold. Since then, Floer homotopy theory in the Cohen-Jones-Segal framework has seen considerable development, cf. \cite{Bon24,Bon25,CK23,Lar21,LS14,Rez24}. Meanwhile, the foundations of Floer homotopy theory are currently being rewritten in ongoing work of Abouzaid-Blumberg \cite{AB24}, and applications of this framework are already beginning to appear, cf. \cite{BB25,PS24,PS25a,PS25b,PS25c}; this is the framework we will use in the present article.\footnote{Note, the article \cite{BB25}, whose main result (using parameterized Lagrangian Floer homotopy instead of Steenrod squares) will recover the improved lower bounded presented in Subsection \ref{subsec:application}, was written after the present article first appeared as a preprint.} 

\begin{assu}\label{assu:main}
In the present article, we will consider the following collection of assumptions.
\begin{enumerate}
\item $(M,\omega)$ is a Liouville manifold.
\item $L_i$ is exact and either closed or cylindrical at infinity.
\item $(L_0,L_1)$ has a framed brane structure $\Lambda$.
\end{enumerate}
\end{assu}

Recall, a \emph{framed brane structure} for $(L_0,L_1)$ is: 
\begin{enumerate}
\item a choice of map $\Lambda:M\to BO$ that factors a classifying map for the stable tangent bundle of $M$ as a complex vector bundle:
    \begin{equation}
    \begin{tikzcd}[column sep=10ex]
    M\arrow[rr,bend right]\arrow[r,"\Lambda"] & BO\arrow[r,"(\cdot)\otimes_\bbR\underline{\bbC}"] & BU
    \end{tikzcd}
    ,
    \end{equation}
where $\underline{\bbC}$ denotes the trivial complex vector bundle $M\times\bbC\to M$;
\item and a choice of homotopy between $\Lambda\vert_{L_i}:L_i\to BO$ and a classifying map for the stable tangent bundle of $L_i$ as a real vector bundle.
\end{enumerate}
More explicitly, a framed brane structure for $(L_0,L_1)$ is a choice of real vector bundle $\Lambda\to M$ with a choice of isomorphism $TM\oplus\underline{\bbC}^d\cong\Lambda\otimes_\bbR\underline{\bbC}$, for some $d\in\bbZ_{\geq0}$, and a choice of homotopy between $\Lambda\vert_{L_i}$ and $TL_i\oplus\underline{\bbR}^d$ through totally real subbundles of $TM\oplus\underline{\bbC}^d$. (Of course, such choices are not guaranteed to exist, cf. \cite{AGLW25}.)

\begin{rem}\label{rem:grading}
A choice of framed brane structure for $(L_0,L_1)$ induces graded structures on $L_0$ and $L_1$, cf. \cite[Example 2.10]{Sei00}. 
\end{rem}

The following result is essentially due to Cohen-Jones-Segal; the details were carried out by Large.

\begin{thm}[\cite{CJS95},\cite{Lar21}]
Suppose Assumption \ref{assu:main}, then the Floer homotopy type of $(L_0,L_1)$ exists.
\end{thm}

Briefly, the Floer homotopy type of $(L_0,L_1)$ is constructed as follows, cf. Section \ref{sec:floerhomotopytheory}.\footnote{Technically, this is not how Large originally constructed the Floer homotopy type -- they had used the Cohen-Jones-Segal framework. However, we may reinterpret their construction in the Abouzaid-Blumberg framework.} First, we construct an unstructured flow category $\bbF^{H,J}_{L_0,L_1}$ associated to a choice of regular admissible Floer data $(H,J)$. Second, we use $\Lambda$ to lift $\bbF^{H,J}_{L_0,L_1}$ to a framed flow category $\bbF^{H,J,\Lambda}_{L_0,L_1}$. Finally, the Floer homotopy type is defined as the mapping spectrum
    \begin{equation}
    \frakF^{H,J,\Lambda}_{L_0,L_1}\equiv\flow^\fr\Big(\unit,\bbF^{H,J,\Lambda}_{L_0,L_1}\Big),
    \end{equation}
where $\flow^\fr$ is the stable $\infty$-category of framed flow categories and $\unit$ is the unit flow category with the trivial rank 0 framing. Note, our convention is such that 
    \begin{align}
    H_*\Big(\frakF^{H,J,\Lambda}_{L_0,L_1};R\Big)&=HF_*(L_0,L_1;R), \\
    H^*\Big(\frakF^{H,J,\Lambda}_{L_0,L_1};R\Big)&=HF^*(L_0,L_1;R),
    \end{align}
where $R$ is any (discrete) commutative ring.

\begin{rem}
It is known $\frakF^{H,J,\Lambda}_{L_0,L_1}$ will be invariant under compactly supported Hamiltonian deformations of $(L_0,L_1)$ and the choice of admissible regular Floer data $(H,J)$. However, $\frakF^{H,J,\Lambda}_{L_0,L_1}$ will depend on the choice of $\Lambda$ (in particular, cf. \cite{BB25}). Hence, saying ``the'' Floer homotopy type is not precise, however, it is understood to be defined up to some homotopical choice (e.g., $\Lambda$).
\end{rem}

\begin{rem}
We could define the Floer homotopy type using the non-Hamiltonian perturbed framework for Lagrangian Floer theory, and the two definitions are equivalent. When we use the non-Hamiltonian perturbed framework, we will drop the Hamiltonian from the notation.
\end{rem}

\subsection{Main results}
The purpose of the present article is to investigate Question \ref{ques:overview}. Our main result is the following.

\begin{thm}\label{thm:main1}
Suppose Assumption \ref{assu:main} and that $L_0$ intersects $L_1$ in finitely many (possibly degenerate) points, denoted $\{x_{s_j}\}_{1\leq j\leq k,1\leq s_j\leq\ell_j}$, such that the points of fixed index $j$ (i.e., $\{x_{s_j}\}_{1\leq s_j\leq\ell_j}$) are of a fixed action $\kappa_j\in\bbR$ which is strictly increasing (i.e., $\kappa_j<\kappa_{j+1}$ $\forall j$). For each $j$, we have that there exists integers $\{d_{s_j}\}_{1\leq s_j\leq\ell_j}\subset\bbZ$ and a cofiber sequence 
    \begin{equation}\label{eq:cofiber}
    \frakF^{J,\Lambda,\kappa_{j-1}}_{L_0,L_1}\to\frakF^{J,\Lambda,\kappa_j}_{L_0,L_1}\to\bigvee_{s_j}\Sigma^{d_{s_j}}\Sigma^\infty\calC_{f_{s_j}},
    \end{equation}
where each $f_{s_j}$ is a smooth function on $\bbR^n$ with a single (possibly degenerate) critical point at the origin.
\end{thm}

Here, we use the fact that $\frakF^{J,\Lambda}_{L_0,L_1}$ is a filtered spectrum: 
    \begin{equation}
    \frakF^{J,\Lambda,\kappa_0}_{L_0,L_1}\to\frakF^{J,\Lambda,\kappa_1}_{L_0,L_1}\to\cdots\to\frakF^{J,\Lambda,\kappa_k}_{L_0,L_1}\equiv\frakF^{J,\Lambda}_{L_0,L_1},
    \end{equation}
where $\frakF^{J,\Lambda,\kappa_j}_{L_0,L_1}$ is the Floer homotopy type of $(L_0,L_1)$ computed with respect to intersection points of action at most $\kappa_j$. (Our convention for action is that Floer trajectories decrease the action. Also, by convention, $\frakF^{J,\Lambda,\kappa_0}_{L_0,L_1}\equiv*$.) Moreover, $\calC_{f_{s_j}}$ is the maximal Conley index of $f_{s_j}$. Essentially, $\calC_{f_{s_j}}$ is the (pointed) space whose stable homotopy type is computed by the Morse homotopy type of $f_{s_j}$; in particular, we prove an extension of the Cohen-Jones-Segal theorem, which identifies the Morse homotopy type of a closed smooth manifold with the suspension spectrum of that manifold, to Conley index theory, cf. Proposition \ref{prop:morsehomotopytype}.

The following result is immediate from Theorem \ref{thm:main1}.

\begin{cor}\label{cor:main1}
Let $R$ be a (discrete) commutative ring. For each $j$, we have a long exact sequence in cohomology: 
    \begin{equation}\label{eq:les}
    \cdots\to\bigoplus_{s_j}\widetilde{H}^{*-d_{s_j}}(\calC_{f_{s_j}};R)\to HF^*_{\kappa_j}(L_0,L_1;R)\to HF^*_{\kappa_{j-1}}(L_0,L_1;R)\to\cdots,
    \end{equation}
where 
    \begin{equation}
    HF^*_{\kappa_j}(L_0,L_1;R)\equiv H^*\Big(\frakF^{J,\Lambda,\kappa_j}_{L_0,L_1};R\Big).
    \end{equation}
We have the analogous long exact sequence in homology. Moreover, when $R=\bbZ/2$, \eqref{eq:les} is compatible with Steenrod squares. 
\end{cor}

The main motivation behind proving Theorem \ref{thm:main1} is the following result.

\begin{thm}\label{thm:main2}
Recall, $2n=\dim M$. Suppose Assumption \ref{assu:main} and that there exists integers 
    \begin{equation}
    \{s_j\}_{1\leq s\leq\ell, 1\leq j\leq k(s)}\subset\bbZ,\;\;s_j\geq(n-1)/2
    \end{equation}
together with 
    \begin{equation}
    \alpha_1,\ldots,\alpha_\ell\in HF^*(L_0,L_1;\bbZ/2)
    \end{equation}
satisfying
    \begin{align}
    \forall s,\;\;\sq^{s_{k(s)}}\cdots\sq^{s_1}(\alpha_s)&\neq0, \\
    \forall s'<s,\;\;\deg\alpha_s-\deg\Big(\sq^{s'_{k(s')}}\cdots\sq^{s'_1}(\alpha_{s'})\Big)&\geq n-1,
    \end{align}
then any compactly supported Hamiltonian deformation of $(L_0,L_1)$ which intersects in a discrete set must do so in at least $\ell+\sum_s k(s)$ points of pairwise distinct action.
\end{thm}

\begin{rem}
In special cases, it could happen that we may use Steenrod squares of lower degree in the hypotheses for Theorem \ref{thm:main2}. For example, when $\dim M=14$, the above result requires Steenrod squares of degree at least 3. But, in this case, we may allow Steenrod squares of degree 2 as well, cf. Corollary \ref{cor:dim14}.
\end{rem}

The proof of Theorem \ref{thm:main1} can be briefly summarized as follows. Around every $x_{s_j}$, $L_1$ may be written locally as a graph $\big\{(p,d(f_{s_j})_p)\big\}\subset T^*\bbR^n$. In particular, we may deform $(L_0,L_1)$ into a transversely intersecting pair of Lagrangians by perturbing each $f_{s_j}$ to be Morse. Thus, if we look at the Floer homotopy type at a fixed action level $\kappa_j$ (i.e., only using the intersection points arising from the critical points of the Morse perturbations of $\{f_{s_j}\}_{1\leq s_j\leq\ell_j}$), the only Floer trajectories will be the Morse trajectories of these Morse perturbations; this is essentially the content of \eqref{eq:cofiber}. Theorem \ref{thm:main2} follows by (1) investigating Steenrod squares on the maximal Conley index of a smooth function on $\bbR^n$ with a single (possibly degenerate) critical point (cf. Subsection \ref{subsec:steenrodsquares}) and (2) using \eqref{eq:les}.

\subsection*{Acknowledgments}
The author would like to thank his advisor Paul Seidel for suggesting and mentoring this project. The author would also like to thank Mohammed Abouzaid, Andrew Blumberg, and Hiro Lee Tanaka for helpful comments. Finally, the author would like to thank the anonymous referee for providing feedback which helped improve the present article in various ways. This work was partially supported by an NSF Graduate Research Fellowship award and an MIT Mathworks Science Fellowship award. This work was partly completed while a visitor in the Nonlinear Algebra group at MPI MIS in Leipzig, and the author would like to thank Bernd Sturmfels and the Institute for their hospitality.

\section{Morse theory: homology}\label{sec:morsehomology}
\subsection{Classical approach}
In order to set notation, we will quickly review the construction of Morse homology. In this section, let $X$ be a closed smooth $n$-manifold, $f\in C^\infty(X)$ a smooth function, and $g$ a Riemannian metric on $X$.

Consider the subset $\crit(f)\subset X$ of critical points of $f$; we say $x\in\crit(f)$ is non-degenerate if the Hessian of $f$ at $x$ is non-degenerate. We will assume the (generic) condition that $f$ is Morse, i.e., all of its critical points are non-degenerate. We consider the (negative) gradient flow of $f$:
    \begin{equation}
    \partial_s\gamma(s)=-\grad f\big(\gamma(s)\big).
    \end{equation}
    
\begin{warning}
In particular, whenever we write ``gradient flow'' in the present article, we really mean the ``negative gradient flow''.
\end{warning}

It is well-known that any gradient flow is defined for all time and limits to critical points of $f$. Let $W^u(x)$ resp. $W^s(x)$ be the unstable resp. stable manifold of $x$:
    \begin{equation}
    W^u(x)\equiv\Big\{p\in X:\lim_{s\to-\infty}\gamma_p(s)=x\Big\}\;\;\textrm{resp.}\;\;W^s(x)\equiv\Big\{p\in X:\lim_{s\to+\infty}\gamma_p(s)=x\Big\},
    \end{equation}
where $\gamma_p$ is the unique gradient flow of $f$ satisfying $\gamma_p(0)=p$. Recall, $W^u(x)$ resp. $W^s(x)$ is a locally closed submanifold of $X$ that is diffeomorphic to the unit disk of dimension $I(x)$ resp. $n-I(x)$, where $I(\cdot)$ denotes the Morse index.

We will assume the (generic) condition that $(f,g)$ is Morse-Smale, i.e., all unstable manifolds intersect all stable manifolds transversely. Therefore, the space of gradient flows connecting $x$ to $y$,
    \begin{equation}
    \widehat{\calM}_X^{f,g}(x,y)\equiv W^u(x)\cap W^s(y),
    \end{equation}
is a smooth $\big(I(x)-I(y)\big)$-manifold. There is a free proper $\bbR$-action on $\widehat{\calM}_X^{f,g}(x,y)$, when $x\neq y$, given by time-shift in the $s$-coordinate. The corresponding $\bbR$-quotient $\calM_X^{f,g}(x,y)$ has a compactification $\bbM_X^{f,g}(x,y)$ given by allowing breakings at critical points. The Morse chain complex (with $\bbZ/2$-coefficients), denoted $CM_*(X;f,g;\bbZ/2)$, is generated in degree $m$ by Morse index $m$ critical points and has differential
    \begin{equation}
    \partial x\equiv\sum_{I(y)=I(x)-1}\big\lvert\bbM_X^{f,g}(x,y)\big\rvert_{\bbZ/2}\cdot y, 
    \end{equation}
where $\big\lvert\bbM_X^{f,g}(x,y)\big\rvert_{\bbZ/2}$ is the cardinality of $\bbM_X^{f,g}(x,y)$ modulo 2.\footnote{Of course, we may define Morse homology with $\bbZ$-coefficients by using a system of coherent orientations, cf. \cite{Sch93}.} The Morse homology, denoted 
    \begin{equation}
    HM_*(X;f,g;\bbZ/2)\equiv H_*\big(CM_*(X;f,g;\bbZ/2),\partial\big),
    \end{equation}
is isomorphic to $H_*(X;\bbZ/2)$.

\subsection{Functional-analytic approach}\label{subsec:morsehomologyfunctional}
It will be useful in the present article to recast Morse theory in a more functional-analytic framework analogous to the framework of Floer theory, cf. \cite{Sch93}.

Let $x,y\in\crit(f)$ and consider the following Banach manifold:
    \begin{equation}
    \scrP^{1,2}_\frakM(x,y)\equiv\Big\{\gamma\in W^{1,2}_\loc(\bbR;X):\lim_{s\to-\infty}\gamma(s)=x,\lim_{s\to+\infty}\gamma(s)=y\Big\}.
    \end{equation}
There is a Banach bundle $\scrE_\frakM\to\scrP^{1,2}_\frakM(x,y)$, with fiber $\scrE_\frakM\vert_\gamma\equiv L^2(\bbR;\gamma^*TX)$, together with a section
    \begin{align}
    \Xi_{f,g}:\scrP^{1,2}_\frakM(x,y)&\to\scrE_\frakM \\
    \gamma(s)&\mapsto\partial_s\gamma(s)+\grad f\big(\gamma(s)\big). \nonumber 
    \end{align}
Elliptic regularity shows $(\Xi_{f,g})^{-1}(0)$ consists of gradient flows connecting $x$ to $y$. Let $\nabla$ be a connection on $TX$ (which we usually take to be the Levi-Civita connection $\nabla^g$), then the linearization of $\Xi_{f,g}$ is given by an operator 
    \begin{equation}
    D(\Xi_{f,g})_\gamma:T_\gamma\scrP^{1,2}_\frakM(x,y)\to\scrE_\frakM\vert_\gamma,
    \end{equation}
where 
    \begin{equation}
    T_\gamma\scrP^{1,2}_\frakM(x,y)=W^{1,2}(\bbR;\gamma^*TX).
    \end{equation}
Note, the linearization at a gradient flow is independent of $\nabla$. When $(f,g)$ is Morse-Smale, we see $D(\Xi_{f,g})_\gamma$, for any two $x,y\in\crit(f)$ and any $\gamma\in(\Xi_{f,g})^{-1}(0)$, is a surjective Fredholm operator of index $I(x)-I(y)$. Hence, $(\Xi_{f,g})^{-1}(0)$ is transversely cut out of dimension $I(x)-I(y)$. As in the classical case, there is a free proper $\bbR$-action on $(\Xi_{f,g})^{-1}(0)$, when $x\neq y$, given by time-shift in the $s$-coordinate; moreover, the corresponding $\bbR$-quotient has a Morse-theoretic compactification given by allowing breakings at critical points. Therefore, we may define Morse homology using this functional-analytic viewpoint. 

The connection between this functional-analytic framework and the classical framework is the diffeomorphism 
    \begin{align}
    (\Xi_{f,g})^{-1}(0)&\xrightarrow{\sim}\widehat{\calM}^{f,g}_X(x,y) \\
    \gamma&\mapsto\gamma(0); \nonumber
    \end{align}
this descends to the $\bbR$-quotients and extends over the compactifications. In particular, the functional-analytic framework is equivalent to the classical framework, and we will elide the distinction between the two.

\section{Conley index theory}\label{sec:conleyindextheory}
\subsection{Basics}
In the present article, we are only concerned with gradient flows on smooth manifolds. However, let us introduce Conley index theory a bit more generally, cf. \cite{Con78,Pea94,Sal85}. In this section let $X$ be a (not necessarily closed) smooth $n$-manifold. 

A \emph{complete (smooth) flow} on $X$ is a global (smooth) $\bbR$-action on $X$:
    \begin{align}
    \phi:\bbR\times X&\to X \\
    (s,x)&\mapsto\phi_s(x).\nonumber
    \end{align}
For any $x\in X$, we denote by $\phi^x:\bbR\to X$ the map $s\mapsto\phi_s(x)$ satisfying $\phi^x(0)=x$. A subset $S\subset X$ is \emph{invariant} for $\phi$ if
    \begin{equation}
    \forall x\in S,\;\;\phi^x(s)\in S\;\;\forall s\in\bbR.
    \end{equation}
If $S\subset X$ is compact, then $S$ contains a \emph{maximal invariant subset}:
    \begin{equation}
    I(S)\equiv\{x\in S:\textrm{$\phi^x(s)\in S$ $\forall s\in\bbR$}\}. 
    \end{equation}
A compact invariant subset $S$ is \emph{isolated} if there exists a compact neighborhood $N\subset X$ of $S$ such that $I(N)=S$; we call $N$ an \emph{isolating neighborhood}. Throughout this section, $S$ will always denote an isolated invariant subset.

\begin{defin}
An \emph{index pair} of $S$ is a pair $(N,N_-)$, $N_-\subset N$, of compact subsets of $X$ such that:
\begin{enumerate}
\item $\overline{N-N_-}$ is an isolating neighborhood of $S$;
\item $N_-$ is \emph{positively invariant}, i.e., if $x\in N_-$ and there exists a positive time $s_0\in\bbR$ such that $\phi^x(s_0)\in N$, then $\phi^x\big([0,s_0]\big)\subset N_-$;
\item and $N_-$ is an \emph{exit set}, i.e., if $x\in N$ and there exists a positive time $s_0\in\bbR$ such that $\phi^x(s_0)\notin N$, then there exists a time $s_0'\in[0,s_0)$ such that $\phi^x(s_0')\in N_-$. 
\end{enumerate}
\end{defin}

\begin{thm}[\cite{Con78}]
An index pair $(N,N_-)$ of an isolated invariant subset $S$ always exists. Moreover, the pointed homotopy type $\calC(S)$ of $N/N_-$ is independent of the choice of index pair. We refer to $\calC(S)$ as the \emph{Conley index} of $S$.
\end{thm}

When perturbing a flow, it is possible that the topology of an isolated invariant subset changes. However, under a small perturbation, we may insure that the Conley index of an isolated invariant subset remains unchanged.

\begin{thm}[\cite{Con78}]\label{thm:continuation}
Let $\Phi:\bbR\times X\times[0,1]\to X$ be a smooth family of complete flows on $X$, i.e., for all $\lambda\in[0,1]$, $\Phi(\cdot,\cdot,\lambda)$ is a complete flow on $X$. Suppose $S^\lambda$ is an isolated invariant subset of $\Phi(\cdot,\cdot,\lambda)$ and there exists a single isolating neighborhood $N\subset X$ of each $S^\lambda$, then 
    \begin{equation}
    \calC\big(S^\lambda\big)\simeq\calC\big(S^{\lambda'}\big),\;\;\forall\lambda,\lambda'.
    \end{equation}
\end{thm}

Since the Conley index is independent of the choice of index pair, we wish to know that we may choose nice index pairs for applications. An \emph{index triple} of $S$ is a triple $(N,N_-,N_+)$ such that: $(N,N_-)$ is an index pair of $S$ in the forward flow, $(N,N_+)$ is an index pair of $S$ in the reverse flow, $N$ is a smooth $n$-manifold with non-empty boundary, and $\partial N$ decomposes as $N_-\cup N_+$, where $N_-$ and $N_+$ are smooth $(n-1)$-manifolds with common (possibly empty) boundary. A proof of the following theorem can be extracted from the explicit construction of the Conley index in the proof of \cite[Theorem 5.4]{RS88}.

\begin{thm}[\cite{RS88}]\label{thm:rs}
An index triple of $S$ exists.
\end{thm}

The proof of the following lemma is straightforward.

\begin{lem}\label{lem:finitecwcomplex}
$\calC(S)$ is a finite CW-complex of dimension at most $n$.
\end{lem}

\begin{proof}
Choose an index triple $(N,N_-,N_+)$ of $S$. The manifold/submanifold pair $(N,N_-)$ is smooth, hence we may triangulate $N$ such that $N_-$ is a subcomplex. In particular, $\calC(S)\simeq N/N_-$ is naturally a CW-complex. Moreover, since $N$ and $N_-$ are finite, so is $\calC(S)$. Finally, the dimension statement follows since $\calC(S)$ is a quotient of an $n$-dimensional CW-complex by an $(n-1)$-dimensional CW-subcomplex. 
\end{proof}

\subsection{Gradient flow}
Let $f\in C^\infty(X)$ be a smooth function and $g$ a Riemannian metric on $X$. As in the closed case, we consider the (negative) gradient flow of $f$:
    \begin{equation}
    \partial_s\gamma(s)=-\grad f\big(\gamma(s)\big).
    \end{equation}
However, in contrast to the closed case, since $X$ is not assumed to be compact, the gradient flow is not necessarily a complete flow. Also, in contrast to the closed case, it is not necessarily true that the gradient flows of $f$ connect critical points -- they may escape in (in)finite time. We will say a gradient flow $\gamma$ is a \emph{connecting trajectory} (or \emph{gradient flow} connecting $x$ to $y$) if $\gamma$ is defined for all time and there exists two $x,y\in\crit(f)$ such that 
    \begin{equation}
    \lim_{s\to-\infty}\gamma(s)=x\;\;\textrm{and}\;\;\lim_{s\to+\infty}\gamma(s)=y.
    \end{equation} 
In the present article, we only care about the dynamics of the connecting trajectories; hence, we may as well assume the gradient flow is defined for all time. 

We denote by $S_f\subset X$ the subset of all critical points of $f$ together with all points on all connecting trajectories; this subset is invariant by construction and we refer to it as the \emph{maximal invariant subset} of $f$. We will always assume $S_f$ is isolated (and, in particular, compact).\footnote{The following results may be stated more generally for any invariant subset of the gradient flow of $f$, which we assume is isolated, that consists only of (1) a subset of critical points and (2) all points on all connecting trajectories of that subset.} We will refer to $\calC_f\equiv\calC(S_f)$ as the \emph{maximal Conley index} of $f$.

\begin{thm}[Theorem 1.29 \cite{Pea94}]
Suppose $N$ is an isolating neighborhood of $S_f$, then there exists a smooth function $f'\in C^\infty(X)$, $C^2$-close to $f$, such that 
\begin{enumerate}
\item $f=f'$ on a neighborhood of $X-\operatorname{int}N$, 
\item and $f'$ has only non-degenerate critical points which are all in $N$.
\end{enumerate}
\end{thm}

When perturbing $f$ into $f'$ using the previous theorem, we observe that 
    \begin{equation}
    \calC_f\simeq\calC_{f'}
    \end{equation}
by Theorem \ref{thm:continuation}. We now wish to perturb $g$ into $g'$ such that $(f',g')$ is Morse-Smale while preserving $\calC_{f'}$ (again, by appealing to Theorem \ref{thm:continuation}). But this is straightforward since we may choose an isolating neighborhood $N$ of $S_{f'}$ and perturb $g$ into $g'$ on a strictly smaller isolating neighborhood of $S_{f'}$ contained in $N$. In particular, after a $C^2$-small perturbation, we may assume $(f,g)$ is Morse-Smale.

Following Section \ref{sec:morsehomology}, we may define the Morse chain complex $CM_*(X;f,g;\bbZ/2)$ with a differential that counts connecting trajectories in dimension 0 Morse moduli spaces. The upshot is the following result, due to Floer.

\begin{thm}[Theorem 1 \cite{Flo89c}]\label{thm:floer1}
We have the following isomorphism:
    \begin{equation}
    HM_*(X;f,g;\bbZ/2)\cong\widetilde{H}_*(\calC_f;\bbZ/2).
    \end{equation}
\end{thm}

\begin{rem}
Of course, the previous theorem also holds with $\bbZ$-coefficients after choosing a system of coherent orientations. 
\end{rem}

\subsection{S-duality}
Given a Morse-Smale pair $(f,g)$, we may consider the Morse-Smale pair $(-f,g)$; it follows that the gradient flow of $-f$ is the reverse of the gradient flow of $f$. In particular, we have the equality $S_f=S_{-f}$, and we wish to relate $\calC_f$ and $\calC_{-f}$.\footnote{Again, the following results may be stated more generally for any invariant subset of the gradient flow of $f$, which we assume is isolated, that consists only of (1) a subset of critical points and (2) all points on all connecting trajectories of that subset. We would then be relating $\calC(S)$ and $\calC(-S)$, where $\calC(-S)$ is the Conley index of $S$ with respect to the gradient flow of $-f$.}

\begin{thm}[Theorem 2.1 \cite{McC92}]\label{thm:mccord}
Suppose $X$ is orientable, then
    \begin{equation}
    \widetilde{H}^m(\calC_f)\cong\widetilde{H}_{n-m}(\calC_{-f}).
    \end{equation}
\end{thm}

For our purposes, we will require a stronger version of the previous theorem, originally proved by Pears. In particular, we wish to show the isomorphism in the previous theorem is induced from a relation of spectra. First, we quickly recall the notion of \emph{Spanier-Whitehead duality} (or \emph{S-duality}) in an arbitrary symmetric monoidal category, cf. \cite{BG99}. 

Let $(\calC,\otimes,\unit)$ be a symmetric monoidal category. An object $A\in\calC$ is called \emph{weakly dualizable} if there exists an object $\calD A\in\calC$ together with a natural bijection
    \begin{equation}\label{eq:adjunction}
    \hom(-\otimes A,\unit)\cong\hom(-,\calD A).
    \end{equation}
We call $A$ \emph{reflexive} if both $\calD A$ and $\calD\calD A$ exist, and the morphism $A\to\calD\calD A$ corresponding to the composition
    \begin{equation}
    A\otimes\calD A\xrightarrow{\sim}\calD A\otimes A\xrightarrow{\epsilon_A}\unit,
    \end{equation}
where $\epsilon_A$ is the counit morphism, is an isomorphism. Suppose $A,B\in\calC$ are objects such that $A$, $B$, and $A\otimes B$ are weakly dualizable, then we may define the morphism $\mu_{A,B}:\calD A\otimes\calD B\to\calD(A\otimes B)$ which corresponds to the morphism
    \begin{equation}
    \calD A\otimes\calD B\otimes A\otimes B\xrightarrow{\sim}\calD A\otimes\calD B\otimes B\otimes A\xrightarrow{1\otimes\epsilon_B\otimes1}\calD A\otimes A\xrightarrow{\epsilon_A}\unit.
    \end{equation}
We call $A$ \emph{strongly dualizable} if it is reflexive and $\mu_{A,\calD A}$ is an isomorphism. Note, weak and strong duals are unique up to isomorphism. Moreover, $A$ admits a strong dual $\calD A$ if and only if, for all objects $B,C\in\calC$, the map 
    \begin{equation}
    \hom(B,C\otimes\calD A)\xrightarrow{(-)\otimes A}\hom(B\otimes A,C\otimes\calD A\otimes A)\xrightarrow{1_C\otimes\epsilon_A\circ(-)}\hom(B\otimes A,C)
    \end{equation}
is an isomorphism.

We denote by $\spectra$ the stable $\infty$-category of spectra and by $\ho\spectra$ its homotopy category, where the latter is the usual stable homotopy category. Recall, $\ho\spectra$ is a symmetric monoidal category via the smash product of spectra.

We return to gradient flow. By Theorem \ref{thm:rs}, we may choose an index triple $(N,N_-,N_+)$ for $S_f$ such that
    \begin{equation}
    \calC_f\simeq N/N_-\;\;\textrm{and}\;\;\calC_{-f}\simeq N/N_+.
    \end{equation}
Let $\widetilde{m}\in\bbZ_{\geq0}$ be sufficiently large so that we have an embedding $\iota:X\hookrightarrow\bbR^{\widetilde{m}}$; we denote by $\zeta$ the normal bundle of $\iota(X)$.

\begin{thm}[Theorem 3.8 \cite{Pea94}]
Let $\iota_N:N\hookrightarrow\bbR^{\widetilde{m}}$ resp. $\iota_{N_+}:N_+\hookrightarrow\bbR^{\widetilde{m}}$ be the restriction of $\iota$ to $N$ resp. $N_+$. The spectra 
    \begin{equation}
    \Sigma^\infty\calC_f\;\;\textrm{and}\;\;\Sigma^{-\widetilde{m}}\Sigma^\infty \big(\th(\iota_N^*\zeta)/\th(\iota_{N_+}^*\zeta)\big)
    \end{equation}
are strongly dual objects in $\ho\spectra$.
\end{thm}

\begin{cor}[Theorem 3.11 \cite{Pea94}]\label{thm:sduality}
We have the following isomorphism:
    \begin{equation}
    \widetilde{H}^m(\calC_f)\cong H_{n-m}(N,N_+;o(\zeta)). 
    \end{equation}
Here, $H_*(N,N_+;o(\zeta))$ is homology of the pair $(N,N_+)$ with local coefficients in the orientation local system $o(\zeta)$ of $\zeta$.     
\end{cor}

\begin{rem}
If $X$ is orientable, then $o(\zeta)$ is trivial and the previous corollary recovers Theorem \ref{thm:mccord}.
\end{rem}

\subsection{Co-H-spaces}
In this subsection, we will take a quick digression to recall the basic notions surrounding co-H-spaces. A \emph{co-H-space} is a path-connected pointed space $Y$ together with a co-multiplication $Y\to Y\vee Y$ such that the compositions $Y\to Y\vee Y\to Y$, where the second map is the projection onto the first resp. second factor, are homotopic to the identity. 

A quick way to cook up co-H-spaces is as follows. First, consider the circle $S^1$ with the co-H-space structure given by the co-multiplication
    \begin{equation}
    t\mapsto\begin{cases}
    (2t,*) & 0\leq t\leq 1/2\\
    (*, 2t-1) & 1/2\leq t\leq1
    \end{cases}
    ,
    \end{equation}
where $t\in S^1$ after using the identification 
    \begin{equation}
    S^1\cong[0,1]/\big(\{0\}\cup\{1\}\big).
    \end{equation}
Now, suppose $Y$ is a co-H-space and $Z$ is any pointed space, then $Y\wedge Z$ is a co-H-space with co-multiplication given by smashing the co-multiplication on $Y$ together with the identity map on $Z$: 
    \begin{equation}
    Y\wedge Z\to (Y\vee Y)\wedge Z\simeq(Y\wedge Z)\vee(Y\wedge Z).
    \end{equation}
In particular, the (reduced) suspension $\Sigma Z\equiv S^1\wedge Z$ of $Z$ is a co-H-space. 

Another way to prove the suspension of any (normal\footnote{A \emph{normal} space is a space such that every pair of disjoint closed sets can be separated by disjoint open neighborhoods. In particular, every CW-complex is normal.}) pointed space is a co-H-space is as follows. Recall that we have defined the LS category of a space. Since the suspension of any pointed space is the union of two cones, which are contractible, we see that the LS category of a suspension is 1. We combine this observation together with the following theorem.

\begin{thm}[Theorem 1.55 \cite{CLOT03}]\label{thm:clot}
Suppose $Y$ is a normal path-connected pointed space, then $Y$ is a co-H-space if and only if $\ls(Y)\leq1$. 
\end{thm}

\begin{warning}
A word of caution, we would really like to emphasize the path-connected hypothesis in Theorem \ref{thm:clot} -- in particular, $S^0$ is not a co-H-space.
\end{warning}

It is natural to wonder whether every co-H-space is given by the suspension of a pointed space. The answer is no, see \cite{BH60} for examples of co-H-spaces that are not suspensions. However, every co-H-space is at least related to a suspension.

\begin{thm}[Theorem 1.1 \cite{Gan70}]\label{thm:ganea}
Suppose $Y$ is a co-H-space, then there exists a retraction $\Sigma\Omega Y\to Y$. 
\end{thm}

\subsection{Isolated critical points}
We return to gradient flow. Recall, our setup is as follows. Let $X$ be a (not necessarily closed) smooth $n$-manifold, $f\in C^\infty(X)$ a smooth function, and $g$ a Riemannian metric on $X$. We will not necessarily assume $(f,g)$ is Morse-Smale, however, we will assume there exists $x\in\crit(f)$ such that $x$ is possibly degenerate and isolated. In particular, $\{x\}\subset X$ is an isolated invariant subset of the gradient flow of $f$.\footnote{Recall, we actually consider the negative gradient flow and drop the ``negative'' for brevity.} We may naturally construct an index pair $(N,N_-)$ of $\{x\}$, where $N$ is given by a small closed $n$-ball containing $x$ and $N_-$ is the subset of $\partial N$ where $-\grad f$ points outwards. We call the index pair just constructed a \emph{standard index pair} of $\{x\}$.

\begin{lem}\label{lem:pathconnected}
Suppose $x$ is not a minimum, then $\calC\big(\{x\}\big)$ is path-connected.
\end{lem}

\begin{proof}
When $x$ is not a minimum, we see that a standard index pair $(N,N_-)$ of $\{x\}$ satisfies $N_-\neq\emptyset$. In particular, since $N$ is path-connected and $N_-$ is non-empty, it follows $\calC\big(\{x\}\big)\simeq N/N_-$ is the continuous image of a path-connected space and thus path-connected.
\end{proof}

\begin{rem}
We exclude the case that $x$ is a minimum because, if $x$ is a minimum, we see that $N_-=\emptyset$ for a standard index pair $(N,N_-)$ of $\{x\}$; this is because we consider the negative gradient flow. In particular, $\calC\big(\{x\}\big)\simeq S^0$ since quotienting by $\emptyset$ adds a disjoint basepoint.
\end{rem}

\begin{lem}[Corollary 4.7 \cite{Pea94}]
We have that
    \begin{equation}
    \ls\Big(\calC\big(\{x\}\big)\Big)\leq1.\footnote{Recall, there is a convention difference for LS category in the present article and in \cite{Pea94}.}
    \end{equation}
\end{lem}

\begin{cor}\label{cor:cohspace}
Suppose $x$ is not a minimum, then $\calC\big(\{x\}\big)$ is a co-H-space.\footnote{Again, we exclude the case of a minimum because, if $x$ is a minimum, $\calC\big(\{x\}\big)\simeq S^0$ is not a co-H-space.} 
\end{cor}

\begin{proof}
Use the previous lemma with Theorem \ref{thm:clot} and Lemmas \ref{lem:finitecwcomplex} and \ref{lem:pathconnected}.
\end{proof}

As an aside, we should note that Pears proved the following general existence result. Recall, if $Z$ is any finite CW-complex, then $Z$ is homotopic to a smooth manifold $Z'$ such that there is an embedding $Z'\hookrightarrow S^{\widetilde{m}}$ (for $\widetilde{m}$ sufficiently large). 

\begin{thm}[Theorem 4.9 \cite{Pea94}]
There exists a smooth function on $\bbR^{\widetilde{m}+1}$ with an isolated critical point at the origin such that $\calC\big(\{0\}\big)\simeq\Sigma Z$.  
\end{thm}

\begin{rem}
As mentioned before, it is known there exists co-H-spaces that are not suspensions. It is an open problem whether such a space arises as $\calC\big(\{0\}\big)$ for a smooth function on $\bbR^{\widetilde{m}+1}$ with an isolated critical point at the origin, cf. \cite[Chapter 7]{CLOT03} for a discussion.
\end{rem}

We may now deduce some elementary properties of the (co)homology of $\calC\big(\{x\}\big)$. As foreshadowed by the verbiage of our previous statements, we will need to separate the minimum case from our discussion. And by the $S$-duality result (Theorem \ref{thm:sduality}), we will need to separate the maximum case from our discussion. Since our main application will be in $\bbR^n$, we will now assume that $X$ is orientable. 

\begin{prop}\label{prop:support}
Let $R$ be a commutative ring.
\begin{enumerate}
\item Suppose $x$ is not a minimum or maximum, then $\widetilde{H}_*\big(\calC\big(\{x\}\big);R\big)$ is supported in at most degrees $\{1,\ldots,n-1\}$.
\item Suppose $x$ is a minimum resp. maximum, then $\widetilde{H}_*\big(\calC\big(\{x\}\big);R\big)$ is supported in degree 0 resp. $n$. 
\end{enumerate}
\end{prop}

\begin{proof}
First we deal with case (2). If $x$ is a minimum of $f$, then $\calC\big(\{x\}\big)\simeq S^0$ and the result holds. If $x$ is a maximum of $f$, then $x$ is a minimum of $-f$. We apply Theorem \ref{thm:mccord} to see
    \begin{equation}
    \widetilde{H}_*\big(\calC\big(\{x\}\big);R\big)\cong\widetilde{H}^{n-*}(S^0;R),
    \end{equation}
and the result holds.

We now deal with case (1). Lemma \ref{lem:finitecwcomplex} shows $\widetilde{H}_*\big(\calC\big(\{x\}\big);R\big)$ is supported in at most degrees $\{0,\ldots,n\}$. Since $x$ is not a minimum of $f$, Lemma \ref{lem:pathconnected} shows $\widetilde{H}_0\big(\calC\big(\{x\}\big);R\big)\cong0$. Moreover, $x$ is not a minimum of $-f$. We apply Theorem \ref{thm:mccord} to see
    \begin{equation}
    \widetilde{H}_n\big(\calC\big(\{x\}\big);R\big)\cong\widetilde{H}^0\big(\calC\big(-\{x\}\big);R\big)\cong0,
    \end{equation}
and the result holds.
\end{proof}

\begin{prop}[Theorem 2 \cite{Dan87}]
$\widetilde{H}_*\big(\calC\big(\{x\}\big);\bbZ\big)$ is torsion-free in degrees 
    \begin{equation}
    \{\ldots,0,1,n-2,n-1,n,\cdots\};
    \end{equation}
in particular, it is free-abelian in the aforementioned degrees.
\end{prop}

\begin{rem}
Of course, $\widetilde{H}_*\big(\calC\big(\{x\}\big);\bbZ\big)$ can only been non-zero in at most degrees $\{0,\ldots,n\}$, so the real content of the previous proposition is torsion-freeness in degrees $\{0,1,n-2,n-1,n\}$.
\end{rem}

\subsection{Steenrod squares}\label{subsec:steenrodsquares}
Recall, given any space $Y$, the Steenrod squares are stable cohomology operations, 
    \begin{equation}
    \big\{\sq^j:H^m(Y;\bbZ/2)\to H^{m+j}(Y;\bbZ/2)\big\}_{j\in\bbZ_{\geq0},m\in\bbZ},
    \end{equation}
axiomatically characterized by a well-known collection of properties. Moreover, the Steenrod squares satisfy the Adems relations:
    \begin{equation}
    \sq^{j_1}\sq^{j_2}=\sum_{\ell=0}^{\lfloor j_1/2\rfloor}\sq^{j_1+j_2-\ell}\sq^\ell,\;\; j_1<2j_2.
    \end{equation}
By the Yoneda lemma, the Steenrod squares are equivalent to a collection of morphisms 
    \begin{equation}
    \big\{\sq^j:K(\bbZ/2,m)\to K(\bbZ/2,m+j)\big\}_{j\in\bbZ_{\geq0},m\in\bbZ},
    \end{equation}
satisfying suitable properties, where $K(\bbZ/2,m)$ is a $\bbZ/2$-Eilenberg-Maclane space. Hence, we may define Steenrod squares on any spectrum $\frakY$ via post-composition:
    \begin{equation}
    \sq^j:(H\bbF_2)^m(\frakY)\equiv\ho\spectra(\Sigma^{-m}\frakY,H\bbF_2)\to(H\bbF_2)^{m+j}(\frakY);
    \end{equation}
this reduces to the usual notion of Steenrod squares on spaces when $\frakY$ is a suspension spectrum.

Suppose $\frakY$ is strongly dualizable, then we wish to relate the Steenrod squares on $\frakY$ and $\calD\frakY$. The following is essentially presented in \cite{Mau68}, albeit here we present it from a more modern viewpoint. Also, we should note that the original relation regarding Steenrod squares below is due to Thom \cite{Tho52}, whereas Maunder generalized Thom's results to other stable cohomology operations arising from other cohomology theories. 

Recall, since $\frakY$ is strongly dualizable, we have the isomorphism
    \begin{equation}
    \ho\spectra(\frakZ\wedge\frakY,\frakW)\cong\ho\spectra(\frakZ,\frakW\wedge\calD\frakY);
    \end{equation}
this is natural in both $\frakZ$ and $\frakW$. In particular, we have the following relations: 
    \begin{equation}
    (H\bbF_2)^*\frakY\cong(H\bbF_2)_{-*}\calD\frakY\;\;\textrm{and}\;\;(H\bbF_2)_*\frakY\cong(H\bbF_2)^{-*}\calD\frakY. 
    \end{equation}
By naturality, we have a diagram which defines a stable homology operation $\overline{\sq}^j$:
    \begin{equation}
    \begin{tikzcd}
    \ho\spectra(\Sigma^{-m}\bbS\wedge\frakY,H\bbF_2)\arrow[r,"\sim"]\arrow[d,"\sq^j"] & \ho\spectra(\Sigma^{-m}\bbS,H\bbF_2\wedge\calD\frakY)\arrow[d,"\overline{\sq}^j"] \\
    \ho\spectra(\Sigma^{-m-j}\bbS\wedge\frakY,H\bbF_2)\arrow[r,"\sim"] & \ho\spectra(\Sigma^{-m-j}\bbS,H\bbF_2\wedge\calD\frakY)
    \end{tikzcd}
    ,
    \end{equation}
where we recall
    \begin{align}
    (H\bbF_2)^m\frakY&\equiv\ho\spectra(\Sigma^{-m}\bbS\wedge\frakY,H\bbF_2), \\
    (H\bbF_2)_{-m}\calD\frakY&\equiv\ho\spectra(\Sigma^{-m}\bbS,H\bbF_2\wedge\calD\frakY).
    \end{align}
We will refer to the collection $\{\overline{\sq}^j\}_{j\in\bbZ_{\geq0}}$ as the \emph{dual Steenrod squares}. Using the universal coefficient theorem, the dual Steenrod squares determine the \emph{conjugate Steenrod squares} $c(\sq^j)$:
    \begin{equation}
    c(\sq^j):(H\bbF_2)^{-m-j}\calD\frakY\to(H\bbF_2)^{-m}\calD\frakY.
    \end{equation}
Thom, and more generally Maunder, showed the conjugate Steenrod squares are related to the usual Steenrod squares by the equation 
    \begin{equation}\label{eq:steenrod}
    \sum_{\ell=0}^j c(\sq^\ell)\sq^{j-\ell}=0,\;\;j>0.
    \end{equation}

We return to gradient flow. Again, the setup is as follows: let $X$ be a (not necessarily closed) smooth $n$-manifold, $f\in C^\infty(X)$ a smooth function, and $g$ a Riemannian metric on $X$. We will not necessarily assume $(f,g)$ is Morse-Smale, however, we will assume there exists $x\in\crit(f)$ such that $x$ is possibly degenerate and isolated. Moreover, $X$ is assumed to be oriented.

Part (2) of Proposition \ref{prop:support} has the following straightforward consequence.

\begin{lem}
Suppose $x$ is a minimum or maximum, then 
    \begin{equation}
    \sq^j:H^*\big(\calC\big(\{x\}\big);\bbZ/2\big)\to H^{*+j}\big(\calC\big(\{x\}\big);\bbZ/2\big)
    \end{equation}
vanishes for $j>0$.
\end{lem}

Hence, for the rest of this subsection, we will assume $x$ is not a minimum or maximum. In particular, $\calC\big(\{x\}\big)$ resp. $\calC\big(-\{x\}\big)$ is a co-H-space by Corollary \ref{cor:cohspace}. 

\begin{lem}\label{lem:steenrod}
We have that 
    \begin{equation}
    \sq^j:H^j\big(\calC\big(\{x\}\big);\bbZ/2\big)\to H^{2j}\big(\calC\big(\{x\}\big);\bbZ/2\big)
    \end{equation}
vanishes for $j>0$.
\end{lem}

\begin{proof}
The result is a direct consequence of $\calC\big(\{x\}\big)$ being a co-H-space. Explicitly, let $Y$ be any co-H-space. By Theorem \ref{thm:ganea}, we have a retraction $\Sigma\Omega Y\to Y$; hence, we have a surjection 
    \begin{equation}
    H^*(\Sigma\Omega Y;\bbZ/2)\twoheadrightarrow H^*(Y;\bbZ/2).
    \end{equation}
The result follows by a commutative diagram:
    \begin{equation}
    \begin{tikzcd}
    H^{j-1}(\Omega Y;\bbZ/2)\cong H^j(\Sigma\Omega Y;\bbZ/2)\arrow[r,two heads]\arrow[d,"\sq^j"] & H^j(Y;\bbZ/2)\arrow[d,"\sq^j"] \\ 
    H^{2j-1}(\Omega Y;\bbZ/2)\cong H^{2j}(\Sigma\Omega Y;\bbZ/2)\arrow[r,two heads] & H^{2j}(Y;\bbZ/2)
    \end{tikzcd}
    ,
    \end{equation}
since Steenrod squares commute with suspension and the left arrow vanishes for degree reasons.
\end{proof}

\begin{prop}\label{prop:steenrod}
We have that 
    \begin{equation}
    \sq^j:H^*\big(\calC\big(\{x\}\big);\bbZ/2\big)\to H^{*+j}\big(\calC\big(\{x\}\big);\bbZ/2\big)
    \end{equation}
vanishes for $j\geq(n-1)/2$ (with $j\neq0$).
\end{prop}

\begin{proof}
If we consider
    \begin{equation}
    \sq^j:H^m\big(\calC\big(\{x\}\big);\bbZ/2\big)\to H^{m+j}\big(\calC\big(\{x\}\big);\bbZ/2\big),
    \end{equation}
then, since Theorem \ref{thm:mccord} is induced by an S-duality result (i.e., Theorem \ref{thm:sduality}), we have the following commutative diagram:
    \begin{equation}
    \begin{tikzcd}
    H^m\big(\calC\big(\{x\}\big);\bbZ/2\big)\arrow[d,equals]\arrow[r,"\sq^j"] & H^{m+j}\big(\calC\big(\{x\}\big);\bbZ/2\big)\arrow[d,equals] \\
    H_{n-m}\big(\calC\big(-\{x\}\big);\bbZ/2\big)\arrow[r,"\overline{\sq}^j"] & H_{n-m-j}\big(\calC\big(-\{x\}\big);\bbZ/2\big)
    \end{tikzcd}
    .
    \end{equation}
In particular, $\sq^j$ vanishes if and only if $\overline{\sq}^j$ vanishes. Now, by the universal coefficient theorem, $\overline{\sq}^j$ vanishes if and only if
    \begin{equation}
    c(\sq^j):H^{n-m-j}\big(\calC\big(-\{x\}\big);\bbZ/2\big)\to H^{n-m}\big(\calC\big(-\{x\}\big);\bbZ/2\big)
    \end{equation}
vanishes. Recall, \eqref{eq:steenrod} shows 
    \begin{equation}
    c(\sq^j)=\sum_{\ell=0}^{j-1}c(\sq^\ell)\sq^{j-\ell},
    \end{equation}
thus $c(\sq^j)$ vanishes if 
    \begin{equation}
    \sq^{j-\ell}:H^{n-m-j}\big(\calC\big(-\{x\}\big);\bbZ/2\big)\to H^{n-m-\ell}\big(\calC\big(-\{x\}\big);\bbZ/2\big),\;\;0\leq\ell\leq j-1
    \end{equation}
vanishes. First, Lemma \ref{lem:steenrod} shows this happens if $1\geq n-m-j$. Thus we have the following relation: $1\geq n-m-j$ implies 
    \begin{equation}
    \sq^j:H^m\big(\calC\big(\{x\}\big);\bbZ/2\big)\to H^{m+j}\big(\calC\big(\{x\}\big);\bbZ/2\big)
    \end{equation}
vanishes. Second, Lemma \ref{lem:steenrod} shows we have the following relation: $j\geq m$ implies 
    \begin{equation}
    \sq^j:H^m\big(\calC\big(\{x\}\big);\bbZ/2\big)\to H^{m+j}\big(\calC\big(\{x\}\big);\bbZ/2\big)
    \end{equation}
vanishes. We now add together the two inequalities; the proposition follows.
\end{proof}

Unfortunately, the bound in the previous proposition is not optimal at all.

\begin{example}\label{example:steenrod1}
Consider the case $X$ is dimension $n=7$. Proposition \ref{prop:steenrod} predicts $\sq^j$ vanishes on $H^*\big(\calC\big(\{x\}\big);\bbZ/2\big)$ when $j\geq3$. We claim the bound can be improved to $j\geq2$, and we argue as follows. By part (1) of Proposition \ref{prop:support}, $\widetilde{H}^m\big(\calC\big(\{x\}\big);\bbZ/2\big)$ is supported in at most degrees $m\in\{1,\ldots,6\}$, hence $\sq^2$ can only possibly be non-zero on $H^m\big(\calC\big(\{x\}\big);\bbZ/2\big)$ when $m\in\{2,3,4\}$. 
\begin{itemize}
\item $m=2$: $\sq^2$ vanishes on $H^2\big(\calC\big(\{x\}\big);\bbZ/2\big)$ by Lemma \ref{lem:steenrod}. 

\item $m=4$: $\sq^2$ on $H^4\big(\calC\big(\{x\}\big);\bbZ/2\big)$ corresponds to
    \begin{equation}
    c(\sq^2):H^1\big(\calC\big(-\{x\}\big);\bbZ/2\big)\to H^3\big(\calC\big(-\{x\}\big);\bbZ/2\big).
    \end{equation}
The relation \eqref{eq:steenrod} shows 
    \begin{equation}
    c(\sq^2)=\sq^2+c(\sq^1)\sq^1=\sq^2+\sq^1\sq^1=\sq^2,
    \end{equation}
where $\sq^1\sq^1=0$ by the Adem relations. In particular, $\sq^2$ vanishes on $H^1\big(\calC\big(-\{x\}\big);\bbZ/2\big)$ by elementary algebraic topology.  

\item $m=3$: $\sq^2$ on $H^3\big(\calC\big(\{x\}\big);\bbZ/2\big)$ corresponds to
    \begin{equation}
    c(\sq^2):H^2\big(\calC\big(-\{x\}\big);\bbZ/2)\to H^4(\calC\big(-\{x\}\big);\bbZ/2\big).
    \end{equation}
As above, $c(\sq^2)=\sq^2$. In particular, $\sq^2$ vanishes on $H^2(\calC\big(-\{x\}\big);\bbZ/2)$ by Lemma \ref{lem:steenrod}. 
\end{itemize} 
\end{example}

The discrepancy just described lies in how we deduced when the conjugate Steenrod squares vanish in the proof of Proposition \ref{prop:steenrod}. By recursively using \eqref{eq:steenrod} and the Adem relations, we can possibly deduce the vanishing of more Steenrod squares as in Example \ref{example:steenrod1}. In fact, even in a case where the bound in Proposition \ref{prop:steenrod} is, in general, optimal, we may still be able to detect some Steenrod squares that must vanish.

\begin{example}
Consider the case $X$ is dimension $n=8$. Proposition \ref{prop:steenrod} predicts $\sq^j$ vanishes on $H^*\big(\calC\big(\{x\}\big);\bbZ/2\big)$ when $j\geq7/2$. We claim that some $\sq^3$'s and $\sq^2$'s must vanish, and we argue as follows. By part (1) of Proposition \ref{prop:support}, $\widetilde{H}^m\big(\calC\big(\{x\}\big);\bbZ/2\big)$ is supported in at most degrees $m\in\{1,\ldots,7\}$, hence $\sq^3$ can only possibly be non-zero on $H^m\big(\calC\big(\{x\}\big);\bbZ/2\big)$ when $m\in\{3,4\}$.
\begin{itemize}
\item $m=3$: $\sq^3$ vanishes on $H^3\big(\calC\big(\{x\}\big);\bbZ/2\big)$ by Lemma \ref{lem:steenrod}.
\end{itemize}
Also, $\sq^2$ can only possibly be non-zero on $H^m\big(\calC\big(\{x\}\big);\bbZ/2\big)$ when $m\in\{2,3,4,5\}$.
\begin{itemize}
\item $m=2$: $\sq^2$ vanishes on $H^2\big(\calC\big(\{x\}\big);\bbZ/2\big)$ by Lemma \ref{lem:steenrod}.
\item $m=4$: $\sq^2$ on $H^4\big(\calC\big(\{x\}\big);\bbZ/2\big)$ corresponds to
    \begin{equation}
    c(\sq^2):H^2\big(\calC\big(-\{x\}\big);\bbZ/2\big)\to H^4\big(\calC\big(-\{x\}\big);\bbZ/2\big).
    \end{equation}
Since $c(\sq^2)=\sq^2$, this must vanish by Lemma \ref{lem:steenrod}.      
\item $m=5$: $\sq^2$ on $H^5\big(\calC\big(\{x\}\big);\bbZ/2\big)$ corresponds to
    \begin{equation}
    c(\sq^2):H^1\big(\calC\big(-\{x\}\big);\bbZ/2\big)\to H^3\big(\calC\big(-\{x\}\big);\bbZ/2\big).
    \end{equation}
Since $c(\sq^2)=\sq^2$, this must vanish by elementary algebraic topology.
\end{itemize}
\end{example}

\subsection{An example}\label{subsec:example}
We will give two examples of the following: a compact pair $(N,N_-)$, $N_-\subset N$, such that no suspension of $N/N_-$ can be the Conley index of an isolated critical point of a smooth function on a smooth 7-manifold.

First, consider the complex projective plane $\bbC P^2$ and recall
    \begin{equation}
    H^*(\bbC P^2;\bbZ/2)\cong\bbZ/2[\kappa]/(\kappa^3),\;\;\deg\kappa=2. 
    \end{equation}
Since
    \begin{equation}
    \sq^2:H^2(\bbC P^2;\bbZ/2)\to H^4(\bbC P^2;\bbZ/2)
    \end{equation}
does not vanish, the pair $(N,\emptyset)$ is one such example by Example \ref{example:steenrod1}. 

Second, \cite[Theorem A]{Fuq94} shows there is a smooth embedding $\iota:\bbC P^2\hookrightarrow\bbR^7$. We will denote by $\zeta$ the normal bundle of $\iota(\bbC P^2)$. Let $N$ be a tubular neighborhood of $\iota(\bbC P^2)$ with boundary $N_-\equiv\partial N$; we have an equivalence 
    \begin{equation}
    \th(\zeta)\simeq N/N_-. 
    \end{equation}
The Thom isomorphism theorem shows
    \begin{equation}
    \widetilde{H}^m(\th(\zeta);\bbZ/2)\cong\begin{cases}
    \bbZ/2 & m=3,5,7 \\
    0 & \textrm{otherwise}
    \end{cases}
    .
    \end{equation}
Let $u\in H^3(\th(\zeta);\bbZ/2)$ be a Thom class. The $i$-th Stiefel-Whitney class $w_i(\zeta)$ of $\zeta$ is defined by
    \begin{equation}
    w_i(\zeta)\equiv\Phi^{-1}\big(\sq^i(u)\big),
    \end{equation}
where $\Phi$ is the Thom isomorphism given by cupping with $u$. We may compute the Stiefel-Whitney classes of $\zeta$ via the relation
    \begin{equation}
    w(T\bbC P^2)w(\zeta)=1\in H^*(\bbC P^2;\bbZ/2)
    \end{equation}
and the fact that $w(T\bbC P^2)=(1+\kappa)^3$. Here, $w(T\bbC P^2)$ resp. $w(\zeta)$ is the total Stiefel-Whitney class of $T\bbC P^2$ resp. $\zeta$. In particular, we see that $w_2(\zeta)\neq0$, hence 
    \begin{equation}
    \sq^2:H^3(\th(\zeta);\bbZ/2)\to H^5(\th(\zeta);\bbZ/2)
    \end{equation}
does not vanish. The pair $(N,N_-)$ is another such example by Example \ref{example:steenrod1}.

\section{Morse theory: homotopy}
\subsection{Abouzaid-Blumberg framework}
Abouzaid-Blumberg \cite{AB24} have begun rewriting the foundations of Floer homotopy theory and, in this subsection, we will review their framework. Since we will not require the full power of the theory, we will write the following statements in less generality.\footnote{Most notably, we will not work with stratified derived orbifolds with corners; we will simply work with stratified smooth manifolds with corners.} Also, the reader only familiar with the original Cohen-Jones-Segal framework (i.e., \cite{CJS95,CK23,Lar21}) may safely use their intuition from this framework to interpret the Floer homotopy theory performed in the present article.

In this subsection, let $X$ be a smooth manifold with corners.\footnote{By a smooth manifold with corners, we will always mean a ``$\langle k\rangle$-manifold'' (cf. \cite{Lau00}).} Associated to $X$ is a category $\calP_X$, defined as follows. The objects of $\calP_X$ are given by components $\partial^\sigma X$ of corner strata of $X$, and there is a morphism $\partial^\sigma X\to\partial^\tau X$ if $\overline{\partial^\sigma X}$ contains $\partial^\tau X$. Note, the components of the interior of $X$ are the minimal objects of $\calP_X$. Moreover, there is a natural functor 
    \begin{equation}
    \codim:\calP_X\to\bbZ_{\geq0}
    \end{equation}
which records the codimension, where we give $\bbZ_{\geq0}$ the natural categorical structure induced by its standard poset structure. 

In the general theory, coarser labellings than labellings of components of boundary strata are sometimes necessary, so the following notion was introduced. Let $\calP$ be any category equipped with a functor $\codim:\calP\to\bbZ_{\geq0}$. For any object $p\in\calP$, we denote by $\calP^p$ the overcategory associated to $p$ (recall, this is the category whose objects are arrows $q\to p$ and whose morphisms are commuting triangles).

\begin{defin}[Definition 2.1 \cite{AB24}]
We say $\calP$ is a \emph{model for manifolds with corners} if, for each $p\in\calP$, $\calP^p$ is isomorphic to the poset $2^{\{1,\ldots,\codim p\}}$.\footnote{Here, we use the partial order induced by increasing inclusion. Note, the empty set is the initial object of $2^{\{1,\ldots,\codim p\}}$.}  
\end{defin}

We will refer to models of manifolds with corners simply as models. A morphism of models $\calP_0\to\calP_1$ is an embedding that (1) preserves the codimension up to an overall shift and (2) induces an isomorphism on the factorizations of a morphism in $\calP_0$ with the factorizations of its image in $\calP_1$. 

\begin{rem}
In the present article, we will only use models which are posets. In particular, we may bypass most of the technicalities that come with general models. In fact, Abouzaid-Blumberg only use models which are posets -- however, more general models will be used in their future work, cf. \cite[Remark 2.4]{AB24}.
\end{rem}

\begin{defin}[Definition 2.6 \cite{AB24}]
A \emph{stratified smooth manifold with corners} is a pair 
    \begin{equation}
    (X,F:\calP_X\to\calP),
    \end{equation}
where $X$ is a smooth manifold with corners and $F:\calP_X\to\calP$ is a functor from $\calP_X$ to a model $\calP$ such that $F$ (1) preserves the codimension and (2) induces isomorphisms $\calP_X^\sigma\xrightarrow{\sim}\calP^{F\sigma}$ for all objects $\sigma\in\calP_X$.
\end{defin}

\begin{defin}[Definition 2.9 \cite{AB24}]
A \emph{morphism of stratified smooth manifolds with corners}
    \begin{equation}
    (X,\calP_X\to\calP)\to(X',\calP_{X'}\to\calP')
    \end{equation}
consists of a smooth map of smooth manifolds with corners $X\to X'$ together with a functor $\calP\to\calP'$ such that the following induced diagram commutes:
    \begin{equation}
    \begin{tikzcd}
    \calP_X\arrow[r]\arrow[d] & \calP_{X'}\arrow[d] \\
    \calP\arrow[r] & \calP'
    \end{tikzcd}
    .
    \end{equation}
\end{defin}

Now, let $\calP$ be any set. For any two $x,y\in\calP$, we may define a model $\calP(x,y)$ as follows (cf. \cite[Definition 3.1]{AB24}). The objects are trees $T$ with only bivalent vertices whose edges are labeled by elements of $\calP$ such that the incoming leaf is labeled by $x$ resp. the outgoing leaf is labeled by $y$, there is a morphism $T_0\to T_1$ if $T_0$ is obtained from $T_1$ by collapsing internal edges, and the functor $\codim:\calP(x,y)\to\bbZ_{\geq0}$ is defined by counting the number of internal edges. There is a natural functor 
    \begin{equation}
    \calP(x,z)\times\calP(z,y)\to\calP(x,y)
    \end{equation}
given by concatenation; hence, the collection $\big\{\calP(x,y)\big\}_{x,y\in\calP}$ assembles into a strict 2-category whose objects are elements of $\calP$ and morphism categories are the categories $\calP(x,y)$ (cf. \cite[Definition 3.3]{AB24}). 

\begin{defin}[Definition 3.4 \cite{AB24}]\label{defin:unstructuredflowcategory}
An \emph{(unstructured\footnote{More generally, a \emph{structured flow category} refers to an unstructured flow category together with ``extra structure'' coming from a choice of bordism theory. In the present article, we will only work with framed flow categories, i.e., those coming with extra structure from framed bordism.}) flow category} $\bbX$ consists of the following data: 
\begin{enumerate}
\item an object set $\calP$, 
\item for every two $x,y\in\calP$, a compact smooth manifold with corners $\bbX(x,y)$ stratified by $\calP(x,y)$,
\item and, for every three $x,y,z\in\calP$, a morphism of stratified smooth manifolds with corners
    \begin{equation}
    \bbX(x,z)\times\bbX(z,y)\to\bbX(x,y),
    \end{equation}
given by an inclusion of a codimension 1 boundary stratum,\footnote{This is the interpretation, in our context, of the condition that the enriching category for unstructured flow categories used in \cite {AB24} (given by stratified derived orbifolds with corners) only has morphisms which are modeled on strong equivalences, cf. \cite[Definition 2.18]{AB24}.} which lifts the functor 
    \begin{equation}
    \calP(x,z)\times\calP(z,y)\to\calP(x,y)
    \end{equation}
such that, for every four $x,y,z_1,z_2\in\calP$, the diagram 
    \begin{equation}
    \begin{tikzcd}\label{eq:unstructureddiagram}
    \bbX(x,z_1)\times\bbX(z_1,z_2)\times\bbX(z_2,y)\arrow[r]\arrow[d] & \bbX(x,z_1)\times\bbX(z_1,y)\arrow[d] \\
    \bbX(x,z_2)\times\bbX(z_2,y)\arrow[r] & \bbX(x,y)
    \end{tikzcd}
    \end{equation}
commutes. Moreover, we require the codimension 1 boundary strata of $\bbX(x,y)$ are enumerated by the aforementioned morphisms.
\end{enumerate}
\end{defin}

\begin{rem}
By part (3) of the previous definition, it can only be the case that, for any two $x,y\in\calP$, either $\bbX(x,y)$ or $\bbX(y,x)$ is non-empty. In particular, $\calP$ has a natural partial ordering. 
\end{rem}

A simple example of an unstructured flow category is the \emph{unit flow category} $\unit$; this consists of a single object and no morphisms.

\begin{defin}[Definition 3.7 \cite{AB24}]\label{defin:framedflowcategory}
A \emph{framed flow category} is a flow category $\bbX$ together with the following additional data:
\begin{enumerate}
\item a real virtual vector space $V_x$ for each $x\in\calP$,
\item a real vector bundle $W(x,y)\to\bbX(x,y)$ for each two $x,y\in\calP$ satisfying 
    \begin{equation}
    W(x,z)\oplus W(z,y)\cong W(x,y),\;\;\forall x,y,z\in\calP,
    \end{equation}

\item a trivial real virtual bundle $I(x,y)\to\bbX(x,y)$ for each two $x,y\in\calP$ satisfying 
    \begin{equation}
    I(x,z)\oplus I(z,y)\cong I(x,y),\;\;\forall x,y,z\in\calP,
    \end{equation}
\item and a choice of equivalence of real virtual bundles
    \begin{equation}\label{eq:frameddiagram}
    T\bbX(x,y)\oplus\underline{V}_y\oplus\underline{\bbR}\cong I(x,y)\oplus W(x,y)\oplus\underline{V}_x
    \end{equation}
over $\bbX(x,y)$, for each two $x,y\in\calP$, such that the natural associativity diagram commutes.
\end{enumerate}
\end{defin}

\begin{rem}
In the present article, we will only work with framed flow categories such that, for any two $x,y\in\calP$, $W(x,y)$ and $I(x,y)$ are trivial of (virtual) rank 0.
\end{rem}

A simple example of a framed flow category is $\Sigma^d\unit$, $d\in\bbZ$; this has underlying unstructured flow category $\unit$ together with the framing consisting of a real virtual vector space $V$ of virtual rank $d$ over the single object.

\cite[Theorem 1.6]{AB24} shows there exists a stable $\infty$-category $\flow^\fr$ whose objects are framed flow categories. More generally, the $d$-simplices in $\flow^\fr$ are given by \emph{framed flow $d$-simplices} (cf. \cite[Section 4]{AB24}). We will not reiterate the definition in the present article since it would be a large digression. The idea of framed flow simplices is to categorify the familiar notions of: continuation maps, homotopies of continuation maps, and so on. Moreover, \cite[Proposition 1.10]{AB24} shows there is an equivalence of stable $\infty$-categories between $\flow^\fr$ and $\spectra$.

\subsection{Morse homotopy type}
In this subsection, we will prove a generalization of the classical result of Cohen-Jones-Segal. Recall, in \cite{CJS95} it was shown that the Morse homotopy type of a closed smooth manifold is equivalent to the suspension spectrum of that manifold; we will generalize this to Conley index theory.

Again, our setup is as follows. Let $X$ be a (not necessarily closed) smooth $n$-manifold, $f\in C^\infty(X)$ a smooth function, and $g$ a Riemannian metric on $X$. We may assume $(f,g)$ is Morse-Smale; moreover, we will assume the maximal invariant subset of $f$ is isolated. We denote by $\bbM^{f,g}_X$ the flow category with object set $\crit(f)$ and morphism spaces $\bbM^{f,g}_X(x,y)$ (the composition is the standard gluing map of broken gradient flows).

\begin{prop}\label{prop:morsehomotopytype}
There exists a lift of $\bbM^{f,g}_X$ to a framed flow category together with the following homotopy equivalence:
    \begin{equation}\label{eq:morseequivalence}
    \frakM^{f,g}_X\equiv\flow^\fr\Big(\unit,\bbM^{f,g}_X\Big)\simeq\Sigma^\infty\calC_f.
    \end{equation}
\end{prop}

This result has been proven, at least in the closed case, in various guises (cf. \cite[Appendix D]{AB21}, \cite[Example 1.7]{AB24}, \cite[Main Theorem]{Bon24}, \cite{CJS95}, \cite[Theorem 3.11]{CK23}); this generalization follows the exact same style of proof. Therefore, we will only sketch the proof since this is the first time it has appeared in the Abouzaid-Blumberg framework. Moreover, the homotopy equivalence portion of this proof does require a considerable amount more from the general theory in \cite{AB24}, and the reader may safely skip this part of the proof with no loss of understanding. 

\begin{rem}
In the remainder of the present article we will invoke standard results about operator gluing resp. the spectral flow,  cf. \cite{FH93} resp. \cite{RS95}. Note, the choices involved in operator gluing lie in contractible spaces. 
\end{rem}

\begin{rem}
The framing we are about to construct for the Morse homotopy type will use the notion of ``abstract caps'' coming from the functional-analytic framework of Morse theory. Note, this is not the ``standard'' way to frame the Morse flow category; there is a more direct way, cf. \cite[Appendix D]{AB21}. Fortunately, the abstract cap framing gives the same homotopical information as the direct framing since both recover the suspension spectrum of the maximal Conley index. However, the advantage to using abstract caps is that we may more easily relate the framings in Morse and Floer theory.
\end{rem}

In the remainder of the present article, we fix a sufficiently small positive constant $\lambda\in\bbR_{>0}$. We begin with a lemma. 

\begin{lem}\label{lem:standardmorseoperator}
Let $F\to\bbR$ be a real vector bundle with connection $\nabla$. The operator 
    \begin{align}
    S:W^{1,2}(\bbR;F)&\to L^2(\bbR;F) \\
    \xi(s)&\mapsto\nabla_s\xi(s)+\lambda\xi(s) \nonumber
    \end{align}
is invertible
\end{lem}

\begin{proof}
Observe, $\nabla$ induces a trivialization $F\cong\underline{\bbR}^{\widetilde{m}}$ together with an identification of $S$ with the operator
    \begin{align}
    S:W^{1,2}(\bbR;\underline{\bbR}^{\widetilde{m}})&\to L^2(\bbR;\underline{\bbR}^{\widetilde{m}}) \\
    \xi(s)&\mapsto\partial_s\xi(s)+\lambda\xi(s). \nonumber
    \end{align}
By examining the spectral flow, we see $S$ is a Fredholm operator of index 0. Suppose for contradiction that there is a non-zero element $\xi\in\ker S$, then $\xi$ is a non-zero solution of the ODE 
    \begin{equation}
    \partial_s\xi(s)=-\lambda\xi(s).
    \end{equation}
It follows that $\xi_j(s)=e^{-\lambda s}$, where $\xi=(\xi_1,\ldots,\xi_{\widetilde{m}})$. But this contradicts the assumption that $\xi$ is $W^{1,2}$, hence $\ker S=0$; the lemma follows. 
\end{proof}

Consider $\gamma\in\widehat{\calM}^{f,g}_X(x,y)$ and recall the associated operator $D(\Xi_{f,g})_\gamma$. We may also consider the invertible self-adjoint operator
    \begin{align}
    D(\Xi_{f,g})_x\in\operatorname{End}(T_xX)\;\;\textrm{resp.}\;\;D(\Xi_{f,g})_y\in\operatorname{End}(T_yY)
    \end{align}
given by linearizing $\Xi_{f,g}$ at $x$ resp. $y$. Note, we have that $D(\Xi_{f,g})_\gamma$ is asymptotic to $D(\Xi_{f,g})_x$ resp. $D(\Xi_{f,g})_y$ at $-\infty$ resp. $+\infty$. For the rest of this subsection, we work with the Levi-Civita connection $\nabla^g$. Observe, $\nabla^g$ induces a trivialization 
    \begin{equation}
    T_xX\cong\bbR^n,
    \end{equation}
which trivializes $\nabla^g$, and an identification of $D(\Xi_{f,g})_x$ with the invertible symmetric matrix 
    \begin{equation}
    \nabla^g\grad f(x)\in\operatorname{End}(\bbR^n).
    \end{equation}

\begin{defin}\label{defin:morseabstractcap}
A \emph{Morse abstract cap} for $x$ is an operator of the form
    \begin{align}
    T_{\frakM,x}:W^{1,2}(\bbR;\bbR^n)&\to L^2(\bbR;\bbR^n) \label{eq:form1} \\
    \xi(s)&\mapsto\partial_s\xi(s)+\calT_{\frakM,x}(s)\xi(s), \nonumber
    \end{align}
where $\calT_{\frakM,x}(s)\in\operatorname{End}(\bbR^n)$ are symmetric matrices that satisfy
    \begin{equation}\label{eqn:form2}
    \lim_{s\to-\infty}\calT_{\frakM,x}(s)=\nabla^g\grad f(x)\;\;\textrm{and}\;\;\lim_{s\to+\infty}\calT_{\frakM,x}(s)=\lambda\cdot\identity.
    \end{equation}
\end{defin}

\begin{rem}
Let $\calT^*_{\frakM,x}(s)\in\operatorname{End}(\bbR^n)$ satisfy
    \begin{equation}
    \calT^*_{\frakM,x}(s)\equiv\calT_{\frakM,x}(-s).
    \end{equation}
A straightforward operator gluing argument shows the operator
    \begin{align}
    T^*_{\frakM,x}:W^{1,2}(\bbR;\bbR^n)&\to L^2(\bbR;\bbR^n) \\
    \xi(s)&\mapsto\partial_s\xi(s)+\calT^*_{\frakM,x}(s)\xi(s) \nonumber 
    \end{align}
has index bundle satisfying 
    \begin{equation}\label{eqn:morseabstractcapequiv}
    \ind T^*_{\frakM,x}(s)\cong-\ind T_{\frakM,x}.\footnote{Recall, the index bundle of a family of Fredholm operators $E$ on a real Hilbert space (parameterized by a space $Y$) is the real virtual bundle defined fiberwise by
        \begin{equation}
        \ind E\vert_y\equiv\ker E\vert_y-\coker E\vert_y
        \end{equation}
    when the kernel/cokernel have locally constant rank (there is a straightforward extension to the case when the kernel/cokernel do not have locally constant rank, cf. \cite[Appendix A]{Ati67}).}
    \end{equation}
\end{rem}

By examining the spectral flow, we see any Morse abstract cap for $x$ is a Fredholm operator of index $I(x)$. Moreover, we see the space of Morse abstract caps for $x$ is contractible since the space of operators of the form \eqref{eq:form1} with fixed asymptotics is contractible (i.e., the space of families of symmetric matrices with asymptotics \eqref{eqn:form2} is contractible).

\begin{proof}[Proof sketch of Proposition \ref{prop:morsehomotopytype}]
First, we will lift $\bbM^{f,g}_X$ to a framed flow category. Consider the glued together Fredholm operator
    \begin{equation}
    T_{\frakM,x,y,\gamma}\equiv T^*_{\frakM,x}\#D(\Xi_{f,g})_\gamma\#T_{\frakM,y}.
    \end{equation}
Under standard operator gluing, the index bundle is additive: 
    \begin{equation}
    \ind T_{\frakM,x,y,\gamma}\cong \ind T^*_{\frakM,x}+\ind D(\Xi_{f,g})_\gamma+\ind T_{\frakM,y}.
    \end{equation}
    
We claim $T_{\frakM,x,y,\gamma}$ can be deformed to an invertible operator in a way independent of choices. We have that $T_{\frakM,x,y,\gamma}$ is of the form
    \begin{align}
    T_{\frakM,x,y,\gamma}:W^{1,2}(\bbR;\gamma^*TX)&\to L^2(\bbR;\gamma^*TX) \label{eq:form1'} \\
    \xi(s)&\mapsto\nabla^g_s\xi(s)+A(s)\xi(s),\nonumber
    \end{align}
where $A(s)\in\operatorname{End}(\gamma^*TX\vert_s)$ are symmetric matrices such that 
    \begin{equation}
    \lim_{s\to\pm\infty}A(s)=\lambda\cdot\identity.
    \end{equation}
Since the space of operators of the form \eqref{eq:form1'} with fixed asymptotics is contractible, we may identity $T_{\frakM,x,y,\gamma}$ with the operator 
    \begin{align}
    T_{\frakM,x,y,\gamma}:W^{1,2}(\bbR;\gamma^*TX)&\to L^2(\bbR;\gamma^*TX) \\
    \xi(s)&\mapsto\nabla^g_s\xi(s)+\lambda\xi(s);\nonumber
    \end{align}
the claim follows by Lemma \ref{lem:standardmorseoperator}.

Now, we have an equivalence of real virtual vector spaces:
    \begin{equation}\label{eq:discussion1}
    \ind T_{\frakM,x}\cong\ind D(\Xi_{f,g})_\gamma+\ind T_{\frakM,y}.
    \end{equation}
Since all choices live in contractible spaces, we can upgrade this equivalence to families of operators. In particular, consider the fiber bundle $\fred\to\widehat{\calM}^{f,g}_X(x,y)$, with fiber 
    \begin{equation}
    \fred\vert_\gamma\equiv\fred\big(W^{1,2}(\bbR;\gamma^*TX),L^2(\bbR;\gamma^*TX)\big),
    \end{equation}
together with the section 
    \begin{align}
    D\Xi_{f,g}:\widehat{\calM}^{f,g}_X(x,y)&\to\fred \\
    \gamma&\mapsto D(\Xi_{f,g})_\gamma.\nonumber
    \end{align}
By the discussion in Subsection \ref{subsec:morsehomologyfunctional}, we have an isomorphism of real vector bundles:
    \begin{equation}
    T\widehat{\calM}^{f,g}_X(x,y)\cong\ind D\Xi_{f,g}.
    \end{equation}
Saying \eqref{eq:discussion1} extends to families is equivalent to saying we have constructed the following equivalence of real virtual bundles:
    \begin{equation}
    T\widehat{\calM}^{f,g}_X(x,y)\oplus\underline{\bbR}^{I(y)}\cong\underline{\bbR}^{I(x)},
    \end{equation}
where we have used the canonical isomorphisms 
    \begin{equation}
    \ind T_{\frakM,x}\cong\underline{\bbR}^{I(x)}\;\;\textrm{and}\;\;\ind T_{\frakM,y}\cong\underline{\bbR}^{I(y)}.
    \end{equation}
Finally, by using the isomorphism of real vector bundles 
    \begin{equation}
    T\calM^{f,g}_X(x,y)\oplus\underline{\bbR}\cong T\widehat{\calM}^{f,g}_X(x,y)
    \end{equation}
and the homotopy equivalence 
    \begin{equation}
    \bbM^{f,g}_X(x,y)\simeq\calM^{f,g}_X(x,y),
    \end{equation}
we have the following equivalence of real virtual bundles:
    \begin{equation}\label{eq:morseframing}
    T\bbM^{f,g}_X(x,y)\oplus\underline{\bbR}^{I(y)}\oplus\underline{\bbR}\cong\underline{\bbR}^{I(x)}.
    \end{equation}
    
It remains to show \eqref{eq:morseframing} is compatible with gluing of broken gradient flows, but this follows from (1) operator gluing and (2) \eqref{eqn:morseabstractcapequiv}. Thus, we have a lift of $\bbM^{f,g}_X$ to a framed flow category.

Second, we will prove the claimed equivalence; we fix on $\bbM^{f,g}_X$ the framing induced by \eqref{eq:morseframing}. Let $\Delta$ be the simplex category, $\Delta^m$ the standard $m$-simplex defined by 
    \begin{equation}
    \Delta^m[\bullet]\equiv\hom_\Delta(\bullet,[m]),
    \end{equation}
and $\sset$ the category of simplicial sets. The 0-th simplicial set of the mapping spectrum $\flow^\fr\Big(\unit,\bbM^{f,g}_X\Big)$ is defined to be the simplicial set whose set of $\bullet$-simplices is
    \begin{equation}\label{eqn:0thsimplicialset}
    \Bigg\{\bbX\in\hom_{\sset}\Big(\Delta^1\times\Delta^\bullet,\flow^\fr\Big):\begin{aligned}
    & \bbX\vert_{\brak{0}\times\Delta^\bullet}=\unit \\
    & \bbX\vert_{\brak{1}\times\Delta^\bullet}=\bbM^{f,g}_X
    \end{aligned}
    \Bigg\},
    \end{equation}
cf. \cite[Section 7]{AB24}; the $j$-th simplicial set of $\flow^\fr\Big(\unit,\bbM^{f,g}_X\Big)$ is simply the $j$-fold delooping of \eqref{eqn:0thsimplicialset}. Let $\bbX_\emptyset:\bbX\to\bbX$ denote the empty framed flow bimodule; this represents the trivial map (cf. \cite[Lemma 7.6]{AB24}). Moreover, let $s\bbX:\bbX\to\bbX$ denote the diagonal framed flow bimodule; this represents the identity map (cf. \cite[Section 6.2]{AB24}). For example, an element of: 
\begin{itemize}
\item the set of 0-simplices of \eqref{eqn:0thsimplicialset} is a choice of framed flow bimodule 
    \begin{equation}
    \unit\to\bbM^{f,g}_X,
    \end{equation}
\item  the set of 1-simplices of \eqref{eqn:0thsimplicialset} is a choice of homotopy commutative diagram 
    \begin{equation}
    \begin{tikzcd}
    \unit\arrow[r] & \bbM^{f,g}_X \\
    \unit\arrow[u,"s\unit"]\arrow[r] & \bbM^{f,g}_X\arrow[u,"s\bbM^{f,g}_X"]
    \end{tikzcd}
    ,
    \end{equation}
\item the set of 2-simplices of \eqref{eqn:0thsimplicialset} is a choice of homotopy commutative diagram
    \begin{equation}
    \begin{tikzcd}
    & \unit\arrow[dr,"s\unit"]\arrow[dd,dashed] & \\
    \unit\arrow[ur,"s\unit"]\arrow[rr,"s\unit" {xshift=-2ex}]\arrow[dd] & & \unit\arrow[dd] \\
    & \bbM^{f,g}_X\arrow[dr,dashed,"s\bbM^{f,g}_X"] & \\
    \bbM^{f,g}_X\arrow[ur,dashed,"s\bbM^{f,g}_X"]\arrow[rr,"s\bbM^{f,g}_X"] & & \bbM^{f,g}_X
    \end{tikzcd}
    ,
    \end{equation}
\item and so on for higher-dimensional simplices. 
\end{itemize}
We wish to construct an equivalence 
    \begin{equation}
    \Sigma^\infty\calC_f\xrightarrow{\sim}\frakM^{f,g}_X.
    \end{equation}
By the standard adjunction,
    \begin{equation}
    \hom_{\sset}\Big(\calC_f,\Omega^\infty\frakM^{f,g}_X\Big)\cong\hom_{\spectra}\Big(\Sigma^\infty\calC_f,\frakM^{f,g}_X\Big),
    \end{equation}
it will suffice to construct a morphism of simplicial sets 
    \begin{equation}\label{eqn:mapprelim}
    \calC_f\to\Omega^\infty\frakM^{f,g}_X,
    \end{equation}
where we identify $\calC_f$ with its singular simplicial set $\sing\calC_f$, and show the induced map of spectra is a weak equivalence.\footnote{Weak equivalence suffices since both $\Sigma^\infty\calC_f$ and $\frakM^{f,g}_X$ are bounded below.} Note, $\Omega^\infty\frakM^{f,g}_X$ is defined as a colimit of iterated stabilizations $\Omega^j\Sigma^j$ of \eqref{eqn:0thsimplicialset}; thus, we may define \eqref{eqn:mapprelim} by defining a map from $\calC_f$ to \eqref{eqn:0thsimplicialset}, where the latter is the initial object of the colimit, and passing to the colimit.

Let $(N,N_-,N_+)$ be an index triple for $S_f$. Since $N_-$ is a submanifold of $N$, we may triangulate $N_-$ and extend to a triangulation of $N$; moreover, we may assume each simplex is a smoothly embedded submanifold with corners. In particular, we may identify $\sing\calC_f$ with the quotient simplicial set
    \begin{align}
    \sing(N,N_-)[\bullet]&\equiv\sing N[\bullet]/\sing N_-[\bullet] \\
    &=\big\{\sigma\in\sing N[\bullet]:\sigma(\Delta^\bullet)\cap(N-N_-)\neq\emptyset\big\}\cup\{*\}.
    \end{align}

\begin{warning}
This is an immense abuse of notation. We are using the notation $\sing N$ and $\sing N_-$ to denote the simplicial set whose simplices are only those in the fixed triangulation of $(N,N_-)$, i.e., not the simplicial set consisting of \emph{all} singular simplices as is typical of this notation. In particular, $\sing\calC_f$ now denotes the simplicial set whose simplices are only those induced by the fixed triangulation of $(N,N_-)$. We hope no confusion arises.
\end{warning}

Two observations to make, which both come from the fact that $N_-$ is an exit set for $-\grad f$, are the following.
\begin{enumerate}
\item We have that, for any $\sigma\in\sing N[\bullet]$, 
    \begin{equation}
    \sigma\big(\sigma^{-1}(N_-)\big)\cap W^s(x)=\emptyset,\;\;x\in\crit(f).
    \end{equation}
By perturbing each relative simplex $\sigma\in\sing(N,N_-)[\bullet]$ within its homotopy class, without altering the map on $\sigma^{-1}(N_-)$, we may assume from the onset that each $\sigma$ intersects every stable manifold of $f$ transversely. 
\item We have that, for any $x\in\crit(f)$, the map 
    \begin{equation}\label{eqn:eval}
    {\rm Eval}:W^s(x)\to N-N_-,
    \end{equation}
given by evaluating a half gradient flow at its starting point, is proper.
\end{enumerate}
We now construct a map $\sing(N,N_-)$ to \eqref{eqn:0thsimplicialset} as follows.
\begin{enumerate}
\item Given any $x\in\crit(f)$ and $\sigma\in\sing(N,N_-)[0]$, we define
    \begin{equation}
    \bbX^{\brak{0}}_\sigma(*,x)\equiv{\rm Eval}^{-1}\big(\sigma(\Delta^0)\big)
    \end{equation}
to be the moduli space of half gradient flows starting on $\sigma(\Delta^0)$ and ending at $x$. Moreover, since $N_-$ is an exit set for $-\grad f$, we see any 0-simplex in $\sing N_-$ maps to the empty framed flow bimodule. We may naturally construct a stable framing on $\bbX^{\brak{0}}_\sigma(*,x)$ via gluing Morse abstract caps. In particular, we define $\bbX^{\brak{0}}_\sigma$ to be the 0-simplex in \eqref{eqn:0thsimplicialset} defined as the framed flow bimodule determined by the $\bbX^{\brak{0}}_\sigma(*,x)$'s.

\item Given any $x\in\crit(f)$ and $\sigma\in\sing(N,N_-)[1]$, we define
    \begin{equation}
    \widetilde{\bbX}^{\brak{01}}_\sigma(*,x)\equiv{\rm Eval}^{-1}\big(\sigma(\Delta^1)\big)
    \end{equation}
to be the moduli space of half gradient flows starting on $\sigma(\Delta^1)$ and ending at $x$; this is (generically) a smooth manifold. Moreover, since $N_-$ is an exit set for $-\grad f$, we see any 1-simplex in $\sing N_-$ maps to the homotopy commutative diagram which fills in the empty framed flow bimodules and obvious homotopy. We may naturally construct a stable framing of $\widetilde{\bbX}^{\brak{01}}_\sigma(*,x)$ via gluing Morse abstract caps. We define 
    \begin{equation}
    \bbX^{\brak{01}}_\sigma(*,x)
    \end{equation}
to be the Morse-theoretic compactification of $\widetilde{\bbX}^{\brak{01}}_\sigma(*,x)$ given by allowing breakings at critical points; this inherits a natural stable framing. In particular, we define $\bbX^{\brak{01}}_\sigma$ to be the 1-simplex in \eqref{eqn:0thsimplicialset} defined as the homotopy 
    \begin{equation}
    \bbX^{\brak{01}}_\sigma:\bbX^{\brak{0}}_{\partial^{\brak{0}}\sigma}\Rightarrow \bbX^{\brak{1}}_{\partial^{\brak{1}}\sigma},
    \end{equation}
where $\bbX^{\brak{j}}_\sigma$ is constructed as in (1), using the $\bbX^{\brak{01}}_\sigma(*,x)$'s.
\item Given any $x\in\crit(f)$ and $\sigma\in\sing(N,N_-)[2]$, we define 
    \begin{equation}
    \widetilde{\bbX}^{\brak{012}}_\sigma(*,x)\equiv{\rm Eval}^{-1}\big(\sigma(\Delta^2)\big)
    \end{equation}
to be the moduli space of half gradient flows starting on $\sigma(\Delta^2)$ and ending at $x$; this is (generically) a smooth manifold. Moreover, since $N_-$ is an exit set for $-\grad f$, we see any 2-simplex in $\sing N_-$ maps to the homotopy commutative diagram which fills in the empty framed flow bimodules, obvious homotopies, and obvious homotopy of homotopies. We may naturally construct a stable framing of $\widetilde{\bbX}^{\brak{012}}_\sigma(*,x)$ via gluing Morse abstract caps. We define 
    \begin{equation}
    \bbX^{\brak{012}}_\sigma(*,x)
    \end{equation}
to be the Morse-theoretic compactification of $\widetilde{\bbX}^{\brak{012}}_\sigma(*,x)$ given by allowing breakings at critical points; this inherits a natural stable framing. In particular, we define $\bbX^{\brak{012}}_\sigma$ to be the 2-simplex in \eqref{eqn:0thsimplicialset} defined as the homotopy 
    \begin{equation}
    \bbX^{\brak{012}}_\sigma:\bbX^{\brak{12}}_{\partial^{\brak{12}}\sigma}\circ\bbX^{\brak{01}}_{\partial^{\brak{01}}\sigma}\Rightarrow\bbX^{\brak{02}}_{\partial^{\brak{02}}\sigma},
    \end{equation}
where $\bbX^{\brak{jk}}_{\partial^{\brak{jk}}\sigma}$ is constructed as in (2), using the $\bbX^{\brak{012}}_\sigma(*,x)$'s.
\item And so on for higher-dimensional simplices.
\end{enumerate}
We have now constructed a map $\sing(N,N_-)$ to \eqref{eqn:0thsimplicialset}; hence, we have constructed a map 
    \begin{equation}
    \sing\calC_f\cong\sing(N,N_-)\to\Omega^\infty\frakM^{f,g}_X.
    \end{equation}
This induces a map of spectra
    \begin{equation}\label{eqn:mapofspectra}
    \Sigma^\infty\calC_f\to\frakM^{f,g}_X,
    \end{equation}
and we consider the induced map on homotopy groups:
    \begin{equation}\label{eqn:homotopygroups}
    \pi_m\big(\Sigma^\infty\calC_f\big)\to\pi_m\big(\frakM^{f,g}_X\big)\equiv\ho\flow^\fr\Big(\Sigma^m\unit,\bbM^{f,g}_X\Big), 
    \end{equation}
where $\ho\flow^\fr$ is the homotopy category associated to $\flow^\fr$.

It remains to show \eqref{eqn:homotopygroups} is an isomorphism; this follows the same geometric idea of the various arguments in the closed case. The spectrum $\frakM^{f,g}_X$ is filtered by the values of $f$; thus, by perturbing so that the critical values of $f$ are distinct and restricting to $f^{-1}\big([a,b]\big)$, we immediately reduce to the case of a single non-degenerate critical point $x$ (i.e., \eqref{eqn:mapofspectra} is a map of filtered spectra, each of which is bounded below; thus, it suffices to establish the induced map on the associated graded is a weak equivalence); we proceed as follows. Since $(f,g)$ is Morse-Smale, we have that the maximal Conley index of $x$ is simply $S^{I(x)}$ (where the latter is thought of as a pointed space). Moreover, we have that $\bbM^{f,g}_X=\Sigma^{I(x)}\unit$. In particular, \eqref{eqn:mapofspectra}, in the case of a single critical point $x$, becomes 
    \begin{equation}\label{eqn:equiv}
    \Sigma^{I(x)}\bbS\to\flow^\fr\Big(\unit,\Sigma^{I(x)}\unit\Big).
    \end{equation}
Now, \cite[Proposition 1.10]{AB24} shows $\flow^\fr$ and $\spectra$ are equivalent stable $\infty$-categories via the equivalence
    \begin{equation}
    \flow^\fr(\unit,-):\flow^\fr\to\spectra;
    \end{equation}
moreover, \cite[Corollary 8.12]{AB24} shows 
    \begin{equation}
    \Sigma^{I(x)}\bbS\simeq\flow^\fr\Big(\unit,\Sigma^{I(x)}\unit\Big).
    \end{equation}
We wish to show \eqref{eqn:equiv} is an equivalence; it suffices to show the induced map on $\pi_{I(x)}$ maps 1 to 1:
     \begin{equation}\label{eqn:piix}
    \pi_{I(x)}\big(\Sigma^{I(x)}\bbS)\cong\bbZ\to\pi_{I(x)}\big(\Sigma^{I(x)}\bbS\big)\cong\bbZ.
    \end{equation}
Let us recall the specific situation we are in and how we obtained the map \eqref{eqn:equiv}. We may take $N$ to be a small open ball centered at $x$ such that the restriction $f\vert_N$ has a unique critical point $x$ at the center of $N$. Thus, as already mentioned:
    \begin{equation}
    \calC_f\simeq S^{I(x)},\;\;\bbM^{f,g}_X=\Sigma^{I(x)}\unit.
    \end{equation}
One thing to note is that $\calC_f$, as we already know, is actually defined to be the quotient $N/N_-$. In particular, the homotopy equivalence $\calC_f\simeq S^{I(x)}$ is constructed by identifying a specific $I(x)$-sphere inside of $\calC_f$ which contains the basepoint $x_0\in\calC_f$ (where $x_0$ is necessarily distinct from $x$ since $N_-\subset\partial N$). The desired $I(x)$-sphere is given by 
    \begin{equation}
    \big(W^u(x)\cap N\big)/\big(W^u(x)\cap N_-\big);
    \end{equation}
the desired homotopy equivalence is then straightforward to construct. Now, we may identify $\sing\calC_f$ with the simplicial set consisting of the non-degenerate simplices of $\Delta^{I(x)}/\partial\Delta^{I(x)}$, i.e., we may assume $\sing\calC_f$ satisfies 
    \begin{equation}
    \sing\calC_f[\bullet]=\begin{cases}
    \{\sigma_\bullet\}, & \bullet\leq I(x) \\
    \emptyset, & \textrm{otherwise}
    \end{cases}
    ,
    \end{equation}
where 
    \begin{equation}
    \sigma_\bullet(\Delta^\bullet)=\begin{cases}
    S^{I(x)}, & \bullet=I(x) \\
    x_0, & \bullet\leq I(x)-1
    \end{cases}
    .
    \end{equation}
(Of course, $\sigma_{I(x)}:\Delta^{I(x)}\to S^{I(x)}$ is the map which crushes $\partial\Delta^{I(x)}$ to $x_0$.) In particular, by the discussion above, we have constructed a map
    \begin{equation}\label{eqn:equiv2}
    \sing\calC_f\to\underline{\flow}^\fr\Big(\unit,\Sigma^{I(x)}\unit\Big),
    \end{equation}
where $\underline{\flow}^\fr\Big(\unit,\Sigma^{I(x)}\unit\Big)$ is the 0-th simplicial set of $\flow^\fr\Big(\unit,\Sigma^{I(x)}\unit\Big)$, given by sending: $\sigma_\bullet$, $\bullet\leq I(x)-1$, to 
    \begin{equation}
    \bbX^{\langle0\cdots\bullet\rangle}_{\sigma_\bullet}\in\underline{\flow}^\fr\Big(\unit,\Sigma^{I(x)}\unit\Big)[\bullet],\;\;\bbX^{\langle0\cdots\bullet\rangle}_{\sigma_\bullet}(*,x)\equiv\emptyset;
    \end{equation}
and $\sigma_{I(x)}$ to
    \begin{equation}\label{eqn:generatorofhomotopy}
    \bbX^{\langle0\cdots I(x)\rangle}_{\sigma_{I(x)}}\in\underline{\flow}^\fr\Big(\unit,\Sigma^{I(x)}\unit\Big)\big[I(x)\big],\;\;\bbX^{\langle0\cdots I(x)\rangle}_{\sigma_{I(x)}}(*,x)\equiv*,
    \end{equation}
where $\bbX^{\langle0\cdots I(x)\rangle}_{\sigma_{I(x)}}(*,x)$ has the stable framing corresponding to the positive orientation. We would like to emphasize that $\bbX^{\langle0\cdots\bullet\rangle}_{\sigma_\bullet}(*,x)=\emptyset$, $\bullet\leq I(x)-1$, resp. $\bbX^{\langle0\cdots I(x)\rangle}_{\sigma_{I(x)}}(*,x)=*$ since the image of $W^s(x)$ inside of $\calC_f$ intersects $\big(W^u(x)\cap N\big)/\big(W^u(x)\cap N_-\big)$ precisely in $\{x\}$. Since 
    \begin{equation}
    \Omega^\infty\Sigma^{I(x)}\bbS\simeq\Omega^\infty\flow^\fr\Big(\unit,\Sigma^{I(x)}\unit\Big),
    \end{equation}
we have that 
    \begin{equation}
    \pi^{\rm st}_{I(x)}\bigg(\underline{\flow}^\fr\Big(\unit,\Sigma^{I(x)}\unit\Big)\bigg)\cong\pi_{I(x)}\bigg(\Omega^\infty\flow^\fr\Big(\unit,\Sigma^{I(x)}\unit\Big)\bigg)\cong\bbZ,
    \end{equation}
where $\pi^{\rm st}_{I(x)}(-)$ denotes the $I(x)$-th stable homotopy group. Now, by construction, the map on $\pi^{\rm st}_{I(x)}$ induced by \eqref{eqn:equiv2} maps 1 to 1: 
    \begin{equation}\label{eqn:previousexpression}
    \pi^{\rm st}_{I(x)}(\sing\calC_f)\cong\bbZ\to\pi^{\rm st}_{I(x)}\bigg(\underline{\flow}^\fr\Big(\unit,\Sigma^{I(x)}\unit\Big)\bigg)\cong\bbZ;
    \end{equation}
this is because the homotopy class of the framed flow bimodule 
    \begin{equation}
    \Sigma^{I(x)}\unit\to\Sigma^{I(x)}\unit
    \end{equation}
associated to \eqref{eqn:generatorofhomotopy} is exactly the generator of 
    \begin{align}
    \ho\flow^\fr\Big(\Sigma^{I(x)}\unit,\Sigma^{I(x)}\unit\Big)&\equiv\pi_{I(x)}\bigg(\flow^\fr\Big(\unit,\Sigma^{I(x)}\unit\Big)\bigg) \\
    &\cong\pi^{\rm st}_{I(x)}\bigg(\underline{\flow}^\fr\Big(\unit,\Sigma^{I(x)}\unit\Big)\bigg).
    \end{align}
Finally, by unwinding the adjunction $\Sigma^\infty\dashv\Omega^\infty$, \eqref{eqn:previousexpression} is exactly \eqref{eqn:piix}.
\end{proof}

\section{Floer homotopy on Liouville manifolds}\label{sec:floerhomotopytheory}
\subsection{Background geometry}
In the next three subsections, we will review the construction of Lagrangian Floer cohomology on Liouville manifolds, cf. \cite{Abo10, AS10a}. In the following subsection, we will review the construction of the Floer homotopy type for a pair of Lagrangians in a Liouville manifold, cf. \cite{Lar21}.

Let $\Big(\widehat{M},\widehat{\omega}=d\widehat{\theta}\Big)$ be an exact symplectic $2n$-manifold. Recall, the Liouville 1-form $\widehat{\theta}$ determines the Liouville vector field $\widehat{Z}_{\widehat{\theta}}$. We call $\widehat{M}$ a Liouville domain if it is compact with boundary and $\widehat{Z}_{\widehat{\theta}}$ is strictly outward pointing along $\partial\widehat{M}$; thus, $\partial\widehat{M}$ may be given the contact structure $\ker\widehat{\theta}\vert_{\partial\widehat{M}}$. Moreover, $\widehat{M}$ can be completed to an open exact symplectic manifold,
    \begin{equation}
    \Big(M\equiv\widehat{M}\cup_{\partial\widehat{M}}\big(\partial\widehat{M}\times[0,\infty)\big),\omega=d \theta\Big),
    \end{equation}
by gluing on the positive-half of the symplectization
    \begin{equation}
    \Big(\partial\widehat{M}\times[0,\infty),e^r\widehat{\theta}\vert_{\partial\widehat{M}}\Big),\;\; r\in[0,\infty).
    \end{equation}
A cotangent bundle of a closed smooth manifold with its natural symplectic structure is a standard example of a Liouville manifold.

Recall, a Lagrangian $L\subset M$ is exact if $\theta\vert_L=df$, where $f\in C^\infty(L)$. We will only consider two kinds of Lagrangians in $M$.
    \begin{itemize}
    \item\emph{Closed exact:} $L\subset M$ is a closed exact Lagrangian.
    \item\emph{Cylindrical at infinity:} $L\subset M$ is an open exact Lagrangian which is cylindrical at infinity and intersects $\partial\widehat{M}$ in a Legendrian $\partial L$.
    \end{itemize}
Moreover, it will be convenient, but not necessary, to assume our Lagrangians are graded. In particular, this induces a fixed absolute $\bbZ$-grading on Hamiltonian chords (or intersection points, depending on our viewpoint) given by the Maslov index $\mu$.
    
\subsection{Admissible Floer data}
Let $(L_0,L_1)$ be a pair of Lagrangians in $M$, where $L_i$ is either closed exact or cylindrical at infinity.

We work with Hamiltonians which are linear at infinity, i.e., a smooth function $H\in C^\infty(M)$ together with a non-negative constant $\tau\in\bbR_{\geq0}$ such that, at infinity, 
    \begin{equation}
    H\equiv\tau e^r.
    \end{equation}
We refer to $\tau$ as the slope of $H$. More generally, we work with time-dependent Hamiltonians $H\equiv\{H_t\}_{t\in[0,1]}$ which are linear at infinity (of a fixed slope). In the case that both $L_0$ and $L_1$ are cylindrical at infinity, we will assume the slope is not the length of a Reeb chord connecting $\partial L_0$ to $\partial L_1$.

Let $X_{H_t}$ be the (time-dependent) Hamiltonian vector field of $H$; we use the convention 
    \begin{equation}
    \omega\big(X_{H_t},\cdot\big)=-dH_t(\cdot).
    \end{equation}
We denote by $\phi^t_H$ the (time-dependent) Hamiltonian flow of $H$. Consider the set of Hamiltonian chords of $(L_0,L_1)$ (associated to $H$):
    \begin{equation}
    \chi(L_0,L_1;H)\equiv\Big\{x\in C^\infty\big([0,1];M\big):\partial_tx(t)=X_{H_t}\big(x(t)\big),x(i)\in L_i\Big\}.
    \end{equation}
Every Hamiltonian chord of $(L_0,L_1)$ corresponds to an intersection point of the Lagrangian pair $(L_0,\phi^{-1}_HL_1)$, and we say a Hamiltonian chord is non-degenerate if the corresponding intersection point is transverse. We will assume the (generic) condition that $H$ is non-degenerate, i.e., each Hamiltonian chord of $H$ is non-degenerate.

We work with ($\omega$-compatible) almost complex structures which are of contact type at infinity, i.e., an almost complex structure $J$ on $M$ such that, at infinity, 
    \begin{equation}
    J(Z_\theta)=R_{\theta\vert_{\partial\widehat{M}}},
    \end{equation}
where $R_{\theta\vert_{\partial\widehat{M}}}$ is the Reeb field of $\theta\vert_{\partial\widehat{M}}$, and $J$ stabilizes $\ker\theta\vert_{\partial\widehat{M}}$. More generally, we work with time-dependent ($\omega$-compatible) almost complex structures $J\equiv\{J_t\}_{t\in[0,1]}$ which are of contact type at infinity. Note, this class of almost complex structures is generic.  

\subsection{Lagrangian Floer cohomology}\label{subsec:floercohomology}
Let $(H,J)$ be a pair such that $H$ is a Hamiltonian which is linear at infinity and $J$ is an almost complex structure which is of contact type at infinity. We call such a pair \emph{admissible}. We may assume $H$ is non-degenerate.

Let $x,y\in\chi(L_0,L_1;H)$ and consider the following Banach manifold:
    \begin{equation}
    \scrP^{k,p}_\frakF(x,y)\equiv\Bigg\{u\in W^{k,p}_\loc(\Theta;M):\begin{aligned}
    & u\big(\bbR\times\{i\}\big)\subset L_i \\
    & \lim_{s\to-\infty}u(s,t)=x(t), \lim_{s\to+\infty}u(s,t)=y(t)
    \end{aligned}
    \Bigg\},
    \end{equation}
$kp>2$, where $(s,t)$ are coordinates on $\Theta\equiv\bbR\times[0,1]$. There is a Banach bundle $\scrE_\frakF\to\scrP^{k,p}_\frakF(x,y)$, with fiber $\scrE_\frakF\vert_u\equiv W^{k-1,p}(\Theta;u^*TM)$, together with a section 
    \begin{align}
    \overline{\partial}_{H,J}:\scrP^{k,p}_\frakF(x,y)&\to\scrE_\frakF \\
    u(s,t)&\mapsto\partial_su(s,t)+J_t\Big(\partial_tu(s,t)-X_{H_t}\big(u(s,t)\big)\Big). \nonumber
    \end{align}
Elliptic regularity shows $\widehat{\calF}^{H,J}_{L_0,L_1}(x,y)\equiv(\overline{\partial}_{H,J})^{-1}(0)$ consists of smooth solutions to the Hamiltonian perturbed pseudo-holomorphic curve equation connecting $x$ to $y$, i.e., Floer trajectories connecting $x$ to $y$. Let $\nabla$ be a connection on $TM$, then the linearization of $\overline{\partial}_{H,J}$ is given by an operator 
    \begin{equation}
    D(\overline{\partial}_{H,J})_u:T_u\scrP^{k,p}_\frakF(x,y)\to\scrE_\frakF\vert_u,
    \end{equation}
where 
    \begin{align}
    T_u\scrP^{k,p}_\frakF(x,y)&=W^{k,p}(\Theta;u^*TM,u^*TL_i) \\
    &\equiv\Big\{\xi\in W^{k,p}(\Theta;u^*TM):\xi\big(\bbR\times\{i\}\big)\subset u^*TL_i\vert_{\bbR\times\{i\}}\Big\}.
    \end{align}
Note, the linearization at a Floer trajectory is independent of $\nabla$. We say $(H,J)$ is \emph{regular} if $D(\overline{\partial}_{H,J})_u$, for any two $x,y\in\chi(L_0,L_1;H)$ and any $u\in\widehat{\calF}^{H,J}_{L_0,L_1}(x,y)$, is a surjective Fredholm operator of index $\mu(x)-\mu(y)$; this condition is generic and we will assume it for the rest of this subsection. Hence, $\widehat{\calF}^{H,J}_{L_0,L_1}(x,y)$ is transversely cut out of dimension $\mu(x)-\mu(y)$. There is a free proper $\bbR$-action on $\widehat{\calF}^{H,J}_{L_0,L_1}(x,y)$, when $x\neq y$, given by time-shift in the $s$-coordinate. The corresponding $\bbR$-quotient $\calF^{H,J}_{L_0,L_1}(x,y)$ has a compactification $\bbF^{H,J}_{L_0,L_1}(x,y)$ given by allowing breakings at Hamiltonian chords. The Lagrangian Floer cochain complex (with $\bbZ/2$-coefficients), denoted $CF^*(L_0,L_1;H,J;\bbZ/2)$, is generated in degree $m$ by Maslov index $m$ Hamiltonian chords and has codifferential 
    \begin{equation}
    y\mapsto\sum_{\mu(x)=\mu(y)+1}\big\lvert\bbF^{H,J}_{L_0,L_1}(x,y)\big\rvert_{\bbZ/2}\cdot x.\footnote{Of course, we may define Lagrangian Floer cohomology with $\bbZ$-coefficients, in the spin case, by using a system of coherent orientations, cf. \cite{FOOO09}.}
    \end{equation}
The Lagrangian Floer cohomology, denoted 
    \begin{equation}
    HF^*(L_0,L_1;H,J;\bbZ/2)\equiv H_*\big(CF^*(L_0,L_1;H,J;\bbZ/2),\delta\big),
    \end{equation}
is known to be invariant under compactly supported Hamiltonian deformations of $(L_0,L_1)$ and the choice of regular admissible pair $(H,J)$.

\subsection{Floer homotopy type}
We now assume $(L_0,L_1)$ has a framed brane structure $\Lambda$. We denote by $\bbF^{H,J}_{L_0,L_1}$ the flow category with object set $\chi(L_0,L_1;H)$ and morphism spaces $\bbF^{H,J}_{L_0,L_1}(x,y)$ (the composition is the standard gluing map of broken Floer trajectories). We will review Large's proof that $\Lambda$ induces a lift of $\bbF^{H,J}_{L_0,L_1}$ to a framed flow category $\bbF^{H,J,\Lambda}_{L_0,L_1}$; hence, we may define the Floer homotopy type
    \begin{equation}
    \frakF^{H,J,\Lambda}_{L_0,L_1}\equiv\flow^\fr\Big(\unit,\bbF^{H,J,\Lambda}_{L_0,L_1}\Big).
    \end{equation}
Again, we emphasize that, while $\frakF^{H,J,\Lambda}_{L_0,L_1}$ will be independent under compactly supported Hamiltonian deformations of $(L_0,L_1)$ and the choice of regular admissible pair $(H,J)$, it will depend on the choice of $\Lambda$.

\begin{rem}
When constructing Lagrangian Floer cohomology, the nonlinearities require us to work in Sobolev spaces satisfying $kp>2$. However, when performing the linear theory (as we are about to do in this subsection), we are allowed to work in the $W^{1,2}$-Sobolev space since the index bundles of our operators will be unaffected; we do this for convenience.  
\end{rem}

Let $R\in\bbR_{\geq0}$; we define
    \begin{equation}
    \Theta^-_R\equiv(-\infty,-R]\times[0,1]\;\;\textrm{resp.}\;\;\Theta^+_R\equiv[R,+\infty)\times[0,1].
    \end{equation}
Recall, we have fixed a sufficiently small positive constant $\lambda\in\bbR_{>0}$. We begin with two lemmata. 

\begin{lem}\label{lem:standardfloeroperator2}
Let $F\to\{0\}$ be a real vector space and $E\to[0,1]$ the complexified vector bundle $E\equiv F\otimes_\bbR\underline{\bbC}$, where we pullback $F$ to $[0,1]$. The operator 
    \begin{align}
    S:W^{1,2}([0,1];E,F)&\to L^2([0,1];E) \\
    \xi(t)&\mapsto i\partial_t\xi(t)+\lambda\xi(t) \nonumber
    \end{align}
is invertible.  
\end{lem}

\begin{proof}
Since $[0,1]$ is contractible, we may choose a unitary trivialization $\Psi:E\xrightarrow{\sim}\underline{\bbC}^{\widetilde{m}}$ together with an identification of $S$ with the operator 
    \begin{align}
    S:W^{1,2}([0,1];\underline{\bbC}^{\widetilde{m}},\Psi(F))&\to L^2([0,1];\underline{\bbC}^{\widetilde{m}}) \\
    \xi(t)&\mapsto i\partial_t\xi(t)+\lambda\xi(t). \nonumber
    \end{align}
Consider the decompositions of Hilbert spaces, 
    \begin{align}
    W^{1,2}([0,1];\underline{\bbC}^{\widetilde{m}},\Psi(F))&\cong\bigoplus_{j\in\bbZ}\Psi(F), \\
    L^2([0,1];\underline{\bbC}^{\widetilde{m}})&\cong\bigoplus_{j\in\bbZ}\Psi(F),
    \end{align}
given by taking a Fourier series in the $t$ variable, i.e., 
    \begin{equation}
    \xi(t)\mapsto\sum_{j\in\bbZ}e^{\pi ijt}\xi_j.
    \end{equation}
The aforementioned decompositions induce a splitting of $S$ into a direct sum of operators:
    \begin{equation}
    S\xi(t)=\sum_{j\in\bbZ}e^{\pi ijt}S_j\xi;
    \end{equation}
here, $S_j$ is the operator 
    \begin{align}
    S_j:\Psi(F)&\to\Psi(F) \\
    \xi_j&\mapsto(\lambda-\pi j)\xi_j.\nonumber
    \end{align}
Clearly, each $S_j$ is invertible (for sufficiently small $\lambda$); the lemma follows.
\end{proof}

\begin{lem}\label{lem:standardfloeroperator1}
Let $F\to\bbR$ be a real vector bundle with connection $\nabla$ and $E\to\Theta$ the complexified vector bundle $E\equiv F\otimes_\bbR\underline{\bbC}$ with complexified connection $\nabla^\bbC$, where we pullback $(F,\nabla)$ to $\Theta$. Assume we have fixed identifications $(E,F)\cong(\underline{\bbC}^{\widetilde{m}},\underline{\bbR}^{\widetilde{m}})$ over $\Theta^\pm_R$, which trivialize $\nabla^\bbC$ resp. $\nabla$, in order to define
    \begin{equation}
    W^{1,2}(\Theta;E,F)\equiv\Big\{\xi\in W^{1,2}(\Theta;E):\xi\big(\bbR\times\{i\}\big)\subset F\vert_{\bbR\times\{i\}}\Big\}.
    \end{equation}
The operator 
    \begin{align}
    S:W^{1,2}(\Theta;E,F)&\to L^2(\Theta;E) \\
    \xi(s,t)&\mapsto\nabla^\bbC_s\xi(s,t)+i\partial_t\xi(s,t)+\lambda\xi(s,t) \nonumber
    \end{align}
is invertible.
\end{lem}
\begin{proof}
Since $\Theta$ is contractible, we may choose a unitary trivialization $\Psi:E\xrightarrow{\sim}\underline{\bbC}^{\widetilde{m}}$, which agrees with our fixed identifications over $\Theta^\pm_R$, together with an identification of $S$ with the operator
    \begin{align}
    S:W^{1,2}(\Theta;\underline{\bbC}^{\widetilde{m}},\Psi(F))&\to L^2(\Theta;\underline{\bbC}^{\widetilde{m}}) \\
    \xi(s,t)&\mapsto\partial_s\xi(s,t)+i\partial_t\xi(s,t)+\lambda\xi(s,t). \nonumber
    \end{align}
By examining the spectral flow and using Lemma \ref{lem:standardfloeroperator2}, we see $S$ is an index 0 Fredholm operator. Consider the decompositions of Hilbert spaces 
    \begin{align}
    W^{1,2}(\Theta;\underline{\bbC}^{\widetilde{m}},\Psi(F))&\cong\bigoplus_{j\in\bbZ}W^{1,2}(\bbR;\Psi(F)), \\
    L^2(\Theta;\underline{\bbC}^{\widetilde{m}})&\cong\bigoplus_{j\in\bbZ}L^2(\bbR;\Psi(F)),
    \end{align}
given by taking a Fourier series in the $t$ variable, i.e., 
    \begin{equation}
    \xi(s,t)\mapsto\sum_{j\in\bbZ}e^{\pi ijt}\xi_j(s).
    \end{equation}
The aforementioned decompositions induce a splitting of $S$ into a direct sum of operators:
    \begin{equation}
    S\xi(s,t)=\sum_{j\in\bbZ}e^{\pi ijt}S_j\xi(s);
    \end{equation}
here, $S_j$ is the operator 
    \begin{align}
    S_j:W^{1,2}(\bbR;\Psi(F))&\to L^2(\bbR;\Psi(F)) \\
    \xi_j(s)&\mapsto\partial_s\xi_j(s)+(\lambda-\pi j)\xi_j(s). \nonumber
    \end{align}
Suppose for contradiction that there exists a non-zero $\xi_j\in\ker S_j$. We may compute:
    \begin{equation}
    \partial_s\big\langle\xi_j(s),\xi_j(s)\big\rangle=2\big\langle\xi_j(s),\partial_s\xi_j(s)\big\rangle=2\pi j\big\langle\xi_j(s),\xi_j(s)\big\rangle-2\lambda\big\langle\xi_j(s),\xi_j(s)\big\rangle,
    \end{equation}
where $\brak{\cdot,\cdot}$ is the inner product on $\Psi(F)$. It follows that, since $\lambda$ is sufficiently small and positive, $\partial_s\big\langle\xi_j(s),\xi_j(s)\big\rangle$ is strictly positive resp. negative for $j>0$ resp. $j\leq0$. But, this contradicts the assumption that $\xi_j$ is $W^{1,2}$, hence $\ker S_j=0$; the lemma follows.
\end{proof}

Consider $u\in\widehat{\calF}^{H,J}_{L_0,L_1}(x,y)$ and recall the associated operator $D(\overline{\partial}_{H,J})_u$. We may also consider the invertible self-adjoint operator
    \begin{align}
    D(\overline{\partial}_{H,J})_x:W^{1,2}([0,1];x^*TM,x^*TL_i)&\to L^2([0,1];x^*TM)\;\;\textrm{resp.} \\
    D(\overline{\partial}_{H,J})_y:W^{1,2}([0,1];y^*TM,y^*TL_i)&\to L^2([0,1];y^*TM)
    \end{align}
given by linearizing $\overline{\partial}_{H,J}$ at $x$ resp. $y$. Note, we have that $D(\overline{\partial}_{H,J})_u$ is asymptotic to $D(\overline{\partial}_{H,J})_x$ resp. $D(\overline{\partial}_{H,J})_y$ at $-\infty$ resp. $+\infty$. Let $\Phi$ be a standard invertible operator,
    \begin{align}
    W^{1,2}(\Theta;\underline{\bbC}^d,\underline{\bbR}^d)&\to L^2(\Theta;\underline{\bbC}^d) \\
    \xi(s,t)&\mapsto\partial_s\xi(s,t)+i\partial_t\xi(s,t)+\lambda\xi(s,t), \nonumber
    \end{align}
and observe that the operator 
    \begin{equation}
    D(\overline{\partial}_{H,J})_u\oplus\Phi:W^{1,2}(\Theta;u^*TM\oplus\underline{\bbC}^d,u^*TL_i\oplus\underline{\bbR}^d)\to L^2(\Theta;u^*TM\oplus\underline{\bbC}^d)
    \end{equation}
has the same kernel/cokernel as $D(\overline{\partial}_{H,J})_u$. We will abuse notation and also denote by $\Phi$ the standard invertible operator 
    \begin{align}
    W^{1,2}([0,1];\underline{\bbC}^d,\underline{\bbR}^d)&\to L^2([0,1];\underline{\bbC}^d) \\
    \xi(t)&\mapsto i\partial_t\xi(t)+\lambda\xi(t). \nonumber
    \end{align}
Again, we have that $D(\overline{\partial}_{H,J})_u\oplus\Phi$ is asymptotic to $D(\overline{\partial}_{H,J})_x\oplus\Phi$ resp. $D(\overline{\partial}_{H,J})_y\oplus\Phi$ at $-\infty$ resp. $+\infty$. Moreover, we will abuse notation and denote by $J$ the almost complex structure on $u^*TM\oplus\underline{\bbC}^d$ given as the direct sum of our chosen admissible $J$ and the standard almost complex structure on $\underline{\bbC}^d$.

\begin{defin}\label{defin:floerabstractcap}
A \emph{Floer abstract cap} for $x$ consists of the following data. 
\begin{enumerate}
\item A connection $\nabla$ on $x^*TM\oplus\underline{\bbC}^d$. Note, such a choice induces an identification
    \begin{equation}
    (x^*TM\oplus\underline{\bbC}^d,x^*TL_i\oplus\underline{\bbR}^d)\cong(\underline{\bbC}^{n+d},\underline{\bbR}^{n+d})
    \end{equation}
    together with an identification of $D(\overline{\partial}_{H,J})_x\oplus\Phi$ with the operator 
    \begin{align}
    W^{1,2}([0,1];\underline{\bbC}^{n+d},\underline{\bbR}^{n+d})&\to L^2([0,1];\underline{\bbC}^{n+d}) \\
    \xi(t)&\mapsto i\partial_t\xi(t)+\nabla\grad H_t\big(x(t)\big)\oplus\identity. \nonumber
    \end{align}
\item An operator of the form 
    \begin{align}
    T_{\frakF,x}:W^{1,2}(\Theta;\underline{\bbC}^{n+d},\underline{\bbR}^{n+d})&\to L^2(\Theta;\underline{\bbC}^{n+d}) \label{eq:form2} \\
    \xi(s,t)&\mapsto\partial_s\xi(s,t)+i\partial_t\xi(s,t)+\calT_{\frakF,x}(s,t)\xi(s,t), \nonumber
    \end{align}
where $\calT_{\frakF,x}(s,t)\in\operatorname{End}(\bbC^{n+d})$ are symmetric matrices that satisfy
    \begin{equation}\label{eqn:form2'}
    \lim_{s\to-\infty}\calT_{\frakF,x}(s,t)=\nabla\grad H_t\big(x(t)\big)\oplus\identity\;\;\textrm{and}\;\;\lim_{s\to+\infty}\calT_{\frakF,x}(s,t)=\lambda\cdot\identity.
    \end{equation}
\end{enumerate}
\end{defin}

\begin{rem}
Let $\calT^*_{\frakF,x}(s,t)\in\operatorname{End}(\bbC^{n+d})$ satisfy
    \begin{equation}
    \calT^*_{\frakF,x}(s,t)\equiv\calT_{\frakF,x}(-s,t).
    \end{equation}
A straightforward operator gluing argument shows the operator
    \begin{align}
    T^*_{\frakF,x}:W^{1,2}(\Theta;\underline{\bbC}^{n+d},\underline{\bbR}^{n+d})&\to L^2(\Theta;\underline{\bbC}^{n+d}) \\
    \xi(s,t)&\mapsto\partial_s\xi(s,t)+i\partial_t\xi(s,t)+\calT^*_{\frakF,x}(s,t)\xi(s,t) \nonumber
    \end{align}
has index bundle satisfying 
    \begin{equation}\label{eqn:dualoperator2}
    \ind T^*_{\frakF,x}\cong-\ind T_{\frakF,x}.
    \end{equation}
\end{rem}

By examining the spectral flow, we see any Floer abstract cap for $x$ is a Fredholm operator of index $\mu(x)$. Moreover, we see the space of Floer abstract caps for $x$ is contractible since (1) the space of connections is contractible and (2) the space of operators of the form \eqref{eq:form2} with fixed asymptotics is contractible (i.e., the space of families of symmetric matrices with asymptotics \eqref{eqn:form2'} is contractible).

Consider the glued together Fredholm operator
    \begin{equation}
    T_{\frakF,x,y,u}\equiv T^*_{\frakF,x}\# D(\overline{\partial}_{H,J})_u\oplus\Phi\# T_{\frakF,y}.
    \end{equation}
Under standard operator gluing, the index bundle is additive:
    \begin{equation}\label{eq:floergluing}
    \ind T_{\frakF,x,y,u}\cong\ind T^*_{\frakF,x}+\ind\big(D(\overline{\partial}_{H,J})_u\oplus\Phi\big)+\ind T_{\frakF,y}.
    \end{equation}

We claim $T_{\frakF,x,y,u}$ can be deformed to an invertible operator in a way that only depends on $\Lambda$. Using the identification
    \begin{equation}
    TM\oplus\underline{\bbC}^d\cong\Lambda\otimes_\bbR\underline{\bbC}
    \end{equation}
together with the homotopy between $\Lambda\vert_{L_i}$ and $TL_i\oplus\underline{\bbR}^d$ through totally real subbundles, we may deform the boundary conditions of $T_{\frakF,x,y,u}$:
    \begin{equation}
    T_{\frakF,x,y,u}:W^{1,2}(\Theta;u^*(\Lambda\otimes_\bbR\underline{\bbC}),u^*\Lambda\vert_{\bbR\times\{i\}})\to L^2(\Theta;u^*(\Lambda\otimes_\bbR\underline{\bbC})).
    \end{equation}
Note, deforming the boundary conditions does not affect the index bundle. By choosing a connection $\nabla$ on $u^*\Lambda\vert_{\bbR\times\{0\}}$, complexifying it to a connection $\nabla^\bbC$ on $u^*(\Lambda\otimes_\bbR\underline{\bbC})\vert_{\bbR\times\{0\}}$, and parallel transporting a suitable frame in the $t$-direction; we have an identification of $T_{\frakF,x,y,u}$ with an operator of the form
    \begin{align}
    W^{1,2}(\Theta;u^*(\Lambda\otimes_\bbR\underline{\bbC}),u^*\Lambda\vert_{\bbR\times\{0\}})&\to L^2(\Theta;u^*(\Lambda\otimes_\bbR\underline{\bbC})) \label{eq:form3} \\ 
    \xi(s,t)&\mapsto\nabla^\bbC_s\xi(s,t)+i\partial_t\xi(s,t)+A(s,t)\xi(s,t), \nonumber
    \end{align}
where $A(s,t)\in\operatorname{End}\big(u^*(\Lambda\otimes_\bbR\underline{\bbC})\vert_{(s,t)}\big)$ are symmetric matrices that satisfy
    \begin{equation}
    \lim_{s\to\pm\infty}A(s,t)=\lambda\cdot\identity.
    \end{equation}
Finally, since the space of operators of the form \eqref{eq:form3} with fixed asymptotics is contractible, we may identify $T_{\frakF,x,y,u}$ with the operator
    \begin{align}
    W^{1,2}(\Theta;u^*(\Lambda\otimes_\bbR\underline{\bbC}),u^*\Lambda\vert_{\bbR\times\{0\}})&\to L^2(\Theta;u^*(\Lambda\otimes_\bbR\underline{\bbC})) \\ 
    \xi(s,t)&\mapsto\nabla^\bbC_s\xi(s,t)+i\partial_t\xi(s,t)+\lambda\cdot\identity; \nonumber
    \end{align}
the claim follows by Lemma \ref{lem:standardfloeroperator1}.

Now, we have an equivalence of real virtual vector spaces:
    \begin{equation}\label{eq:discussion2}
    \ind T_{\frakF,x}\cong\ind\big(D(\overline{\partial}_{H,J})_u\oplus\Phi\big)\oplus\ind T_{\frakF,y}.
    \end{equation}
Since all choices live in contractible spaces (except for the fixed global choice of $\Lambda$), we can upgrade this equivalence to families of operators. In particular, consider the fiber bundle $\fred\to\widehat{\calF}^{H,J}_{L_0,L_1}(x,y)$, with fiber 
    \begin{equation}
    \fred\vert_u\equiv\fred\big(W^{1,2}(\Theta;u^*TM\oplus\underline{\bbC}^d,u^*TL_i\oplus\underline{\bbR}^d),L^2(\Theta;u^*TM\oplus\underline{\bbC}^d)\big),
    \end{equation}
together with the section 
    \begin{align}
    D\overline{\partial}_{H,J}:\widehat{\calF}^{H,J}_{L_0,L_1}(x,y)&\to\fred \\
    u&\mapsto D(\overline{\partial}_{H,J})_u\oplus\Phi. \nonumber
    \end{align}
By the discussion in Subsection \ref{subsec:floercohomology}, we have an isomorphism of real vector bundles 
    \begin{equation}
    T\widehat{\calF}^{H,J}_{L_0,L_1}(x,y)\cong\ind D\overline{\partial}_{H,J}.
    \end{equation}
Saying \eqref{eq:discussion2} extends to families is equivalent to saying we have constructed the following equivalence of real virtual bundles:
    \begin{equation}
    T\widehat{\calF}^{H,J}_{L_0,L_1}(x,y)\oplus\underline{\bbR}^{\mu(y)}\cong\underline{\bbR}^{\mu(x)},
    \end{equation}
where we have used the canonical isomorphisms
    \begin{equation}
    \ind T_{\frakF,x}\cong\underline{\bbR}^{\mu(x)}\;\;\textrm{and}\;\;\ind T_{\frakF,y}\cong\underline{\bbR}^{\mu(y)}.
    \end{equation}
Finally, by using the isomorphism of real vector bundles
    \begin{equation}
    T\calF^{H,J}_{L_0,L_1}(x,y)\oplus\underline{\bbR}\cong T\widehat{\calF}^{H,J}_{L_0,L_1}(x,y)
    \end{equation}
and the homotopy equivalence 
    \begin{equation}
    \bbF^{H,J}_{L_0,L_1}(x,y)\simeq\calF^{H,J}_{L_0,L_1}(x,y),
    \end{equation}
we have the following equivalence of real virtual bundles:
    \begin{equation}\label{eq:floerframing}
    T\bbF^{H,J}_{L_0,L_1}(x,y)\oplus\underline{\bbR}^{\mu(y)}\oplus\underline{\bbR}\cong\underline{\bbR}^{\mu(x)}.
    \end{equation}

It remains to show \eqref{eq:floerframing} is compatible with gluing of broken Floer trajectories, but this follows from (1) operator gluing and (2) \eqref{eqn:dualoperator2}. Thus, we have a lift of $\bbF^{H,J}_{L_0,L_1}$ to a framed flow category $\bbF^{H,J,\Lambda}_{L_0,L_1}$.

\section{From Floer to Morse}
\subsection{Original correspondence}
Let $X$ be a (not necessarily closed) smooth $n$-manifold, $f\in C^\infty(X)$ a smooth function, and $g$ a Riemannian metric on $X$. We may assume $(f,g)$ is Morse-Smale; moreover, we will assume the maximal invariant subset of $f$ is isolated. We denote by $\pi$ the projection 
    \begin{equation}
    T^*X\to X;
    \end{equation}
moreover, we denote by $\calO_X$ the image of the 0-section embedding $X\hookrightarrow T^*X$. Let $\omega_\std$ be the standard symplectic structure on $T^*X$. We will take the classical approach of Floer \cite{Flo89c} to define the Lagrangian Floer cohomology of $(\calO_X,\calO_X)$.

Consider the Hamiltonian 
    \begin{equation}
    H^f\equiv -f\circ\pi\in C^\infty(T^*X).
    \end{equation}
Observe, the flow of $X_{H^f}$ is tangential to the fibers. In particular, we have a 1-1 correspondence 
    \begin{equation}
    \chi(\calO_X,\calO_X;H^f)\cong\crit(f).
    \end{equation}
There is a natural choice of ($\omega_\std$-compatible) almost complex structure $J^g$ on $T^*X$; this is determined by the splitting
    \begin{equation}\label{eq:splitting}
    T(T^*X)\cong\pi^*TX\otimes_{\bbR}\underline{\bbC}
    \end{equation}
induced by $\nabla^g$.

\begin{thm}[Theorem 1 \cite{Flo89c}]
Suppose $f$ is sufficiently $C^2$-small, then, for any two $x,y\in\crit(f)$, the map
    \begin{align}
    \widehat{\calF}^{H^f,J^g}_{\calO_X,\calO_X}(x,y)&\to\widehat{\calM}^{f,g}_X(x,y) \label{eq:bijection1}\\
    u(s,t)&\mapsto\gamma(s)\equiv u(s,0) \nonumber
    \end{align}
is a bijection.
\end{thm}

Floer's result is enough to define the Lagrangian Floer cohomology of $(\calO_X,\calO_X)$ and identify it with the Morse homology of $f$, as follows. We see that \eqref{eq:bijection1} descends to the $\bbR$-quotients; moreover, 
this bijection on the $\bbR$-quotients will straightforwardly extend over the corresponding compactifications:
    \begin{equation}\label{eq:bijection2}
    \bbF^{H^f,J^g}_{\calO_X,\calO_X}(x,y)\cong\bbM^{f,g}_X(x,y).
    \end{equation}
In the following subsection, we will upgrade this to an equivalence between the Floer homotopy type and the Morse homotopy type.

\subsection{Upgraded correspondence}
First, \eqref{eq:splitting} gives a natural choice of framed brane structure for $(\calO_X,\calO_X)$; namely, let $\Lambda_\std$ be defined as the totally real subbundle $\pi^*TX$ and choose the constant homotopy between $\Lambda_\std\vert_{\calO_X}$ and $T\calO_X$. We denote by $\nabla$ the complexification of $\pi^*\nabla^g$ on $T(T^*X)$.

\emph{A priori}, we do not know that $(H^f,J^g)$ is regular. This follows from the following general lemma.

\begin{lem}\label{lem:standardfloermorseoperator1}
Let $F\to\bbR$ be a real vector bundle with connection $\nabla$ and $E\to\Theta$ the complexified vector bundle $E\equiv F\otimes_\bbR\underline{\bbC}$ with complexified connection $\nabla^\bbC$, where we pullback $(F,\nabla)$ to $\Theta$. Assume we have fixed identifications $(E,F)\cong(\underline{\bbC}^{\widetilde{m}},\underline{\bbR}^{\widetilde{m}})$ over $\Theta^\pm_R$, which trivialize $\nabla^\bbC$ resp. $\nabla$, in order to define the operator
    \begin{align}
    S_\Theta:W^{1,2}(\Theta;E,F)&\to L^2(\Theta;E) \\
    \xi(s,t)&\mapsto\nabla^\bbC_s\xi(s,t)+i\partial_t\xi(s,t)+A(s)\xi(s,t), \nonumber
    \end{align}
where $A(s)\in\operatorname{End}(E\vert_{(s,t)})$ are symmetric matrices of sufficiently small norm such that (1) the asymptotics are invertible symmetric matrices of sufficiently small norm and (2) $A(s)(F\vert_{(s,0)})\subset F\vert_{(s,0)}$. Suppose $S_\bbR$ is the operator
    \begin{align}
    S_\bbR:W^{1,2}(\bbR;F)&\to L^2(\bbR;F) \\
    \xi(s)&\mapsto\nabla_s\xi(s)+A(s)\vert_{F\vert_{(s,0)}}\xi(s),\nonumber
    \end{align}
then we have an isomorphism of real vector spaces 
    \begin{equation}
    \ker S_\theta\cong\ker S_\bbR\;\;\textrm{and}\;\;\coker S_\theta\cong\coker S_\bbR.
    \end{equation}
\end{lem}

\begin{proof}
We use an approach analogous to the proof of Lemma \ref{lem:standardfloeroperator1}. Since $\Theta$ is contractible, we may choose a unitary trivialization $\Psi:E\xrightarrow{\sim}\underline{\bbC}^{\widetilde{m}}$, which agrees with our fixed identifications over $\Theta^\pm_R$, together with an identification of $S_\Theta$ with the operator
    \begin{align}
    S_\Theta:W^{1,2}(\Theta;\underline{\bbC}^{\widetilde{m}},\Psi(F))&\to L^2(\Theta;\underline{\bbC}^{\widetilde{m}}) \\
    \xi(s,t)&\mapsto\partial_s\xi(s,t)+i\partial_t\xi(s,t)+A(s)\xi(s,t). \nonumber
    \end{align}
Consider the decompositions of Hilbert spaces 
    \begin{align}
    W^{1,2}(\Theta;\underline{\bbC}^{\widetilde{m}},\Psi(F))&\cong\bigoplus_{j\in\bbZ}W^{1,2}(\bbR;\Psi(F)), \\
    L^2(\Theta;\underline{\bbC}^{\widetilde{m}})&\cong\bigoplus_{j\in\bbZ}L^2(\bbR;\Psi(F)),
    \end{align}
given by taking a Fourier series in the $t$ variable, i.e., 
    \begin{equation}
    \xi(s,t)\mapsto\sum_{j\in\bbZ}e^{\pi ijt}\xi_j(s).
    \end{equation}
The aforementioned decompositions induce a splitting of $S_\Theta$ into a direct sum of operators:
    \begin{equation}
    S_\Theta\xi(s,t)=\sum_{j\in\bbZ}e^{\pi ijt}S_j\xi(s);
    \end{equation}
here, $S_j$ is the operator 
    \begin{align}
    S_j:W^{1,2}(\bbR;\Psi(F))&\to L^2(\bbR;\Psi(F)) \\
    \xi_j(s)&\mapsto\partial_s\xi_j(s)+\big(A(s)-\pi j\big)\xi_j(s). \nonumber
    \end{align}
Observe, $S_0=S_\bbR$. Suppose for contradiction that there exists a non-zero $\xi_j\in\ker S_j$, where $j\neq0$. We may compute:
    \begin{equation}
    \partial_s\big\langle\xi_j(s),\xi_j(s)\big\rangle=2\big\langle\xi_j(s),\partial_s\xi_j(s)\big\rangle=2\pi j\big\langle\xi_j(s),\xi_j(s)\big\rangle-2\big\langle\xi_j(s),A(s)\xi_j(s)\big\rangle,
    \end{equation}
where $\brak{\cdot,\cdot}$ is the inner product on $\Psi(F)$. We have that 
    \begin{equation}
    \big\lvert\big\langle\xi_j(s),A(s)\xi_j(s)\big\rangle\big\rvert\leq\lvert\lvert A(s)\lvert\lvert\cdot\rvert\rvert\xi_j(s)\rvert\rvert^2.
    \end{equation}
It follows that, since $A(s)$ has sufficiently small norm, $\partial_s\big\langle\xi_j(s),\xi_j(s)\big\rangle$ is strictly positive resp. negative for $j>0$ resp. $j<0$. But, this contradicts the assumption that $\xi_j$ is $W^{1,2}$, hence $\ker S_j=0$ for $j\neq0$; this shows 
    \begin{equation}
    \ker S_\theta\cong\ker S_\bbR.
    \end{equation}
In order to prove
    \begin{equation}
    \coker S_\theta\cong\coker S_\bbR,
    \end{equation}
we simply repeat the argument just performed for the formal adjoint of $S_\Theta$,
    \begin{align}
    S^*_\Theta:W^{1,2}(\Theta;E,F)&\to L^2(\Theta;E) \\
    \xi(s,t)&\mapsto-\nabla^\bbC_s\xi(s,t)+i\partial_t\xi(s,t)+A(s)\xi(s,t), \nonumber
    \end{align}
which satisfies 
    \begin{equation}
    \ker S^*_\Theta=\coker S_\Theta.
    \end{equation}
This proves the lemma.
\end{proof}

\begin{cor}\label{cor:regular}
We have an isomorphism of real vector spaces
    \begin{equation}
    \ker D(\overline{\partial}_{H^f,J^g})_u\cong\ker D(\Xi_{f,g})_\gamma\;\;\textrm{and}\;\;\coker D(\overline{\partial}_{H^f,J^g})_u\cong\coker D(\Xi_{f,g})_\gamma,
    \end{equation}
where $\gamma(s)\equiv u(s,0)$. In particular, $(H^f,J^g)$ is regular if and only if $(f,g)$ is Morse-Smale.
\end{cor}

\cite[Theorem 6.12]{Lar21} shows, in the Liouville manifold case, that the standard gluing map of broken Floer trajectories endows the compactified Floer moduli spaces with the structure of a smooth manifold with corners. The proof of this result only relies on exactness of the pair of Lagrangians and regularity of the Floer data; hence, it is seen to extend to the case considered here (namely, when $X$ is open). Moreover, it was shown by Large (cf. \cite[Proposition 7.3, Proposition 7.4]{Lar21}) that the various index bundles given by linearizing the Floer equation at various Floer trajectories, and the gluings of these index bundles over the various strata of the compactified Floer moduli spaces, determine the same real vector bundles as the tangent bundles of the compactified Floer moduli spaces given by the smooth manifold with corners structures. Again, the proof of this result is seen to extend to the case considered here. Corollary \ref{cor:regular} shows \eqref{eq:bijection1} identifies the index bundles of the corresponding linearized operators:
    \begin{equation}
    \ind D(\overline{\partial}_{H^f,J^g})_u\cong\ind D(\Xi_{f,g})_\gamma,\;\; \gamma(s)=u(s,0);
    \end{equation}
moreover, this identification is compatible with the $\bbR$-actions and gluing maps. The upshot is the following result. 

\begin{cor}\label{cor:stablediffeo}
\eqref{eq:bijection2} determines a stable diffeomorphism:
    \begin{equation}
    \bbR^N\times\bbF^{H^f,J^g}_{\calO_X,\calO_X}(x,y)\cong\bbR^N\times\bbM^{f,g}_X(x,y),\;\;N\gg0.
    \end{equation}
\end{cor}

\begin{proof}
We already have that \eqref{eq:bijection2} determines a homeomorphism of compact topological manifolds with corners. Since the tangent microbundle of $\bbF^{H^f,J^g}_{\calO_X,\calO_X}(x,y)$ admits a vector bundle lift via $\ind D\overline{\partial}_{H^f,J^g}$ and the tangent microbundle of $\bbM^{f,g}_X(x,y)$ admits a vector bundle lift via $\ind D\Xi_{f,g}$, and we have just seen how the vector bundle lifts are compatible via directly comparing the Floer and Morse linearized operators, this corollary follows by general smoothing theory.
\end{proof}

\begin{rem}
The reason for appealing to general smoothing theory is so that we do not have to directly compare the construction of charts for smooth structures on the Floer and Morse side. Moreover, since Abouzaid-Blumberg's framework (1) already incorporates flow categories whose morphism spaces admit global Kuranishi chart presentations and (2) declares global Kuranishi chart presentations to be equivalent if one can be obtained from the other via stabilization, the stable diffeomorphism of Corollary \ref{cor:stablediffeo} is already good enough to identify the unstructured Floer and Morse categories; cf. also \cite[Lemma 4.3]{PS25a}.
\end{rem}

It remains to show the Floer framing (i.e., \eqref{eq:floerframing} determined by $\Lambda_\std$) and the Morse framing (i.e., \eqref{eq:morseframing}) are identified via \eqref{eq:bijection2}. Recall, $\nabla$ is the complexification of $\pi^*\nabla^g$. A Floer abstract cap for $x$ (in this situation) is an operator of the form 
    \begin{align}
    T_{\frakF,x}:W^{1,2}(\Theta;\underline{\bbC}^n,\underline{\bbR}^n)&\to L^2(\Theta;\underline{\bbC}^n) \\
	\xi(s,t)&\mapsto\partial_s\xi(s,t)+i\partial_t\xi(s,t)+\calT_{\frakF,x}(s,t)\xi(s,t),\nonumber
    \end{align}
where we use the identification $(x^*T(T^*X),x^*T\calO_X)\cong(\underline{\bbC}^n,\underline{\bbR}^n)$ induced by $\nabla$ and $\calT_{\frakF,x}(s,t)\in\operatorname{End}(\bbC^n)$ are symmetric matrices that satisfy
    \begin{equation}
    \lim_{s\to-\infty}\calT_{\frakF,x}(s)=\nabla\grad H^f\big(x(t)\big)\;\;\textrm{and}\;\;\lim_{s\to+\infty}\calT_{\frakF,x}(s)=\lambda\cdot\identity.
    \end{equation}
We may choose $\calT_{\frakF,x}(s,t)$ so that (1) it is $t$-independent, i.e., $\calT_{\frakF,x}(s,t)=\calT_{\frakF,x}(s,t')$ for all $t,t'\in[0,1]$, and (2) $\calT_{\frakM,x}(s)\equiv\calT_{\frakF,x}(s,0)\vert_{\bbR^n}$ gives a Morse abstract cap for $x$. In this way, choosing a Morse abstract cap for $x$ determines a Floer abstract cap for $x$. Combining this observation with Lemma \ref{lem:standardfloermorseoperator1} gives the following result.

\begin{cor}
We have an isomorphism of real vector spaces: 
    \begin{equation}
    \ker T_{\frakF,x}\cong\ker T_{\frakM,x}\;\;\textrm{and}\;\;\coker T_{\frakF,x}\cong\coker T_{\frakM,x}.
    \end{equation}
\end{cor}

\begin{cor}
\eqref{eq:bijection2} determines an identification between the Floer framing (determined by $\Lambda_\std$) and the Morse framing.
\end{cor}

The culmination of this section is the following result.

\begin{prop}\label{prop:floermorsecorrespondence}
We have the following identification of framed flow categories: 
    \begin{equation}
    \bbF^{H^f,J^g,\Lambda_\std}_{\calO_X,\calO_X}\cong\bbM^{f,g}_X.
    \end{equation}
\end{prop}

\section{Steenrod squares on Lagrangian Floer cohomology}
\subsection{Local Floer theory}\label{subsec:localfloer}
Let $(M,\omega=d\theta)$ be a Liouville manifold, $(L_0,L_1')$ a pair of Lagrangians, where $L_i$ is either closed exact or cylindrical at infinity, and $(H,J')$ an admissible pair. Consider the Lagrangian 
    \begin{equation}
    L_1\equiv\phi^{-1}_H(L_1').
    \end{equation}
We have an identification
    \begin{equation}
    \chi(L_0,L_1)\cong\chi(L_0,L_1';H),
    \end{equation}
where $\chi(L_0,L_1)$ is the set of intersection points of $(L_0,L_1)$. Let $J\equiv\{J_t\}_{t\in[0,1]}$ be the almost complex structure defined by
    \begin{equation}
    J_t\equiv(\phi^{-t}_H)_*J_t'.
    \end{equation}
We may assume $(H,J')$ is regular for $(L_0,L_1')$, or equivalently, $(L_0,L_1)$ intersects transversely and $J$ is regular for $(L_0,L_1)$. We have an identification of Floer moduli spaces:
    \begin{equation}
    \widehat{\calF}^J_{L_0,L_1}(x,y)\cong\widehat{\calF}^{H,J'}_{L_0,L_1'}(x',y'),
    \end{equation}
where $x,y\in\chi(L_0,L_1)$ correspond to $x',y'\in\chi(L_0,L_1;H)$, respectively. The upshot is that we have an isomorphism of Lagrangian Floer cohomology: 
    \begin{equation}
    HF(L_0,L_1;J;\bbZ/2)\cong HF(L_0,L_1';H,J';\bbZ/2).
    \end{equation}
Finally, if we assume in addition that $\Lambda'$ is a framed brane structure for $(L_0,L_1')$, then $(L_0,L_1)$ has a framed brane structure $\Lambda$ determined in a natural way from $\Lambda'$ and $\phi^{-1}_H$. The upshot is that we have an equivalence of Floer homotopy types:
    \begin{equation}
    \frakF^{J,\Lambda}_{L_0,L_1}\cong\frakF^{H,J',\Lambda'}_{L_0,L_1'}.
    \end{equation}

\begin{rem}
We have just described the equivalence between the non-Hamiltonian perturbed framework and the Hamiltonian perturbed framework for Lagrangian Floer theory. The reason we are belaboring this point is because the constructions in this section are more easily described in the non-Hamiltonian perturbed framework. 
\end{rem}

Now, assume $L_0$ intersects $L_1$ only in a single (possibly degenerate) point $x$. Let $U\subset M$ be a Weinstein neighborhood of $L_0$, i.e., $U$ is symplectomorphic to a neighborhood $V\subset T^*L_0$ of $\calO_{L_0}$. Consider a ball $B_r(x)\subset U$ of sufficiently small radius $r\in\bbR_{>0}$ centered at $x$. We define 
    \begin{equation}
    \widetilde{L}_i\equiv L_i\cap B_r(x).
    \end{equation}
We may choose local coordinates so that there exists a sufficiently $C^2$-small smooth function $f\in C^\infty(\widetilde{L}_0)$ together with an identification of $\widetilde{L}_1$ and
    \begin{equation}
    \big\{(p,df_p)\big\}\subset T^*\widetilde{L}_0.\footnote{The existence of such a smooth function can be shown as follows. We use local coordinates, mapping $x$ to the origin, and a ``partial Morse lemma'' to write $L_1$ locally as $\{(p,df_p)\}$, where $f(x_1,\ldots,x_n)=x^2_1+\cdots+x^2_{n-m}+\widetilde{f}(x_{n-m+1},\ldots,x_n)$, for some $m$ with $1\leq m\leq n$, such that $\widetilde{f}$ has vanishing Hessian at the origin. We now change coordinates so that the quadratic part of $f$ is $C^2$-small.}
    \end{equation}
In particular, $f$ only has a single (possibly degenerate) critical point at $x$ and $S_f=\{x\}$. We will show we can investigate the Lagrangian Floer theory of $(L_0,L_1)$ ``locally''; this sort of idea, usually termed local Floer theory, is not new, cf. \cite{Alb06,Auy23,Flo89b,GG10,Oh96}.

Let $g$ be a Riemannian metric on $L_0$; we abuse notation and also denote by $g$ the restriction of this Riemannian metric to $\widetilde{L}_0$. Recall, we have a natural $\omega_\std$-compatible almost complex structure $J^g$ on $T^*\widetilde{L}_0$ and Hamiltonian $H^f\in C^\infty(T^*\widetilde{L}_0)$. We define
    \begin{equation}
    \widetilde{J}^g_t\equiv(\phi^{-t}_{H^f})_*J^g,\;\;\widetilde{J}^g\equiv\{\widetilde{J}^g_t\}.
    \end{equation}
Perhaps after shrinking $B_r(x)$, we may perturb $J$ so that 
    \begin{equation}
    J_t\vert_{B_r(x)}=\widetilde{J}^g_t\vert_{B_r(x)}.
    \end{equation}
Let $\{r_\nu\}_{\nu\in\bbZ_{\geq0}}\subset(0,r)$ be a strictly decreasing sequence of real numbers converging to 0. We may perturb $(f,g)$, via a $C^2$-small perturbation compactly supported in $L_0\cap B_{r_\nu}(x)$, into a Morse-Smale pair $(f_\nu,g_\nu)$ such that 
\begin{enumerate}
\item $(f_\nu,g_\nu)$ agrees with $(f,g)$ outside of $L_0\cap B_{r_\nu}(x)$,
\item and $\crit(f_\nu)\subset L_0\cap B_{r_\nu}(x)$.
\end{enumerate}
The perturbation of $f$ into $f_\nu$ induces a compactly supported Hamiltonian deformation of $L_1$ into a Lagrangian $L^\nu\subset M$ such that 
    \begin{equation}
    \widetilde{L}^\nu\equiv L^\nu\cap B_r(x)
    \end{equation}
is identified with 
    \begin{equation}
    \Big\{\big(p,d(f_\nu)_p\big)\Big\}\subset T^*\widetilde{L}_0.
    \end{equation}
Meanwhile, the perturbation of $g$ into $g_\nu$ induces a perturbation of $J$ into an almost complex structure $J^\nu\equiv\{J^\nu_t\}$ on $M$ which agrees with $J$ outside of $B_{r_\nu}(x)$ and satisfies 
    \begin{equation}
    J^\nu_t\vert_{B_r(x)}=\widetilde{J}^{g_\nu}_t\vert_{B_r(x)}.
    \end{equation}

For any two $x^\nu,y^\nu\in\chi(L_0,L^\nu)$, we have a natural inclusion: 
    \begin{equation}\label{eq:aux}
    \widehat{\calF}^{\widetilde{J}^{g_\nu}}_{\widetilde{L}_0,\widetilde{L}^\nu}(x^\nu,y^\nu)\hookrightarrow\widehat{\calF}^{J^\nu}_{L_0,L^\nu}(x^\nu,y^\nu).
    \end{equation}

\begin{lem}\label{lem:localfloer}
For $\nu\gg0$, \eqref{eq:aux} is a bijection.
\end{lem}

\begin{proof}
Suppose for contradiction the conclusion does not hold, then we may construct a sequence $\{u_\nu\}$ of $J^\nu$-holomorphic strips with Lagrangian boundary conditions on the pair $(L_0,L^\nu)$ connecting $x^\nu$ to $y^\nu$ together with a sequence $\{(s_\nu,t_\nu)\}\subset\Theta$ such that $u(s_\nu,t_\nu)\notin B_r(x)$ for $\nu\gg0$. From the \emph{a priori} energy bound on pseudo-holomorphic strips given by the action, it follows that $E(u_\nu)\to0$ as $\nu\to\infty$, where $E(u_\nu)$ is the energy of $u_\nu$. Therefore, by Gromov compactness, we may assume $\{u_\nu\}$ converges uniformly on compact subsets of $\Theta$ to a $J$-holomorphic strip $u$ with Lagrangian boundary conditions on $(L_0,L_1)$ connecting $x$ to itself. Since $E(u)=0$, $u$ is constant. From the $\bbR$-action given by time-shift in the $s$-coordinate, we may assume $s_\nu=0$ for all $\nu$. Moreover, by passing to a subsequence, we may assume $t_\nu\to t$ as $\nu\to\infty$. By the uniform convergence on compact subsets of $\Theta$, we see
    \begin{equation}
    u_\nu(s_\nu,t_\nu)=u_\nu(0,t_\nu)\to u(0,t)=x 
    \end{equation}
as $\nu\to\infty$; this is the desired contradiction. 
\end{proof}

We know $\widetilde{J}^{g_\nu}$ is regular for $(\widetilde{L}_0,\widetilde{L}^\nu)$. Thus, we have the following result.

\begin{cor}
For $\nu\gg0$ and any two $x^\nu,y^\nu\in\chi(L_0,L^\nu)$, we have a diffeomorphism of compact stratified smooth manifolds with corners:
    \begin{equation}
    \bbF^{\widetilde{J}^{g_\nu}}_{\widetilde{L}_0,\widetilde{L}^\nu}(x^\nu,y^\nu)\cong\bbF^{J^\nu}_{L_0,L^\nu}(x^\nu,y^\nu).
    \end{equation}
Moreover, we have an identification of unstructured flow categories: 
    \begin{equation}
    \bbF^{\widetilde{J}^{g_\nu}}_{\widetilde{L}_0,\widetilde{L}^\nu}\cong\bbF^{J^\nu}_{L_0,L^\nu}.
    \end{equation}
\end{cor}

Finally, suppose in addition that $(L_0,L_1)$ has a framed brane structure $\Lambda$, then $(\widetilde{L}_0,\widetilde{L}_1)$ has a framed brane structure $\widetilde{\Lambda}$ determined in a natural way from $\Lambda$ and restriction. Moreover, our Hamiltonian deformation of $L_1$ into $L^\nu$ resp. $\widetilde{L}_1$ into $\widetilde{L}^\nu$ induces a framed brane structure $\Lambda^\nu$ resp. $\widetilde{\Lambda}^\nu$, which comes from $\Lambda$, on $(L_0,L^\nu)$ resp. $(\widetilde{L}_0,\widetilde{L}^\nu)$ such that $\Lambda^\nu$ restricts to $\widetilde{\Lambda}^\nu$. The upshot is the following result.

\begin{cor}
For $\nu\gg0$, we have an identification of Floer homotopy types: 
    \begin{equation}
    \frakF^{\widetilde{J}^g,\widetilde{\Lambda}}_{\widetilde{L}_0,\widetilde{L}_1}\equiv\frakF^{\widetilde{J}^{g_\nu},\widetilde{\Lambda}^\nu}_{\widetilde{L}_0,\widetilde{L}^\nu}\simeq\frakF^{J^\nu,\Lambda^\nu}_{L_0,L^\nu}\equiv\frakF^{J,\Lambda}_{L_0,L_1}.
    \end{equation}
\end{cor}

\subsection{Proof of main results}
We are now ready to prove our main results. We recall the statement of Theorem \ref{thm:main1} for convenience. 

\begin{thm}[Theorem \ref{thm:main1}]
Suppose Assumption \ref{assu:main} and that $L_0$ intersects $L_1$ in finitely many (possibly degenerate) points, denoted $\{x_{s_j}\}_{1\leq j\leq k,1\leq s_j\leq\ell_j}$, such that the points of fixed index $j$ (i.e., $\{x_{s_j}\}_{1\leq s_j\leq\ell_j}$) are of a fixed action $\kappa_j\in\bbR$ which is strictly increasing (i.e., $\kappa_j<\kappa_{j+1}$ $\forall j$). For each $j$, we have that there exists integers $\{d_{s_j}\}_{1\leq s_j\leq\ell_j}\subset\bbZ$ and a cofiber sequence 
    \begin{equation}
    \frakF^{J,\Lambda,\kappa_{j-1}}_{L_0,L_1}\to\frakF^{J,\Lambda,\kappa_j}_{L_0,L_1}\to\bigvee_{s_j}\Sigma^{d_{s_j}}\Sigma^\infty\calC_{f_{s_j}},
    \end{equation}
where each $f_{s_j}$ is a smooth function on $\bbR^n$ with a single (possibly degenerate) critical point at the origin.
\end{thm}

\begin{proof}[Proof of Theorem \ref{thm:main1}]
As in the case of a single intersection, we may choose, for each $x_{s_j}$, a sufficiently small ball $B_r(x_{s_j})$. We define
    \begin{equation}
    \widetilde{L}^{s_j}_i\equiv L_i\cap B_r(x_{s_j}).
    \end{equation}
Also, we may choose a sufficiently $C^2$-small smooth function $f\in C^\infty(\widetilde{L}^{s_j}_0)$ together with an identification of $\widetilde{L}^{s_j}_1$ and
    \begin{equation}
    \Big\{\big(p,d(f_{s_j})_p\big)\Big\}\subset T^*\widetilde{L}^{s_j}_0.
    \end{equation}
In particular, $f_{s_j}$ only has a single (possibly degenerate) critical point at $x_{s_j}$ and $S_{f_{s_j}}=\{x_{s_j}\}$. Moreover, by choosing $r$ sufficiently small, we may identify $T^*\widetilde{L}^{s_j}_0$ with $T^*\bbR^n$ and $\widetilde{L}^{s_j}_0$ with $\calO_{\bbR^n}$; hence, $f_{s_j}$ may be thought of as a smooth function on $\bbR^n$ with only a single (possibly degenerate) critical point at the origin. Finally, since $T^*\bbR^n$ and $\calO_{\bbR^n}$ are contractible, any two framed brane structures for $(\widetilde{L}^{s_j}_0,\widetilde{L}^{s_j}_1)$ are equivalent after stabilization. The upshot is that the framed brane structure $\widetilde{\Lambda}^{s_j}$ for $(\widetilde{L}^{s_j}_0,\widetilde{L}^{s_j}_1)$ induced by $\Lambda$ is equivalent to $\Lambda_\std$ after a $d_{s_j}$-fold stabilization, where $d_{s_j}\in\bbZ$.

Let $g$ be a Riemannian metric on $L_0$; we abuse notation and also denote by $g$ the restriction of this Riemannian metric to $\widetilde{L}^{s_j}_0$. Again, we may perturb, via a $C^2$-small perturbation, each $(f_{s_j},g)$ into a Morse-Smale pair $(h_{s_j},g')$ on $\widetilde{L}^{s_j}_0$. This induces a compactly supported Hamiltonian deformation of $L_1$ into a Lagrangian $L_1'\subset M$ such that 
    \begin{equation}
    \chi(L_0,L_1')=\bigcup_{j,s_j}\crit(h_{s_j}),
    \end{equation}
where this intersection is transverse. We define 
    \begin{equation}
    \crit(h_{s_j})\equiv\big\{y^{s_j}_w\big\}_{1\leq w\leq w_{s_j}}.
    \end{equation}
Also, this perturbation induces a perturbation of $J$ into an almost complex structure $J'\equiv\{J_t'\}$ on $M$ which agrees with $J$ outside of $\cup_{j,s_j}B_r(x_{s_j})$ and satisfies 
    \begin{equation}
    J_t'\vert_{B_r(x_{s_j})}=\widetilde{J}^{g'}_t\vert_{B_r(x_{s_j})}.
    \end{equation}
Finally, $(L_0,L_1')$ has a framed brane structure $\Lambda'$ determined in a natural way from $\Lambda$ and the compactly supported Hamiltonian deformation of $L_1$ into $L_1'$; this restricts to a framed brane structure $(\widetilde{\Lambda}^{s_j})'$ for $\big(\widetilde{L}^{s_j}_0,(\widetilde{L}^{s_j}_1)'\big)$ which is the $d_{s_j}$-fold stabilization of $\Lambda_\std$. Here, 
    \begin{equation}
    (\widetilde{L}^{s_j}_1)'\equiv L_1'\cap B_r(x_{s_j}).
    \end{equation}

By taking the perturbation of $(f_{s_j},g)$ into $(h_{s_j},g')$ sufficiently $C^2$-small, we see that the critical values of $h_{s_j}$,
    \begin{equation}
    \Big\{h\big(y^{s_j}_w\big)\Big\}_{1\leq w\leq w_{s_j}},
    \end{equation}
are arbitrarily close to the critical value 
    \begin{equation}
    f\big(x_{s_j}\big)=\kappa_j.
    \end{equation}
In particular, the intersection points are ordered by action in the following sense: 
    \begin{equation}
    \Big\{\big\{y^{1_1}_w\big\},\ldots,\big\{y^{\ell_1}_w\big\}\Big\}<\cdots<\Big\{\big\{y^{1_k}_w\big\},\ldots,\big\{y^{\ell_k}_w\big\}\Big\}.
    \end{equation}
Hence, $\bbF^{J',\Lambda',\kappa_j}_{L_0,L_1'}$ is the full framed flow subcategory of $\bbF^{J',\Lambda'}_{L_0,L_1'}$ spanned by
    \begin{equation}
    \Big\{\big\{y^{1_1}_w\big\},\ldots,\big\{y^{\ell_1}_w\big\}\Big\}<\cdots<\Big\{\big\{y^{1_j}_w\big\},\ldots,\big\{y^{\ell_j}_w\big\}\Big\}.
    \end{equation}
Consider the inclusion
    \begin{equation}
    \bbF^{J',\Lambda',\kappa_{j-1}}_{L_0,L_1'}\to\bbF^{J',\Lambda',\kappa_j}_{L_0,L_1'};
    \end{equation}
we see the cofiber is the framed flow category $\bbF^{J',\Lambda',\kappa_j,\kappa_{j-1}}_{L_0,L_1'}$ with objects 
    \begin{equation}
    \Big\{\big\{y^{1_j}_w\big\},\ldots,\big\{y^{\ell_j}_w\big\}\Big\}
    \end{equation}
and the obvious morphisms. By a straightforward extension of the argument in Subsection \ref{subsec:localfloer}, we may compute $\bbF^{J',\Lambda',\kappa_j,\kappa_{j-1}}_{L_0,L_1'}$ locally, i.e., $\bbF^{J',\Lambda',\kappa_j,\kappa_{j-1}}_{L_0,L_1'}$ splits as a disjoint union of framed flow categories:
    \begin{equation}
    \bbF^{J',\Lambda',\kappa_j,\kappa_{j-1}}_{L_0,L_1'}\cong\coprod_{s_j}\bbF^{\widetilde{J}^{g'},(\widetilde{\Lambda}^{s_j})'}_{\widetilde{L}^{s_j}_0,(\widetilde{L}^{s_j}_1)'}.
    \end{equation}
After passing to spectra, we have the cofiber sequence 
    \begin{equation}\label{eq:co}
    \frakF^{J',\Lambda',\kappa_{j-1}}_{L_0,L_1'}\to\frakF^{J',\Lambda',\kappa_j}_{L_0,L_1'}\to\bigvee_{s_j}\frakF^{\widetilde{J}^{g'},(\widetilde{\Lambda}^{s_j})'}_{\widetilde{L}^{s_j}_0,(\widetilde{L}^{s_j}_1)'}.
    \end{equation}
Now, by using the fact that $\widetilde{\Lambda}^{s_j}$ is the $d_{s_j}$-fold stabilization of $\Lambda_\std$ and referring to Propositions \ref{prop:floermorsecorrespondence} and \ref{prop:morsehomotopytype}, we have the following chain of equivalences:
    \begin{align}
    \frakF^{\widetilde{J}^g,\widetilde{\Lambda}^{s_j}}_{\widetilde{L}^{s_j}_0,\widetilde{L}^{s_j}_1}&\simeq\Sigma^{d_{s_j}}\frakF^{\widetilde{J}^{g'},\Lambda_\std}_{\widetilde{L}^{s_j}_0,(\widetilde{L}^{s_j}_1)'} \\
    &\simeq\Sigma^{d_{s_j}}\frakF^{H^{h_{s_j}},J^{g'},\Lambda_\std}_{\calO_{\widetilde{L}^{s_j}_0},\calO_{\widetilde{L}^{s_j}_0}} \\
    &\simeq\Sigma^{d_{s_j}}\frakM^{h_{s_j},g'}_{\bbR^n} \\
    &\simeq\Sigma^{d_{s_j}}\Sigma^\infty\calC_{h_{s_j}} \\
    &\simeq\Sigma^{d_{s_j}}\Sigma^\infty\calC_{f_{s_j}}.
    \end{align}
Finally, we may replace the first two terms in \eqref{eq:co} by definition:
    \begin{equation}
    \frakF^{J,\Lambda,\kappa_{j-1}}_{L_0,L_1}\to\frakF^{J,\Lambda,\kappa_j}_{L_0,L_1}\to\bigvee_{s_j}\Sigma^{d_{s_j}}\Sigma^\infty\calC_{f_{s_j}};
    \end{equation}
this proves the theorem.
\end{proof}

We recall the statement of Theorem \ref{thm:main2} for convenience.

\begin{thm}[Theorem \ref{thm:main2}]
Recall, $2n=\dim M$. Suppose Assumption \ref{assu:main} and that there exists integers 
    \begin{equation}
    \{s_j\}_{1\leq s\leq\ell, 1\leq j\leq k(s)}\subset\bbZ,\;\;s_j\geq(n-1)/2
    \end{equation}
together with 
    \begin{equation}
    \alpha_1,\ldots,\alpha_\ell\in HF^*(L_0,L_1;\bbZ/2)
    \end{equation}
satisfying
    \begin{align}
    \forall s,\;\;\sq^{s_{k(s)}}\cdots\sq^{s_1}(\alpha_s)&\neq0, \\
    \forall s'<s,\;\;\deg\alpha_s-\deg\Big(\sq^{s'_{k(s')}}\cdots\sq^{s'_1}(\alpha_{s'})\Big)&\geq n-1,
    \end{align}
then any compactly supported Hamiltonian deformation of $(L_0,L_1)$ which intersects in a discrete set must do so in at least $\ell+\sum_s k(s)$ points of pairwise distinct action.
\end{thm}

\begin{proof}[Proof of Theorem \ref{thm:main2}]
We prove the theorem in three steps.

First, suppose we are in the $\ell=1$ case, i.e., there exists integers 
    \begin{equation}
    \{1_j\}_{1\leq j\leq k(1)}\subset\bbZ,\;\;1_j\geq(n-1)/2
    \end{equation}
together with 
    \begin{equation}
    \alpha\in HF^*(L_0,L_1;\bbZ/2)
    \end{equation}
satisfying
    \begin{equation}
    \sq^{1_{k(1)}}\cdots\sq^{1_1}(\alpha)\neq0.
    \end{equation}
(In particular, $\alpha\neq0$.) Via a compactly supported Hamiltonian deformation, we may assume $L_0$ intersects $L_1$ in finitely many (possibly degenerate) points $\{x_{w_{j'}}\}_{1\leq j'\leq k',1\leq w_{j'}\leq q_{j'}}$, such that the points of fixed index $j'$ are of a fixed action $\kappa_{j'}\in\bbR$ which is strictly increasing. We wish to show $k'$ is at least $k(1)+1$. We consider the action of the Steenrod squares on a portion of \eqref{eq:les} at the highest action $\kappa_{k'}$ (we omit various auxiliary data for brevity):
    \begin{equation}
    \begin{tikzcd}[column sep=small]
    \bigoplus_{w_{k'}}\widetilde{H}^{*-d_{w_{k'}}}(\calC_{f_{w_{k'}}})\arrow[d,"\sq^{1_1}"]\arrow[r] & HF^*_{\kappa_{k'}}(L_0,L_1)\arrow[d,"\sq^{1_1}"]\arrow[r,"i^*_{k'-1}"] & HF^*_{\kappa_{k'-1}}(L_0,L_1)\arrow[d,"\sq^{1_1}"] \\
    \vdots\arrow[d,"\sq^{1_k}"] & \vdots\arrow[d,"\sq^{1_k}"] & \vdots\arrow[d,"\sq^{1_k}"] \\
    \bigoplus_{w_{k'}}\widetilde{H}^{*-d_{w_{k'}}+\sum_j 1_j}(\calC_{f_{w_{k'}}})\arrow[r] & HF^{*+\sum_j 1_j}_{\kappa_{k'}}(L_0,L_1)\arrow[r,"i^*_{k'-1}"] & HF^{*+\sum_j 1_j}_{\kappa_{k'-1}}(L_0,L_1)
    \end{tikzcd}
    .
    \end{equation}
Note, the previous diagram commutes precisely because we have shown a relation of spectra, i.e., Theorem \ref{thm:main1}. Now, since $\calC_{f_{w_{k'}}}$ is a co-H-space and $1_1\geq(n-1)/2$, it must be the case that
    \begin{equation}\label{eqn:proofaux}
    i^*_{k'-1}(\alpha)\neq0.
    \end{equation}
This is because, if the previous equation did not hold, by exactness we have
    \begin{equation}
    \alpha\in\operatorname{image}\Bigg(\varphi:\bigoplus_{w_{k'}}\widetilde{H}^{*-d_{w_{k'}}}(\calC_{f_{w_{k'}}})\to HF^*_{\kappa_{k'}}(L_0,L_1)\Bigg).
    \end{equation}
We denote the preimage of $\alpha$ by $\alpha'$. By commutativity
    \begin{equation}
    \varphi\Big(\sq^{1_1}(\alpha')\Big)=\sq^{1_1}(\alpha)\neq0,
    \end{equation}
where the second equality is by assumption; hence, 
    \begin{equation}
    \sq^{1_1}(\alpha')\neq0.
    \end{equation}
But this is a contradiction to Proposition \ref{prop:steenrod} since $\calC_{f_{w_{k'}}}$ is a co-H-space and $1_1\geq(n-1)/2$, thus \eqref{eqn:proofaux} holds. By induction,
    \begin{equation}
    \sq^{1_j}\cdots\sq^{1_1}(i^*_{k'-1}\alpha)\neq0,\;\;\forall j\leq k(1)-1; 
    \end{equation}
now, by induction, 
    \begin{equation}
    i^*_{k'-k(1)}\cdots i^*_{k'-1}(\alpha)\neq0,
    \end{equation}
and we see $k'$ is at least $k(1)+1$.\footnote{Recall, by convention, $\frakF^{J,\Lambda,\kappa_0}_{L_0,L_1}\equiv*$; hence, $i_0:\frakF^{J,\Lambda,\kappa_0}_{L_0,L_1}\to\frakF^{J,\Lambda,\kappa_1}_{L_0,L_1}$ is the inclusion of a point.}

Second, suppose we are in the $\ell=2$ and $k(2)=0$ case, i.e., there exists integers 
    \begin{equation}
    \{1_j\}_{1\leq j\leq k(1)}\subset\bbZ,\;\;1_j\geq(n-1)/2
    \end{equation}
together with 
    \begin{equation}
    \alpha,\beta\in HF^*(L_0,L_1;\bbZ/2)
    \end{equation}
satisfying
    \begin{align}
    \sq^{1_{k(1)}}\cdots\sq^{1_1}(\alpha)&\neq0, \\
    \beta&\neq0, \\
    \deg\beta-\deg\Big(\sq^{1_{k(1)}}\cdots\sq^{1_1}(\alpha)\Big)&\geq n-1. \label{eqn:assumption2}
    \end{align}
(In particular, $\alpha\neq0$.) Via a compactly supported Hamiltonian deformation, we may assume $L_0$ intersects $L_1$ in finitely many (possibly degenerate) points $\{x_{w_{j'}}\}_{1\leq j'\leq k',1\leq w_{j'}\leq q_{j'}}$, such that the points of fixed index $j'$ are of a fixed action $\kappa_{j'}\in\bbR$ which is strictly increasing. We wish to show $k'$ is at least $k(1)+2$. Again, if we consider the action of the Steenrod squares on a portion of \eqref{eq:les} at the highest action $\kappa_{k'}$, we see
    \begin{equation}
    \sq^{1_j}\cdots\sq^{1_1}(i^*_{k'-1}\alpha)\neq0,\;\;\forall j\leq k(1)-1.
    \end{equation}
Now, since $\calC_{f_{w_{k'}}}$ is a co-H-space and we have \eqref{eqn:assumption2}, it cannot be the case that both
    \begin{equation}
    i^*_{k'-1}\beta,i^*_{k'-1}\Big(\sq^{1_{k(1)}}\cdots\sq^{1_1}(\alpha)\Big)=0.
    \end{equation}
This is because, if the previous statement were not true, then both
    \begin{equation}
    \beta,\sq^{1_{k(1)}}\cdots\sq^{1_1}(\alpha)\in\operatorname{image}\Bigg(\varphi:\bigoplus_{w_{k'}}\widetilde{H}^{*-d_{w_{k'}}}(\calC_{f_{w_{k'}}})\to HF^*_{\kappa_{k'}}(L_0,L_1)\Bigg).
    \end{equation}
But this is a contradiction to Proposition \ref{prop:support} since $\calC_{f_{w_{k'}}}$ is a co-H-space and we have \eqref{eqn:assumption2}. If $i^*_{k'-1}\beta=0$, then we are back in the previous case, and we see $k'$ is at least $k(s)+2$. If $i^*_{k'-1}\beta\neq0$, then, by induction, it follows that
    \begin{equation}
    i^*_{k'-k(1)-1}\cdots i^*_{k'-1}\beta\neq0,
    \end{equation}
and we see $k'$ is at least $k(1)+2$. 

The third case is the general case, and this follows by induction from the previous two cases; this proves the theorem.
\end{proof}

By using Example \ref{example:steenrod1}, we obtain the following result.

\begin{cor}\label{cor:dim14}
Let $\dim M=14$. Suppose Assumption \ref{assu:main} and that there exists integers 
    \begin{equation}
    \{s_j\}_{1\leq s\leq\ell, 1\leq j\leq k(s)}\subset\bbZ,\;\;s_j\geq2
    \end{equation}
together with 
    \begin{equation}
    \alpha_1,\ldots,\alpha_\ell\in HF^*(L_0,L_1;\bbZ/2)
    \end{equation}
satisfying
    \begin{align}
    \forall s,\;\;\sq^{s_{k(s)}}\cdots\sq^{s_1}(\alpha_s)&\neq0, \\
    \forall s'<s,\;\;\deg\alpha_s-\deg\Big(\sq^{s'_{k(s')}}\cdots\sq^{s'_1}(\alpha_{s'})\Big)&\geq 6,
    \end{align}
then any compactly supported Hamiltonian deformation of $(L_0,L_1)$ which intersects in a discrete set must do so in at least $\ell+\sum_s k(s)$ points of pairwise distinct action.
\end{cor}

\subsection{An application}\label{subsec:application}
Consider the embedding $\iota:\bbC P^2\hookrightarrow S^7$. Recall, the plumbing of two disjoint copies of the unit co-disk bundle $D^*S^7$ along $\bbC P^2$ is constructed as follows. Let 
    \begin{equation}
    N_{\bbC P^2}S^7\to\bbC P^2
    \end{equation}
be the normal bundle of $\bbC P^2$ in $S^7$. By Weinstein's theorem, there exists a symplectomorphism 
    \begin{equation}
    \chi:U\subset T^*S^7\xrightarrow{\sim} V\subset\tau^*\big(N_{\bbC P^2}S^7\otimes_\bbR\underline{\bbC}\big),
    \end{equation}
where $U$ is a neighborhood of $\bbC P^2$ (thought of as inside $\calO_{S^7}$), $V$ is a neighborhood of the 0-section, and $\tau: T^*\bbC P^2\to\bbC P^2$ is the projection. By construction, there are two natural totally real subbundles
    \begin{equation}
    N_{\bbC P^2}S^7,\sqrt{-1}N_{\bbC P^2}S^7\subset N_{\bbC P^2}S^7\otimes_\bbR\underline{\bbC};
    \end{equation}
note, $\chi$ may be chosen so that 
    \begin{equation}
    \chi\big(U\cap S^7\big)=N_{\bbC P^2}S^7.
    \end{equation}
We may now glue one copy $D^*S^7_0$ to another copy $D^*S^7_1$ along $\chi(U)$ after complex multiplication by $\sqrt{-1}$. This process results in a Liouville 14-manifold $(M,\omega=d\theta)$, where $\theta$ agrees with the standard Liouville 1-form on $D^*S^7_i$ away from a neighborhood of $\bbC P^2$ and has a natural description near $\bbC P^2$. Moreover, there is a natural pair of Lagrangians $(L_0,L_1)$ in $M$, each diffeomorphic to $S^7$, which intersect cleanly along $\bbC P^2$.

Now, recall $\calO_{S^7_i}\subset T^*S^7_i$ has a natural framed brane structure $\Lambda^i_\std$ induced by the identification
    \begin{equation}
    T(T^*S^7_i)\cong\tau_i^*TS^7_i\otimes_\bbR\underline{\bbC}\equiv\Lambda^i_\std\otimes_\bbR\underline{\bbC},\;\;\tau_i:T^*S^7_i\to S^7_i
    \end{equation}
and the constant homotopy between $\Lambda^i_\std\vert_{\calO_{S^7_i}}$ and $T\calO_{S^7_i}$. 

\begin{lem}
There exists a framed brane structure $\Lambda$ for $(L_0,L_1)$ such that
    \begin{equation}
    \frakF^{J,\Lambda}_{L_0,L_1}\simeq\Sigma^\infty_+\bbC P^2.
    \end{equation}
\end{lem}

\begin{proof}
Since $(L_0,L_1)$ intersects cleanly in $\bbC P^2$, we may pass to the following local model for a moment. By \cite[Proposition 3.4.1]{Poz94}, we may assume we are in the following situation: $L\to\bbC P^2$ is a rank 3 real vector bundle and we are considering the Lagrangian pair 
    \begin{equation}
    (\calO_L,\widetilde{L}),\;\;\widetilde{L}\equiv\big\{\xi\in T^*L\vert_{\bbC P^2}:\xi\vert_{\bbC P^2}=0\big\}
    \end{equation}
in the symplectic 14-manifold $(T^*L,\omega_\std)$. Note, 
    \begin{equation}
    \calO_L\cap\widetilde{L}=\bbC P^2,
    \end{equation}
and this intersection is clean. Choose a Riemannian metric $g$ on $L$ and denote by 
    \begin{equation}
    N_{\bbC P^2}L\to\bbC P^2
    \end{equation}
the normal bundle of $\bbC P^2$ in $L$. We may identify $\widetilde{L}$ with the co-normal bundle 
    \begin{equation}
    N^*_{\bbC P^2}L\to L.
    \end{equation}
Choose a tubular neighborhood $U\subset T^*L$ of $\bbC P^2$ such that (1) $U\cap\calO_L$ is a tubular neighborhood of $\bbC P^2$ in $\calO_L$ and (2) $U\cap\widetilde{L}=D^*_{\bbC P^2}L$, where $D^*_{\bbC P^2}L\subset N^*_{\bbC P^2}L$ is the unit co-disk bundle. We have an identification 
    \begin{equation}\label{eqn:identificationconormal}
    U\cap\calO_L\xrightarrow{\sim}D_{\bbC P^2}L\cong D^*_{\bbC P^2}L,
    \end{equation}
where $D_{\bbC P^2}L\subset N_{\bbC P^2}L$ is the unit disk bundle,
via the inverse of the Riemannian exponential map. The upshot is that we have a natural choice of smooth function 
    \begin{equation}
    f\in C^\infty(U\cap\calO_L)
    \end{equation}
such that $f$ has a global minimum along $\bbC P^2$ and no other critical points; in fact, $f$ is simply
    \begin{equation}
    f(p)=d\big(p,\bbC P^2\big),
    \end{equation}
where $d\big(p,\bbC P^2\big)$ is the distance from $p$ to $\bbC P^2$ using the metric induced by $g$. By modifying $g$ appropriately, we may assume $f$ is sufficiently $C^2$-small. It follows that $U\cap\widetilde{L}$ is identified with 
    \begin{equation}
    \big\{(p,df_p)\big\}\subset T^*L.
    \end{equation}
We would like to point out that $U\cap\calO_L$ and $U\cap\widetilde{L}$ are Hamiltonian isotopic via the locally defined Hamiltonian $H^f\in C^\infty(U)$.

We now return to our global $(L_0,L_1)$ in $M$. Using the Liouville flow of $M$, we see $M$ retracts onto its Lagrangian skeleton, and the latter can be identified as the following homotopy pushout:
    \begin{equation}
    \begin{tikzcd}
    \bbC P^2\arrow[r]\arrow[d] & L_0\arrow[d] \\
    L_1\arrow[r] & L_0\cup_{\bbC P^2}L_1\simeq M
    \end{tikzcd}
    ,
    \end{equation}
where $\bbC P^2\to L_i$ is the inclusion. Consider a classifying map for $TM$: 
    \begin{equation}
    M\to BU(7).
    \end{equation}
We have that the restriction to $L_i$, 
    \begin{equation}
    L_i\simeq T^*L_i\to M\to BU(7),
    \end{equation}
is a classifying map for $T(T^*L_i)$. Using $\Lambda^i_\std$, we have a factorization
    \begin{equation}
    L_i\simeq T^*L_i\xrightarrow{\Lambda^i_\std} BO(7)\to BU(7).
    \end{equation}
We see that the restriction to $\bbC P^2$ of this factorization, for $i=0$, is homotopic to the restriction to $\bbC P^2$ of this factorization, for $i=1$, via using the locally defined Hamiltonian isotopy in the local model just discussed. The upshot is that, since $M$ is a homotopy pushout, we induce a factorization
    \begin{equation}
    M\xrightarrow{\Lambda} BO(7)\to BU(7).
    \end{equation}
Moreover, the constant homotopy between $\Lambda^i_\std\vert_{\calO_{S^7_i}}$ and $T\calO_{S^7_i}$ naturally induces a homotopy between $\Lambda\vert_{L_i}$ and $TL_i$. Finally, a straightforward extension of the argument in Subsection \ref{subsec:localfloer} shows 
    \begin{equation}
    \frakF^{J,\Lambda}_{L_0,L_1}\simeq\Sigma^\infty\calC_f\simeq\Sigma^\infty_+\bbC P^2,
    \end{equation}
where: $f\in C^\infty(W)$ is a smooth function defined in a tubular neighborhood $W\subset L_0$ of $\bbC P^2$ with a global minimum along $\bbC P^2$ and no other critical points, and the first equivalence uses Proposition \ref{prop:morsehomotopytype}.
\end{proof}

\begin{rem}
In the construction of the function $f$ in the previous lemma, we could instead use 
    \begin{equation}
    -f=-d\big(p,\bbC P^2\big).
    \end{equation}
In particular, we used $f$ in our construction of $\Lambda$. The upshot is that $-f$ is a smooth function defined in a tubular neighborhood of $\bbC P^2$ in $L_0$ with a global maximum along $\bbC P^2$ and no other critical points; this will induce a different choice of framed brane structure $\Lambda'$ for $(L_0,L_1)$ such that 
    \begin{equation}
    \frakF^{J,\Lambda'}_{L_0,L_1}\simeq\calD\big(\Sigma^\infty_+\bbC P^2\big).
    \end{equation}
Of course, our application does not care if we use $f$ or $-f$ by our computations in Subsection \ref{subsec:example}.
\end{rem}

We now go to the general construction. Consider $S^7$ with $n$-many disjoint embedded copies of $\bbC P^2$: 
    \begin{equation}
    C\equiv\bbC P^2_1\amalg\cdots\amalg\bbC P^2_n\subset S^7.
    \end{equation}
We now plumb $T^*S^7_0$ to $T^*S^7_1$ along $C$; the result is a Liouville manifold $(M,\omega)$ with a pair of Lagrangians $(L_0,L_1)$, each diffeomorphic to $S^7$, which intersect cleanly along $C$.

\begin{lem}
There exists a framed brane structure $\Lambda$ for $(L_0,L_1)$ such that 
    \begin{equation}
    \frakF^{J,\Lambda}_{L_0,L_1}\simeq\Sigma^\infty_+\big(\bbC P^2_1\vee\cdots\vee\bbC P^2_n\big).
    \end{equation}
\end{lem}

\begin{proof}
Again, we may construct a smooth function $f\in C^\infty(W)$ defined in a tubular neighborhood $W\subset L_0$ of $C$, which is a disjoint union of tubular neighborhoods of each $\bbC P^2_j$ in $L_0$, with a global minimum along $C$ and no other critical points. It follows $\frakF^{J,\Lambda}_{L_0,L_1}$ is a filtered spectrum whose associated graded is 
    \begin{equation}
    \frakF^{J,\Lambda,\kappa_j,\kappa_{j-1}}_{L_0,L_1}\simeq\Sigma^\infty_+\bbC P^2_j;
    \end{equation}
hence, the key is to show there exists no Floer trajectories connecting intersection points associated to different pieces of the associated graded. But this is straightforward, since, for purely topological reasons, there is no continuous map 
    \begin{equation}
    u:[-1,1]\times[0,1]\to M
    \end{equation}
such that 
    \begin{equation}
    u(-1,t)\in\bbC P^2_j,\;\;u(1,t)\in\bbC P^2_{j'},\;\;u(s,i)\in L_i
    \end{equation}
for $j\neq j'$. I.e., we have a non-trivial element $\gamma_{j,j'}$, for $j\neq j'$, of $\pi_1(M)$ given by starting on $\bbC P^2_j$, following $L_1$ to $\bbC P^2_{j'}$, and then returning along $L_0$.
\end{proof}

In order to apply Corollary \ref{cor:dim14}, we will have to appropriately shift the degrees of the wedge summands, and we will do this by appropriately altering $\Lambda$. 

\begin{lem}
There exists a framed brane structure $\Lambda$ for $(L_0,L_1)$ such that 
    \begin{equation}
    \frakF^{J,\Lambda}_{L_0,L_1}\simeq\Sigma^\infty_+\big(\bbC P^2_1\vee\cdots\vee\Sigma^{8(n-1)}\bbC P^2_n\big).
    \end{equation}
\end{lem}

\begin{proof}
Let $\gamma_j$, for $j\neq1$, be a non-nullhomotopic loop given by starting on $\bbC P^2_1$, following $L_1$ to $\bbC P^2_j$, and then returning along $L_0$. We have that $\gamma_j$ induces a map
    \begin{equation}
    \gamma^j:M\to S^1.
    \end{equation}
Since $H^1(L_i;\bbZ/2)=0$, we see $\gamma^j\vert_{L_i}$ is null-homotopic; we choose a null-homotopy. Consider the framed brane structure for $(L_0,L_1)$ where
    \begin{equation}
    \widetilde{\Lambda}\equiv\Lambda\oplus e^{4J\gamma^1}\underline{\bbR}\oplus\cdots\oplus e^{4(n-1)J\gamma^n}\underline{\bbR}\subset TM\oplus\underline{\bbC}^n
    \end{equation}
and the homotopy between $\widetilde{\Lambda}\vert_{L_i}$ and $TL_i\oplus\underline{\bbR}^n$ is given by the homotopy between $\Lambda\vert_{L_i}$ and $TL_i$ stabilized by the choices of null-homotopies of $\big\{\gamma^j\vert_{L_i}\big\}$. By construction, this increases the relative grading between Lagrangian Floer cohomology classes coming from $\bbC P^2_1$ and Lagrangian Floer cohomology classes coming from $\bbC P^2_j$ by $8(j-1)$; the lemma follows.
\end{proof}

By solely examining the degrees of Lagrangian Floer cohomology classes, we see that any compactly supported Hamiltonian deformation of $(L_0,L_1)$ which intersects in a discrete set must do so in at least $n$ points of pairwise distinct action. However, since there are $n$ distinct Lagrangian Floer cohomology classes with non-vanishing $\sq^2$ satisfying the appropriate degree condition, Corollary \ref{cor:dim14} shows any compactly supported Hamiltonian deformation of $(L_0,L_1)$ which intersects in a discrete set must do so in at least $2n$ points of pairwise distinct action. Finally, we cannot use the results of Hofer, Floer, or Hirschi-Porcelli resp. the quantum cap-length estimate since $(L_0,L_1)$ is not Hamiltonian isotopic resp. the quantum cap products of $(L_0,L_1)$ vanish for degree reasons.

\begin{rem}\label{rem:filling}
The determined reader may try to use the fact that the path-space of $(L_0,L_1)$ has $n$ components to compute the $2n$ lower bound; we may circumvent this as follows. Let $\gamma_{j,j+1}$ be a non-trivial element of $\pi_1(M)$ as before. We may assume 
    \begin{equation}
    \gamma_{j,j+1}:S^1\to M
    \end{equation}
is a smooth embedding whose image actually lies in the contact boundary $\partial M$. Moreover, since $\gamma_{j,j+1}$ is a map from $S^1$, its image is isotropic. A straightforward computation, using 
    \begin{equation}
    T\partial M\vert_{\gamma_{j,j+1}}\cong TS^1\oplus N_{\gamma_{j,j+1}}\partial M,
    \end{equation}
shows
    \begin{equation}
    w_1\big(N_{\gamma_{j,j+1}}\partial M\big)=0;
    \end{equation}
in particular, we may perform a subcritical Weinstein attachment along $\gamma_{j,j+1}$ to reconnect the path-space. Note, the Floer homotopy type will be invariant under subcritical Weinstein handle attachment so long as we show the framed brane structure for $(L_0,L_1)$ extends across the attached subcritical handle; hence, we will show such an extension exists. Let $\widetilde{M}$ be $M$ after subcritical Weinstein handle attachment: 
    \begin{equation}
    \widetilde{M}=M\cup_{S^1\times D^{12}}D^2\times D^{12}.
    \end{equation}
By tracing through the construction of $\Lambda$, we see that it is actually orientable, i.e.,
    \begin{equation}
    \Lambda:M\to BSO(7+n).
    \end{equation}
(This fact essentially boils down to the fact that $S^7$ is orientable.) Note, $\widetilde{M}$ is the homotopy cofiber of $\gamma_{j,j+1}$. First, we have a homotopy commutative diagram
    \begin{equation}
    \begin{tikzcd}
    S^1\arrow[r,"\gamma_{j,j+1}"]\arrow[d] & M\arrow[d]\arrow[ddr,"TM\oplus\underline{\bbC}^n", bend left=60] & \\
    *\arrow[r]\arrow[drr,bend right=60] & \widetilde{M}\arrow[dr,"T\widetilde{M}\oplus\underline{\bbC}^n",swap] & \\
    & & BU(7+n)
    \end{tikzcd}
    .
    \end{equation}
Second, since $\pi_1\big(BSO(7+n)\big)=0$, we have a homotopy commutative diagram
    \begin{equation}
    \begin{tikzcd}
    S^1\arrow[r,"\gamma_{j,j+1}"]\arrow[d] & M\arrow[d]\arrow[ddr,"\Lambda", bend left=60] & \\
    *\arrow[r]\arrow[drr,bend right=60] & \widetilde{M}\arrow[dr,"\widetilde{\Lambda}",swap] & \\
    & & BSO(7+n)
    \end{tikzcd}
    ,
    \end{equation}
where $\widetilde{\Lambda}$ is defined by the universal property of the homotopy cofiber. Third, since $\widetilde{\Lambda}\vert_M=\Lambda$ and 
    \begin{equation}
    \pi_1\big(BSO(7+n)\big)\cong0\to\pi_1\big(BU(7+n)\big)\cong0
    \end{equation}
is the trivial map, we have a homotopy commutative diagram
    \begin{equation}
    \begin{tikzcd}
    S^1\arrow[r,"\gamma_{j,j+1}"]\arrow[d] & M\arrow[d]\arrow[dddrr,"TM\oplus\underline{\bbC}^n", bend left=60] & & \\
    *\arrow[r]\arrow[ddrrr,bend right=60] & \widetilde{M}\arrow[dr,"\widetilde{\Lambda}",swap] & & \\
    & & BSO(7+n)\arrow[dr,"(\cdot)\otimes_\bbR\underline{\bbC}",swap] & \\
    & & & BU(7+n)
    \end{tikzcd}
    .
    \end{equation}
Hence, by the universal property of the homotopy cofiber, 
    \begin{equation}
    T\widetilde{M}\oplus\underline{\bbC}^n\cong\widetilde{\Lambda}\otimes_\bbR\underline{\bbC}.
    \end{equation}
Finally, since $(L_0,L_1)$ lies inside of $M\subset\widetilde{M}$, we may use the same homotopies as in $M$ to see $\widetilde{\Lambda}$ is a framed brane structure for $(L_0,L_1)$.
\end{rem}

\section{Quantum cap-length}\label{section:caplength}

We recall the statement of Theorem \ref{thm:caplength} for convenience. 

\begin{thm}[Theorem \ref{thm:caplength}]
Suppose Assumption \ref{assu:caplength} and that there exists 
    \begin{equation}
    \alpha_1,\ldots,\alpha_k\in H^*(L_i;\bbZ/2),\;\;\deg\alpha_j\geq1
    \end{equation}
and $\beta\in HF(L_0,L_1;\bbZ/2)$ such that
    \begin{equation}
    \beta*\alpha_1*\cdots*\alpha_k\neq0, 
    \end{equation}
then any compactly supported Hamiltonian deformation of $(L_0,L_1)$ which intersects in a discrete set must do so in at least $k+1$ points of pairwise distinct action.
\end{thm}

\begin{proof}
For concreteness, we will assume 
    \begin{equation}
    \alpha_1,\ldots,\alpha_k\in H^*(L_0;\bbZ/2),\;\;\deg\alpha_j\geq1.
    \end{equation}
Via a compactly supported Hamiltonian deformation, we may assume $L_0$ intersects $L_1$ in finitely many (possibly degenerate) points $\{x_{w_{j'}}\}_{1\leq j'\leq k',1\leq w_{j'}\leq q_{j'}}$, such that the points of fixed index $j'$ are of a fixed action $\kappa_{j'}\in\bbR$ which is strictly increasing. We wish to show $k'$ is at least $k+1$. (Note, since $\beta\neq0$, we see $k'$ is at least 1.) 

As in the case of a single intersection, we may choose, for each $x_{s_j}$, a sufficiently small ball $B_r(x_{s_j})$. We define
    \begin{equation}
    \widetilde{L}^{s_j}_i\equiv L_i\cap B_r(x_{s_j}).
    \end{equation}
Also, we may choose a sufficiently $C^2$-small smooth function $f\in C^\infty(\widetilde{L}^{s_j}_0)$ together with an identification of $\widetilde{L}^{s_j}_1$ and
    \begin{equation}
    \Big\{\big(p,d(f_{s_j})_p\big)\Big\}\subset T^*\widetilde{L}^{s_j}_0.
    \end{equation}
In particular, $f_{s_j}$ only has a single (possibly degenerate) critical point at $x_{s_j}$ and $S_{f_{s_j}}=\{x_{s_j}\}$. Moreover, by choosing $r$ sufficiently small, we may identify $T^*\widetilde{L}^{s_j}_0$ with $T^*\bbR^n$ and $\widetilde{L}^{s_j}_0$ with $\calO_{\bbR^n}$; hence, $f_{s_j}$ may be thought of as a smooth function on $\bbR^n$ with only a single (possibly degenerate) critical point at the origin.

Let $g$ be a Riemannian metric on $L_0$; we abuse notation and also denote by $g$ the restriction of this Riemannian metric to $\widetilde{L}^{s_j}_0$. Again, we may perturb, via a $C^2$-small perturbation, each $(f_{s_j},g)$ into a Morse-Smale pair $(h_{s_j},g')$ on $\widetilde{L}^{s_j}_0$. This induces a compactly supported Hamiltonian deformation of $L_1$ into a Lagrangian $L_1'\subset M$ such that 
    \begin{equation}
    \chi(L_0,L_1')=\bigcup_{j,s_j}\crit(h_{s_j}),
    \end{equation}
where this intersection is transverse. We define 
    \begin{equation}
    \crit(h_{s_j})\equiv\big\{y^{s_j}_w\big\}_{1\leq w\leq w_{s_j}}.
    \end{equation}
Also, this perturbation induces a perturbation of $J$ into an almost complex structure $J'\equiv\{J_t'\}$ on $M$ which agrees with $J$ outside of $\cup_{j,s_j}B_r(x_{s_j})$ and satisfies 
    \begin{equation}
    J_t'\vert_{B_r(x_{s_j})}=\widetilde{J}^{g'}_t\vert_{B_r(x_{s_j})}.
    \end{equation}

By taking the perturbation of $(f_{s_j},g)$ into $(h_{s_j},g')$ sufficiently $C^2$-small, we see that the critical values of $h_{s_j}$,
    \begin{equation}
    \Big\{h\big(y^{s_j}_w\big)\Big\}_{1\leq w\leq w_{s_j}},
    \end{equation}
are arbitrarily close to the critical value 
    \begin{equation}
    f\big(x_{s_j}\big)=\kappa_j.
    \end{equation}
In particular, the intersection points are ordered by action in the following sense: 
    \begin{equation}
    \Big\{\big\{y^{1_1}_w\big\},\ldots,\big\{y^{\ell_1}_w\big\}\Big\}<\cdots<\Big\{\big\{y^{1_k}_w\big\},\ldots,\big\{y^{\ell_k}_w\big\}\Big\}.
    \end{equation}
By Lemma \ref{lem:localfloer}, the Lagrangian Floer theory of $(L_0,L_1)$ near $x_{s_j}$ is local. In particular, we have a long exact sequence (we omit various auxiliary choices for brevity)
    \begin{equation}\label{eqn:les8}
    HF_{\kappa_{j'-1}}(L_0,L_1')\xrightarrow{i^{j'-1}_*}HF_{\kappa_{j'}}(L_0,L_1')\to\bigoplus_{s_{j'}}HM_*(\widetilde{L}^{s_{j'}}_0;h_{s_{j'}},g').
    \end{equation}

Let $(F,G)$ be a Morse-Smale pair on $L_0$. Recall, the quantum cap product is induced by a chain-level morphism
    \begin{equation}
    CF(L_0,L_1';J';\bbZ/2)\otimes_{\bbZ/2}CM^*(L_0;F,G;\bbZ/2)\to CF(L_0,L_1';J';\bbZ/2),
    \end{equation}
where $CM^*(L_0;F,G;\bbZ/2)$ is the Morse cochain complex, defined by 
    \begin{equation}
    \big(y_-,a\big)\mapsto\sum_{\mu(y_-,y_+)-I(a)=0}\big\lvert\bbF^{J'}_{L_0,L_1'}(y_-,y_+,a)\big\rvert_{\bbZ/2}\cdot y_+,
    \end{equation}
where $\mu(\cdot,\cdot)$ is the relative Maslov index and 
    \begin{equation}
    \bbF^{J'}_{L_0,L_1'}(y_-,y_+,a)
    \end{equation}
is the Floer-theoretic compactification of the moduli space consisting of Floer trajectories $u\in\widehat{\calF}^{J'}_{L_0,L_1'}(y_-,y_+)$ such that $u(0,0)$ is contained in $W^s(a;F)$. Consider the map 
    \begin{equation}
    HM_*(\widetilde{L}^{s_{j'}}_0;h_{s_{j'}},g')\otimes_{\bbZ/2}HM^*(L_0;F,G)\to HM_*(\widetilde{L}^{s_{j'}}_0;h_{s_{j'}},g')
    \end{equation}
induced by the chain-level map 
    \begin{equation}
    \big(y_-,a\big)\mapsto\sum_{I(y_-)-I(y_+)-I(a)=0}\big\lvert\bbM^{h_{s_{j'}},g'}_{\widetilde{L}_0^{s_{j'}}}(y_-,y_+,a)\big\vert_{\bbZ/2}\cdot y_+,
    \end{equation}
where 
    \begin{equation}
    \bbM^{h_{s_{j'}},g'}_{\widetilde{L}_0^{s_{j'}}}(y_-,y_+,a)
    \end{equation}
is the Morse-theoretic compactification of the moduli space consisting of gradient flows $\gamma\in\widehat{\calM}^{h_{s_{j'}},g'}_{\widetilde{L}_0^{s_{j'}}}(y_-,y_+)$ such that $\gamma(0)$ is contained in $W^s(a;F)$; this is the Morse cap product. By construction, \eqref{eqn:les8} is compatible with the quantum cap product and the Morse cap product. We may choose $F$ such that its restriction to $\widetilde{L}^{s_{j'}}_0\cong\bbR^n$ is a function on $\bbR^n$ whose only critical point is a single non-degenerate index 0 critical point at the origin (in particular, the (negative) gradient flow of the restriction of $F$ points radially inward toward the origin and the restriction of $(F,G)$ is Morse-Smale). Thus, we have that the Morse cap product actually factors via restriction:
    \begin{equation}\label{eqn:factors}
    \begin{tikzcd}
    HM_*(\widetilde{L}^{s_{j'}}_0;h_{s_{j'}},g')\otimes_{\bbZ/2}HM^*(L_0;F,G)\arrow[d,"\mathrm{Id}\otimes\mathrm{res}"]\arrow[dr] & \\
    HM_*(\widetilde{L}^{s_{j'}}_0;h_{s_{j'}},g')\otimes_{\bbZ/2}HM^*(\widetilde{L}^{s_{j'}}_0;F,G)\arrow[r] & HM_*(\widetilde{L}^{s_{j'}}_0;h_{s_{j'}},g')
    \end{tikzcd}
    ,
    \end{equation}
where the bottom arrow is, again, a Morse cap product. Finally, we observe that, by construction,
    \begin{equation}
    HM^*(\widetilde{L}^{s_{j'}}_0;F,G;\bbZ/2)\cong\widetilde{H}^*(\calC_F;\bbZ/2)\cong\widetilde{H}^*(S^0;\bbZ/2).
    \end{equation}

Proving the theorem is now straightforward. Since we have a non-zero quantum cap product $\beta*\alpha_1$, $\deg\alpha_1\geq1$, it must be the case that 
    \begin{equation}
    \beta*\alpha_1\in\operatorname{image}i^{k'-1}_*,
    \end{equation}
as follows. Push $\beta$ to $\oplus_{s_{k'}}HM_*(\widetilde{L}^{s_{k'}}_0;h_{s_{k'}},g')$ using 
    \begin{equation}
    \varphi:HF_{\kappa_{k'}}(L_0,L_1')\to\bigoplus_{s_{j'}}HM_*(\widetilde{L}^{s_{k'}}_0;h_{s_{j'}},g').
    \end{equation}
Since each $HM^*(\widetilde{L}^{s_{k'}}_0;F,G;\bbZ/2)$ is concentrated in degree 0 and $\deg\alpha_1\geq1$, it follows that 
    \begin{equation}
    \mathrm{res}(\alpha_1)\in HM^*(\widetilde{L}^{s_{k'}}_0;F,G;\bbZ/2)
    \end{equation}
vanishes. Finally, since \eqref{eqn:les8} is compatible with the quantum cap product and the Morse cap product, and we have seen the Morse cap product factors via \eqref{eqn:factors}, it must be the case that 
    \begin{equation}
    \beta*\alpha_1\in\ker\varphi=\operatorname{image}i^{k'-1}_*,
    \end{equation}
as desired. Hence, $k'$ is at least 2; the theorem follows by induction.
\end{proof}

\bibliography{References}{}
\bibliographystyle{alpha.bst}
\end{document}